\documentclass[12pt,english]{amsart}
\usepackage[T1]{fontenc}
\usepackage[latin9]{inputenc}
\usepackage{geometry}
\usepackage{babel}
\usepackage{mathrsfs}
\usepackage{enumitem}
\usepackage{amstext}
\usepackage{amssymb}
\usepackage{stmaryrd}
\usepackage[unicode=true,pdfusetitle,
 bookmarks=true,bookmarksnumbered=false,bookmarksopen=false,
 breaklinks=false,pdfborder={0 0 1},backref=false,colorlinks=false]
 {hyperref}

\makeatletter
\usepackage{amsthm}

\usepackage{rotating}
\usepackage{enumitem}

\usepackage{frcursive}
\usepackage{amsfonts}\usepackage{amscd}\usepackage[all]{xy}\usepackage{color}\usepackage{tikz}\usepackage{amsthm}
\font\calli=rsfs10 at 12 pt
\numberwithin{equation}{section}

\newtheorem{thm}{Theorem}[section]
\newtheorem*{thm*}{Theorem}
\newtheorem{cor}[thm]{Corollary}
\newtheorem*{cor*}{Corollary}
\newtheorem{lem}[thm]{Lemma}
\newtheorem*{lem*}{Lemma}
\newtheorem{prop}[thm]{Proposition}
\newtheorem*{prop*}{Proposition}

\newtheorem*{conjecture*}{Conjecture}

\newtheorem*{fact*}{Conjecture}

\newtheorem*{criterion*}{Criterion}

\newtheorem*{algorithm*}{Algorithm}

\newtheorem*{ax*}{Axiom}

\newtheorem*{assumption*}{Assumption}

\newtheorem*{question*}{Question}

\theoremstyle{remark}

\newtheorem{rem}[thm]{Remark}
\newtheorem*{rem*}{Remark}
\newtheorem{rems}[thm]{Remarks}
\newtheorem*{rems*}{Remarks}

\newtheorem*{claim*}{Claim}

\newtheorem*{exercise*}{Exercise}

\newtheorem*{note*}{Note}
\newtheorem{notation}[thm]{Notation}
\newtheorem*{notation*}{Notation}

\newtheorem*{summary*}{Summary}

\newtheorem*{acknowledgement*}{Acknowledgement}

\newtheorem*{conclusion*}{Conclusion}

\theoremstyle{definition}

\newtheorem{defn}[thm]{Definition}
\newtheorem*{defn*}{Definition}

\newtheorem*{example*}{Example}
\newtheorem{examples}[thm]{Examples}
\newtheorem*{examples*}{Examples}

\newtheorem*{problem*}{Problem}

\newtheorem*{xca*}{Exercise}

\newtheorem*{xcas*}{Exercises}

\newtheorem*{condition*}{Condition}

\author[Cluckers]
{Raf Cluckers}
\address{Univ.~Lille, CNRS, UMR 8524 - Laboratoire Paul Painlev\'e, F-59000 Lille, France, and,
KU Leuven, Department of Mathematics, B-3001 Leuven, Belgium}
\email{Raf.Cluckers@univ-lille.fr}
\urladdr{http://rcluckers.perso.math.cnrs.fr/}

\author[Comte]{Georges Comte}
\address{Universit\'e Savoie Mont Blanc, LAMA,
CNRS UMR 5127,
F-73000 Chamb\'ery, France}
\email{georges.comte@univ-smb.fr}
\urladdr{http://gcomte.perso.math.cnrs.fr/}

\author[Servi]{Tamara Servi}
\address{Institut de Math\'ematiques de Jussieu -- Paris Rive Gauche \\
	Universit\'{e} Paris Cit\'{e} and Sorbonne Universit\'{e}, CNRS, IMJ-PRG, F-75013 Paris, France}
\email{tamara.servi@imj-prg.fr}
\urladdr{http://www.logique.jussieu.fr/~servi/index.html}

\AtBeginDocument{
  
}

\makeatother

\begin{document}
\title[Fourier and Mellin transforms of power-constructible functions]{Parametric Fourier and Mellin transforms of power-constructible functions}
\begin{abstract}    


We enrich the class of power-constructible functions, introduced in \cite{cmrs:mellin_constructible}, to  
a class $\mathcal{C}^{\mathcal{M,F}}$ of algebras of functions which contains all complex powers of subanalytic functions, their parametric Mellin and Fourier transforms,  
and which is stable under parametric integration.
By describing a 
set of generators of a special prepared form we deduce information 
on the asymptotics and on the loci of integrability of the functions of $\mathcal{C}^{\mathcal{M,F}}$. We furthermore identify a subclass $\mathcal{C}^{\mathbb{C},\mathcal{F}}$ of $\mathcal{C}^{\mathcal{M,F}}$ which is the smallest class containing all power-constructible functions and stable under parametric Fourier transforms and right-composition with subanalytic maps. This class is also stable under parametric integration, under taking pointwise and $\text{L}^p$-limits, and under parametric Fourier-Plancherel transforms. Finally, we give a full asymptotic expansion in the power-logarithmic scale, uniformly in the parameters, for functions in $\mathcal{C}^{\mathbb{C},\mathcal{F}}$. 
\setcounter{tocdepth}{2}
\tableofcontents{}
\end{abstract}

\maketitle
\begin{acknowledgement*}
The third author would like to thank the Fields Institute for Research
in Mathematical Sciences for its hospitality and financial support,
as part of this work was done while at its Thematic Program on Tame
Geometry and Applications in 2022. The first author was partially
supported by KU Leuven IF C16/23/010 and the Labex CEMPI (ANR-11-LABX-0007-01).
\end{acknowledgement*}

\section{Introduction\label{sec:Introduction}}
 
Understanding integrals is at the heart of many mathematical problems, and often brings together challenges from both geometry and analysis. Indeed, integration is a transcendental process usually applied to functions naturally arising from basic geometric problems, and as such, having remarkable properties one aims to preserve. The present work is in the same spirit as Liouville's theorem on elementary integrals and its recent variants by Pila and Tsimerman (see  \cite{pila2022axschanuel}); it concerns rich classes of functions whose parametric integrals are of a somewhat similar nature as the original functions. 

To be more accurate, two types of problems may be considered in this spirit. 

The first consists in describing a  class of functions, possibly the smallest one, stable under parametric integration and containing a given class of functions. 
For instance, in the context of real o-minimal geometry, this kind of problem has been addressed for the class of semialgebraic and subanalytic functions. Indeed, in \cite{lr:int, clr,cluckers-miller:stability-integration-sums-products,cluckers_miller:loci_integrability} it has been proved that
the class $\mathcal{C}$
of constructible functions (that is to say the functions which are polynomials in globally subanalytic functions and their logarithms) form the smallest class of real-valued functions 
which contains all globally subanalytic functions and which is stable under parametric integration. 
In \cite{kaiser:integration_semialgebraic_nash} a proper subclass of $\mathcal{C}$ is introduced. This class is based on Nash functions (and their anti-derivatives) and turns out to be a small class of functions stable under parametric integration and containing the semialgebraic functions. Furthermore this class is suitable for studying families of periods as parametric integrals in the viewpoint of \cite{kontsevitch_zagier:periods} (see \cite{kaiser2023periods}).
In a similar spirit, we fully describe here the smallest class $\mathcal{C}^{\mathbb{C},\mathcal{F}}$ of functions which contains all complex powers and complex exponentials (of module one) of globally subanalytic functions, and stable under parametric integration (a natural framework for studying families of exponential periods, see \cite[Section 4.3]{kontsevitch_zagier:periods}).

The second type of problems, addressed here for the class $\mathcal{C}^{\mathcal{M,F}}$, still consists in describing a class of functions containing a given class of functions and stable under parametric integration, but we additionally require our class to be stable under other analytic key operations like Fourier and Mellin transforms. A hard part of this program consists then in finding the geometric properties preserved by the parametric integration process and our analytic transformations. The challenge here comes from the fact that by the action of these analytic transformations we leave the convenient framework of o-minimal geometry by introducing, via Mellin transforms, the (meromorphic) dependence on a complex parameter $s$. 
%
%

 Finally, let us note that several formalisms of motivic (and uniform $p$-adic) integration have a similar flavor and set-up. Such classes can then, for example, be used to define tame classes of distributions, which are at the same time stable under Fourier transform and analytically (wave front) holonomic \cite{AizC,AizCRS}.



\medskip

In this work, our starting point is the class $\mathcal{C}^{\mathbb{C}}$
of power-constructible functions, defined and studied in \cite{cmrs:mellin_constructible}, which
extends $\mathcal{C}$ by including complex powers of globally subanalytic
functions. This class includes complex-valued oscillatory functions,
hence we leave the realm of o-minimality, but many tame geometric
and analytic properties are preserved, such as stability under parametric
integration and well-understood (convergent) power-logarithmic asymptotics.
In \cite{ccmrs:integration-oscillatory} we studied the parametric
Fourier transforms of constructible functions (thus equally leaving
the realm of real geometry) and described a class containing such
transforms and stable under parametric integration. In \cite{cmrs:mellin_constructible}
we studied Mellin transforms of power-constructible functions and
showed that stability under parametric integration is preserved.

In the current paper we combine the action of parametric Fourier and
Mellin transforms on the class $\mathcal{C}^{\mathbb{C}}$ of power-constructible
functions. We define a system $\mathcal{C}^{\mathcal{M},\mathcal{F}}$
of $\mathbb{C}$-algebras containing all such transforms and stable
under parametric integration (see Definition \ref{def: C^M,F-1} and Theorem \ref{thm: Stability of C^M,F}). We describe a set of generators of a
particular prepared form which allows us to prove the stability under
parametric integration and deduce information about the asymptotics
at infinity in a chosen variable of the functions of the class. We furthermore
identify a subclass $\mathcal{C}^{\mathbb{C},\mathcal{F}}$ (see Definition \ref{def: C^F}) of $\mathcal{C}^{\mathcal{M},\mathcal{F}}$
which is the smallest class containing $\mathcal{C}^{\mathbb{C}}$
and stable under parametric Fourier transform and right-composition with subanalytic maps, and give, for the functions
of this class, asymptotic expansions in the power-logarithmic scale  (Theorem \ref{thm: asympt exp}),
in a chosen variable $y$ and uniformly in the other variables $x$
(which serve as parameters and range in a given globally subanalytic
set). We also deduce the stability of the class $\mathcal{C}^{\mathbb{C},\mathcal{F}}$
under taking pointwise and $\text{L}^{p}$-limits, and under the Fourier-Plancherel transform (Theorems \ref{prop:limits}, \ref{prop:complete}, \ref{cor:Plancherel}).\\

The main geometric tools for achieving this program come from o-minimality (see \cite{vdd:o_minimal_real_analytic}) and, more precisely, from the geometry of subanalytic sets and functions. They consist in resolution results in the form of preparation theorems, in the spirit of \cite{paru:lip1}, \cite{lr:int} or \cite{miller_dan:preparation_theorem_weierstrass_systems}. The key analytic tool we use is 
the theory of continuously uniformly distributed modulo one functions (c.u.d. mod 1, for short), building on  \cite{Weyl:Gleichverteilung,kuipers_niederreiter:uniform_distributions_sequences,ccmrs:integration-oscillatory}, and its uniform variants. One of the deep challenges comes from the oscillatory nature of functions in $\mathcal{C}^{\mathbb{C},\mathcal{F}}$ and $\mathcal{C}^{\mathcal{M},\mathcal{F}}$, which imposes a careful study of the integration loci (see Definition \ref{def: stability for functions on X} and Theorem \ref{thm: interpolation}). This is where the interaction between the theory of c.u.d. mod 1 functions and the geometry of subanalytic sets comes into play.
\\

The paper is organized as follows.

In Section \ref{sec: context} we introduce the classes $\mathcal{C}^{\mathbb{C},\mathcal{F}}$
and $\mathcal{C}^{\mathcal{M},\mathcal{F}}$, and state our main results
(Theorems \ref{thm: stability of C^F} and \ref{thm: Stability of C^M,F}).
The class $\mathcal{C}^{\mathbb{C},\mathcal{F}}$ is a collection
of functions defined on globally subanalytic sets, the class $\mathcal{C}^{\mathcal{M},\mathcal{F}}$
is a collection of functions that also depend on a complex parameter
$s$, which is only allowed to range in a vertical open strip with
bounded width. However, it is possible to extend a function on a given strip
to a larger strip (Proposition \ref{prop: C^M,F has extension}).

In Section \ref{sec: Generators}, we choose suitable generators for
$\mathcal{C}^{\mathcal{M},\mathcal{F}}$ as an abelian group, which allow us to prove the extension result (see Section \ref{subsec:proof2.20}).

In Section \ref{sec: Strategy} we identify two
special types of generators, \emph{strongly integrable} and \emph{monomial} (see Definition \ref{def: superintegrable and monomial genrators}), and
show that their parametric integrals still belong to the class $\mathcal{C}^{\mathcal{M},\mathcal{F}}$ (Corollary \ref{cor: integration of strongly integrable} and Proposition \ref{prop: integration of monomial generators}).
In Section \ref{subsec: Proof strategy} we give the proof of Theorems
\ref{thm: stability of C^F} and \ref{thm: Stability of C^M,F}, assuming Theorem \ref{thm: interpolation}, which is a precise form of Theorem \ref{thm: Stability of C^M,F} when integrating over only one variable.   
 
Sections \ref{sec: Preparation} and \ref{sec: stability} are devoted
to the proof of Theorem \ref{thm: interpolation}, which requires
both subanalytic resolution of singularities and preparation techniques,
and non-compensation results based on the theory of c.u.d. mod 1 functions. 

In Section \ref{sec:Asymptotic-expansions-and} we study
the asymptotics of the functions in the power-logarithmic scale 
and prove the stability of the class
$\mathcal{C}^{\mathbb{C},\mathcal{F}}$ under pointwise limits (Theorems \ref{thm: asympt exp} and \ref{prop:limits}).  

In Section \ref{sec:Plancherel} we prove the $\text{L}^{p}$-completeness of the class
$\mathcal{C}^{\mathbb{C},\mathcal{F}}$ and its stability under the parametric Fourier-Plancherel transform  (Theorems \ref{prop:complete} and \ref{cor:Plancherel}).  

\section{Context, definitions and main results\label{sec: context}}

A subset $X$ of $\mathbb{R}^{m}$ is globally subanalytic if it is
the image under the canonical projection from $\mathbb{R}^{m+n}$
to $\mathbb{R}^{m}$ of a globally semianalytic subset of $\mathbb{R}^{m+n}$
(i.e. a subset $Y\subseteq\mathbb{R}^{m+n}$ such that, in a neighborhood
of every point of $\mathbb{P}^{1}\left(\mathbb{R}\right)^{m+n}$,
$Y$ is described by finitely many analytic equations and inequalities).
Equivalently, $X$ is definable in the o-minimal structure $\mathbb{R}_{\text{an}}$
(see for example \cite{vdd:d}). Thus, the logarithm $\log:\left(0,+\infty\right)\longrightarrow\mathbb{R}$
and the power map $x^{y}:\left(0,+\infty\right)\times\mathbb{R}\longrightarrow\mathbb{R}$
are functions whose graph is not subanalytic, but they are definable
in the o-minimal structure $\mathbb{R}_{\mathrm{an},\exp}$ (see for
example \cite{dmm:exp}).

Throughout this paper $X\subseteq\mathbb{R}^{m}$ will be a globally
subanalytic set (from now on, just \textquotedblleft \emph{subanalytic}
\emph{set}\textquotedblright , for short). Denote by $\mathcal{S}\left(X\right)$
the collection of all subanalytic functions on $X$, i.e. all the
functions of domain $X$ whose graph is a subanalytic set, and let
$\mathcal{S}_{+}\left(X\right)=\left\{ f\in\mathcal{S}\left(X\right):\ f\left(X\right)\subseteq\left(0,+\infty\right)\right\} $.
\begin{notation}
\label{notation: class G}Whenever we fix, for every $m\in\mathbb{N}$
and $X\subseteq\mathbb{R}^{m}$ subanalytic, a collection $\mathcal{G}\left(X\right)$
of real- or complex-valued functions defined on $X$, we denote by
$\mathcal{G}$ the \emph{system} of all collections $\mathcal{G}\left(X\right)$.
For instance, $\mathcal{S}$ is the system of collections of all subanalytic
functions defined on subanalytic sets:
\[
\mathcal{S}=\left\{ \mathcal{S}\left(X\right):\ X\subseteq\mathbb{R}^{m}\text{ subanalytic},\ m\in\mathbb{N}\right\} .
\]
\end{notation}
\begin{defn}
\label{def: S^C and exp(iS)}For $X\subseteq\mathbb{R}^{m}$ subanalytic,
define
\begin{align*}
\mathcal{S}_{+}^{\mathbb{C}}\left(X\right) & =\left\{ f^{\alpha}:\ f\in\mathcal{S}_{+}\left(X\right),\ \alpha\in\mathbb{C}\right\} ,\\
\log\mathcal{S}_{+}\left(X\right) & =\left\{ \log f:\ f\in\mathcal{S}_{+}\left(X\right)\right\} ,\\
\text{e}^{\text{i}\mathcal{S}}\left(X\right) & =\left\{ \text{e}^{\text{i}f}:\ f\in\mathcal{S}\left(X\right)\right\} .
\end{align*}
\end{defn}
A function defined on $X$ and taking its values in $\mathbb{C}$
is called a \emph{complex-valued subanalytic function} if its real
and imaginary parts are in $\mathcal{S}\left(X\right)$. For example,
if $f\in\mathcal{S}\left(X\right)$ is bounded (i.e. for all $x\in X,$
$\left|f\left(x\right)\right|\leq M$, for some $M>0$), then $\text{e}^{\text{i}f}$
is a complex-valued subanalytic function. If such a bounded $f$ is furthermore strictly
positive (i.e. $f\in\mathcal{S}_{+}\left(X\right)$) and bounded away
from zero (i.e. for all $x\in X,\ f\left(x\right)\geq m$, for some
$m>0$), then $\log f$ is a real-valued subanalytic function and
for all $\alpha\in\mathbb{C}$, $f^{\alpha}$ is a complex-valued
subanalytic function.
\begin{defn}
\label{def: stability for functions on X}Let $\mathcal{G}$ be a
system as in Notation \ref{notation: class G}. For $h\in\mathcal{G}\left(X\times\mathbb{R}^{n}\right)$,
the \emph{integration locus} of $h$ on $X$ is the set
\[
\text{Int}\left(h;X\right)=\left\{ x\in X:\ y\longmapsto h\left(x,y\right)\in L^{1}\left(\mathbb{R}^{n}\right)\right\} .
\]
We say that $\mathcal{G}$ is \emph{stable under parametric integration}
if for all $h\in\mathcal{G}\left(X\times\mathbb{R}^{n}\right)$ there
exists $H\in\mathcal{G}\left(X\right)$ such that
\[
\forall x\in\text{Int}\left(h;X\right),\ H\left(x\right)=\int_{\mathbb{R}^{n}}h\left(x,y\right)\text{d}y.
\]
Finally, define $\mathcal{G}\left(X\times\mathbb{R}^{n}\right)_{\text{int}}=\left\{ h\in\mathcal{G}\left(X\times\mathbb{R}^{n}\right):\ \text{Int}\left(h;X\right)=X\right\} $.
\end{defn}
Thus, for example, $h\in\mathcal{G}\left(\left(X\times\mathbb{R}\right)\times\mathbb{R}\right)_{\text{int}}$
means that for all $\left(x,y\right)\in X\times\mathbb{R},\ t\longmapsto h\left(x,y,t\right)\in L^{1}\left(\mathbb{R}\right)$,
whereas $h\in\mathcal{G}\left(X\times\mathbb{R}^{2}\right)_{\text{int}}$
means that for all $x\in X,\ \left(y,t\right)\longmapsto h\left(x,y,t\right)\in L^{1}\left(\mathbb{R}^{2}\right)$.

Next, we introduce the parametric Fourier transform acting on a system
$\mathcal{G}$ as in Notation \ref{notation: class G}.
\begin{defn}
Let $h\in\mathcal{G}\left(X\times\mathbb{R}\right)_{\text{int}}$.
Define the \emph{parametric Fourier transform of $h$} as the function
\[
\mathcal{F}\left[h\right]:X\times\mathbb{R}\ni\left(x,t\right)\longmapsto\int_{\mathbb{R}}h\left(x,y\right)\text{e}^{-2\pi\text{i}ty}\text{d}y
\]
and the \emph{fixed frequency parametric Fourier transform of} $h$
as the function obtained from $\mathcal{F}\left[h\right]$ by fixing
$t=-\frac{1}{2\pi}$, i.e.
\[
\mathfrak{f}\left[h\right]:X\ni x\longmapsto\int_{\mathbb{R}}h\left(x,y\right)\text{e}^{\text{i}y}\text{d}y.
\]
\end{defn}
\begin{notation}
\label{notation: chi, open over}The letter $\chi$ will be used for
characteristic functions. Thus, if $A\subseteq\mathbb{R}^{n}$, then
$\chi_{A}$ will be the characteristic function of the set $A$.

We will often work in restriction to subanalytic cells $A\subseteq X\times\mathbb{R}$,
for some $X\subseteq\mathbb{R}^{m}$ subanalytic. If $x\in X$, then
$A_{x}$ denotes the fiber of $A$ over $x$, i.e. the set $\left\{ y\in\mathbb{R}:\ \left(x,y\right)\in A\right\} $.
As $X$ serves as a space of parameters (we will never integrate with
respect to the variables $x$ ranging in $X$), we are allowed to
partition $X$ into subanalytic cells, replace $X$ by one of the
cells of the partition and work disjointly in restriction to such
a cell. In particular, we may always assume that $X$ is itself a
subanalytic cell, and that all cells in $X\times\mathbb{R}$ project
onto $X$. Moreover, we will always concentrate on cells $A$ which
are \emph{open over }$X$ (see \cite[Definition 3.1]{cmrs:mellin_constructible}),
as these are the only cells whose fibers give a nonzero contribution
when integrating a function defined on $X\times\mathbb{R}$ with respect
to its last variable.
\end{notation}

\subsection{\label{subsec: Fourier transf}Fourier transforms of power-constructible
functions}

In \cite{ccmrs:integration-oscillatory} we constructed the smallest
system of $\mathbb{C}$-algebras containing $\mathcal{S}\cup\text{e}^{\text{i}\mathcal{S}}$
and stable under parametric integration. Such a system contains in
particular the parametric Fourier transforms of all subanalytic functions.
The first aim of this paper is to extend such a construction to describe
the smallest system $\mathcal{C}^{\mathbb{C},\mathcal{F}}$ containing
$\mathcal{S}_{+}^{\mathbb{C}}\cup\text{e}^{\text{i}\mathcal{S}}$
and stable under parametric integration. For this, our starting point
is the system $\mathcal{C}^{\mathbb{C}}$ of \emph{power-constructible
functions} defined in \cite{cmrs:mellin_constructible}. Let us recall
its definition and main properties.
\begin{thm}[{\cite[Definition 2.2  and Theorem 2.4]{cmrs:mellin_constructible}}]
\label{thm: power-constructible def and stability}For $X\subseteq\mathbb{R}^{m}$
subanalytic, let $\mathcal{C}^{\mathbb{C}}\left(X\right)$ be the
$\mathbb{C}$-algebra generated by $\mathcal{S}_{+}^{\mathbb{C}}\left(X\right)\cup\log\mathcal{S}_{+}\left(X\right)$.
The system $\mathcal{C}^{\mathbb{C}}$ of power-constructible functions
is the smallest system of $\mathbb{C}$-algebras containing $\mathcal{S}_{+}^{\mathbb{C}}$
and stable under parametric integration.
\end{thm}
A natural candidate for the smallest system containing $\mathcal{S}_{+}^{\mathbb{C}}\cup\text{e}^{\text{i}\mathcal{S}}$
and stable under parametric integration would be the system $\mathcal{C}^{\mathbb{C},\text{i}\mathcal{S}}$
of $\mathbb{C}$-algebras $\mathcal{C}^{\mathbb{C},\text{i}\mathcal{S}}\left(X\right)$
generated by $\mathcal{C}^{\mathbb{C}}\left(X\right)\cup\text{e}^{\text{i}\mathcal{S}}\left(X\right)$.
However, we will show (see Corollary \ref{cor: naive not enough})
that such a system is not stable under parametric integration. This
motivates the following definition.
\begin{defn}
\label{def: C^F}Consider the fixed frequency parametric Fourier operator
$\mathfrak{f}$ acting on $\mathcal{C}^{\mathbb{C}}$:
\[
\mathfrak{f}\left[g\right]\left(x\right)=\int_{\mathbb{R}}g\left(x,y\right)\text{e}^{\text{i}y}\text{dy}\ \ \ \left(g\in\mathcal{C}^{\mathbb{C}}\left(X\times\mathbb{R}\right)_{\text{int}}\right).
\]
Define
\[
\mathcal{C}^{\mathbb{C},\mathcal{F}}\left(X\right)=\left\{ \mathfrak{f}\left[g\right]:\ g\in\mathcal{C}^{\mathbb{C}}\left(X\times\mathbb{R}\right)_{\text{int}}\right\} .
\]
\end{defn}
\begin{rem}
\label{rem: 1 is a transcendental}Notice that $\mathcal{C}^{\mathbb{C},\mathcal{F}}\left(X\right)$
is a $\mathbb{C}$-module and that $1=\mathfrak{f}\left[\frac{\mathrm{i}}{2}\chi_{\left[\pi,2\pi\right]}\right]$.
In particular, $\mathcal{C}^{\mathbb{C}}\left(X\right)\subseteq\mathcal{C}^{\mathbb{C},\mathcal{F}}\left(X\right)$.
At this stage, it is not clear whether $\mathcal{C}^{\mathbb{C},\mathcal{F}}\left(X\right)$
is a $\mathbb{C}$-algebra.

Note also that $\mathcal{C}^{\mathbb{C},\mathcal{F}}$ is stable under
right-composition with subanalytic maps: if $X\subseteq\mathbb{R}^{m},Y\subseteq\mathbb{R}^{n}$
are subanalytic sets, $G:Y\longrightarrow X$ is a subanalytic map
and $h\in\mathcal{C}^{\mathbb{C},\mathcal{F}}\left(X\right)$, then
$h\circ G\in\mathcal{C}^{\mathbb{C},\mathcal{F}}\left(Y\right)$.
\end{rem}
Our first result is the following.
\begin{thm}
\label{thm: stability of C^F}The system $\mathcal{C}^{\mathbb{C},\mathcal{F}}$
is stable under parametric integration. It is a system of $\mathbb{C}$-algebras,
and indeed the smallest such system containing $\mathcal{S}_{+}^{\mathbb{C}}\cup\mathrm{e}^{\mathrm{i}\mathcal{S}}$
and stable under parametric integration. It is also the smallest such
system containing $\mathcal{C}^{\mathbb{C}}$ and stable under the
parametric Fourier transform and right-composition with subanalytic
maps.
\end{thm}
Subsequently, we derive results on asymptotic expansions, pointwise limits, $\text{L}^{p}$-limits, and the Fourier-Plancherel transform for the class $\mathcal{C}^{\mathbb{C},\mathcal{F}}$. Such results are stated and proven in Sections \ref{sec:Asymptotic-expansions-and} and \ref{sec:Plancherel}.

\subsection{\label{subsec: Fourier and Mellin}Parametric Mellin and Fourier
transforms of power-constructible functions}

We now turn our attention to the Mellin transform.
\begin{defn}
\label{def: mellin transf}Let $\Sigma\subseteq\mathbb{C}$ be an
open set. For $h\in\mathcal{C}^{\mathbb{C}}\left(X\times[0,+\infty)\right)$
such that for all $s\in\Sigma$ and for all $x\in X$, the function
$y\longmapsto y^{s-1}h\left(x,y\right)$ belongs to $L^{1}\left([0,+\infty)\right)$,
the \emph{parametric Mellin transform of }$h$ \emph{on }$\Sigma$
is the function
\[
\mathcal{M}_{\Sigma}\left[h\right]:\Sigma\times X\ni\left(s,x\right)\longmapsto\int_{0}^{+\infty}y^{s-1}h\left(x,y\right)\text{d}y.
\]
\end{defn}
In \cite{cmrs:mellin_constructible} we studied the parametric Mellin
transforms of power-constructible functions: we constructed a system
$\mathcal{C}^{\mathcal{M}}$ containing such transforms and stable
under parametric integration (see Definition \ref{def: C^M-1} and
Theorem \ref{thm: C_M stability} below). The second aim of this work
is to construct a system containing both the parametric Mellin transforms
and the parametric Fourier transforms of power-constructible functions,
and stable under parametric integration. As the Mellin transform introduces
a new complex variable $s$, the domains of the functions we consider
will be suitable subsets of $\mathbb{C}\times$$\mathbb{R}^{m}$,
rather than just subsets of $\mathbb{R}^{m}$. The notions of integration
locus, parametric integral transform and stability under parametric
integration need to be made precise in this new context, which is
what we do next.  

\medskip{}

In what follows, we will consider several collections of functions
defined on sets of the form $\Sigma\times X$, where $\Sigma$ is
a suitable subset of $\mathbb{C}$ and $X$ is a subanalytic subset
of $\mathbb{R}^{m}$, for some $m\in\mathbb{N}$. We will study the
action of some integral operators on these collections of functions,
and, more generally, the nature of the parametric integrals of such
functions. Let us fix some notation.
\begin{defn}
\label{def: strip}An open vertical strip of bounded width in $\mathbb{C}$
is a set of the form
\[
\Sigma=\left\{ s\in\mathbb{C}:\ p<\Re\left(s\right)<q\right\} ,
\]
where $p,q\in\mathbb{R}$ and $p<q$. For short, we will say that
$\Sigma$ is a \emph{strip}.
\end{defn}
\begin{notation}
\label{notation: class D}Given a strip $\Sigma\subseteq\mathbb{C}$
and a subanalytic set $X\subseteq\mathbb{R}^{m}$, let $\mathcal{D}_{\Sigma}\left(X\right)$
be a collection of complex-valued functions such that for all $h\in\mathcal{D}_{\Sigma}\left(X\right)$
there is a closed discrete set $P\subseteq\mathbb{C}$ such that the
domain of $h$ contains $\left(\Sigma\setminus P\right)\times X$
(we say that $h$ \emph{has no poles outside }$P$). Denote by $\mathcal{D}_{\Sigma}$
the system $\left\{ \mathcal{D}_{\Sigma}\left(X\right):\ X\subseteq\mathbb{R}^{m}\ \text{subanalytic, }m\in\mathbb{N}\right\} $.

Suppose furthermore that the collection $\left\{ \mathcal{D}_{\Sigma}:\ \Sigma\subseteq\mathbb{C}\ \text{strip}\right\} $
has the \emph{extension property}: for every subanalytic set $X\subseteq\mathbb{R}^{m}$,
given any two strips $\Sigma,\Sigma'$ such that $\Sigma\subseteq\Sigma'$
and a closed discrete set $P\subseteq\mathbb{C}$ and $h\in\mathcal{D}_{\Sigma}\left(X\right)$
without poles outside $P$, there exists $h'\in\mathcal{D}_{\Sigma'}\left(X\right)$
without poles outside $P$ such that $h'\restriction\left(\Sigma\setminus P\right)\times X=h$.
Define $\mathcal{D}\left(X\right)$ as the direct limit of $\left\{ \mathcal{D}_{\Sigma}\left(X\right):\ \Sigma\subseteq\mathbb{C}\ \text{strip}\right\} $
and $\mathcal{D}=\left\{ \mathcal{D}\left(X\right):\ X\subseteq\mathbb{R}^{m}\text{ subanalytic},\ m\in\mathbb{N}\right\} $.
For $h\in\mathcal{D}\left(X\right)$ and a closed discrete set $P\subseteq\mathbb{C}$,
we say that $h$ has no poles outside $P$ if this is the case for
some representative of $h$ on each strip $\Sigma$.
\end{notation}
\begin{defn}
\label{def: locus}Given $h\in\mathcal{D}_{\Sigma}\left(X\times\mathbb{R}^{n}\right)$
without poles outside some closed discrete set $P\subseteq\mathbb{C}$,
define the \emph{integration locus} of $h$ as
\[
\text{Int}\left(h;\left(\Sigma\setminus P\right)\times X\right)=\left\{ \left(s,x\right)\in\left(\Sigma\setminus P\right)\times X:\ y\longmapsto h\left(s,x,y\right)\in L^{1}\left(\mathbb{R}^{n}\right)\right\}
\]
and we set
\begin{align*}
\mathcal{D}_{\Sigma}\left(X\times\mathbb{R}^{n}\right)_{\text{int}}= & \{h\in\mathcal{D}_{\Sigma}\left(X\times\mathbb{R}^{n}\right):\ \text{Int}\left(h;\left(\Sigma\setminus P\right)\times X\right)=\left(\Sigma\setminus P\right)\times X,\\
 & \text{for some closed\ discrete }P\subseteq\mathbb{C}\}.
\end{align*}
\end{defn}
We consider the following parametric integral transforms acting on
$\mathcal{D}$, where the word \emph{generalized} refers to the fact
that, unlike the case of the corresponding classical transforms, we
allow the operator to act on functions for which the integral transform
is not everywhere defined.
\begin{defn}
\label{def: gen Mellin, Fourier transforms}Let $\mathcal{D}$ be
as in Notation \ref{notation: class D} and $h\in\mathcal{D}_{\Sigma}\left(X\times\mathbb{R}\right)$
be without poles outside some closed discrete set $P\subseteq\mathbb{C}$.
\begin{itemize}
\item Let $\chi_{+}$ be the characteristic function of the half-line $[0,+\infty)$
and
\[
\widetilde{h}\left(s,x,y\right)=\chi_{+}\left(y\right)y^{s-1}h\left(s,x,y\right).
\]
The \emph{generalized parametric Mellin transform} of $h$ is the
function defined on\\
$\text{Int}\left(\widetilde{h};\left(\Sigma\setminus P\right)\times X\right)$
given by
\[
\mathcal{M}\left[h\right]\left(s,x\right)=\int_{0}^{+\infty}y^{s-1}h\left(s,x,y\right)\text{d}y.
\]
The integration kernel of this transform is the function $\left(s,y\right)\longmapsto\chi_{+}\left(y\right)y^{s-1}$.
\item The \emph{generalized parametric Fourier transform} of $h$ is the
function defined on \\
$\text{Int}\left(h;\left(\Sigma\setminus P\right)\times X\right)\times\mathbb{R}$
given by
\[
\mathcal{F}\left[h\right]\left(s,x,t\right)=\int_{\mathbb{R}}h\left(s,x,y\right)\text{e}^{-2\pi\text{i}ty}\text{d}y.
\]
The integration kernel is the function $\left(t,y\right)\longmapsto\text{e}^{-2\pi\text{i}ty}$.
\item The \emph{generalized fixed frequency} \emph{parametric Fourier transform}
of $h$ is the function defined on $\text{Int}\left(h;\left(\Sigma\setminus P\right)\times X\right)$
given by
\[
\mathfrak{f}\left[h\right]\left(s,x\right)=\mathcal{F}\left[h\right]\left(s,x,-\frac{1}{2\pi}\right)=\int_{\mathbb{R}}h\left(s,x,y\right)\text{e}^{\text{i}y}\text{d}y.
\]
The integration kernel is the function $y\longmapsto\text{e}^{\text{i}y}$.
\end{itemize}
\end{defn}
For each of these operators, the elements of the pairs $\left(s,x\right)$
for which the parametric transform of $h$ is defined are called the
\emph{parameters} of the transform.

\begin{defn}
\label{def: stable under transform} Let $\mathcal{D}$ be as in Notation
\ref{notation: class D}.
\begin{itemize}
\item $\mathcal{D}$ is \emph{stable under the generalized parametric Mellin
transform} if for all $\Sigma$ and $X$, for all $h\in\mathcal{D}_{\Sigma}\left(X\times\mathbb{R}\right)$
without poles outside some closed discrete set $P\subseteq\mathbb{C}$
there are a closed discrete set $P'\subseteq\mathbb{C}$ such that
$P\subseteq P'\subseteq\mathbb{C}$, and a function $H\in\mathcal{D}_{\Sigma}\left(X\right)$
without poles outside $P'$ such that, if $\widetilde{h}\left(s,x,y\right)=\chi_{+}\left(y\right)y^{s-1}h\left(s,x,y\right)$,
then
\[
\forall\left(s,x\right)\in\text{Int}\left(\widetilde{h};\left(\Sigma\setminus P'\right)\times X\right),\ H\left(s,x\right)=\mathcal{M}\left[h\right]\left(s,x\right).
\]
\item $\mathcal{D}$ is \emph{stable under the generalized parametric Fourier
transform} if for all $\Sigma$ and $X$, for all $h\in\mathcal{D}_{\Sigma}\left(X\times\mathbb{R}\right)$
without poles outside some closed discrete set $P\subseteq\mathbb{C}$
there are a closed discrete set $P'\subseteq\mathbb{C}$ such that
$P\subseteq P'\subseteq\mathbb{C}$, and a function $H\in\mathcal{D}_{\Sigma}\left(X\times\mathbb{R}\right)$
without poles outside $P'$ such that
\[
\forall\left(s,x,t\right)\in\text{Int}\left(h;\left(\Sigma\setminus P'\right)\times X\right)\times\mathbb{R},\ H\left(s,x,t\right)=\mathcal{F}\left[h\right]\left(s,x,t\right).
\]
\end{itemize}
\end{defn}
\begin{defn}
\label{def: stability under integration general}A system $\mathcal{D}$
as in Notation \ref{notation: class D} is \emph{stable under parametric
integration} if for every strip $\Sigma\subseteq\mathbb{C}$ and every
subanalytic set $X\subseteq\mathbb{R}^{m}$, given $h\in\mathcal{D}_{\Sigma}\left(X\times\mathbb{R}^{n}\right)$
without poles outside some closed discrete set $P\subseteq\mathbb{C}$
there exists a closed discrete set $P'\subseteq\mathbb{C}$ such that
$P\subseteq P'$ and $P'\setminus P$ is contained in a finitely generated
$\mathbb{Z}$-lattice, and there exists a function $H\in\mathcal{D}_{\Sigma}\left(X\right)$
without poles outside $P'$ such that
\[
\forall\left(s,x\right)\in\text{Int}\left(h;\left(\Sigma\setminus P'\right)\times X\right),\ H\left(s,x\right)=\int_{\mathbb{R}^{n}}h\left(s,x,y\right)\text{d}y.
\]
\end{defn}

\subsubsection{\label{subsec: parametric power constructible}Parametric power-constructible
functions}

Recall the following definitions and results from \cite{cmrs:mellin_constructible}.
\begin{defn}
\label{def: C^M-1}$\ $
\begin{itemize}
\item ($1$-bounded subanalytic maps) For $N\in\mathbb{N}$, we let $\mathcal{S}_{c}^{N}\left(X\right)$
be the collection of all maps $\psi:X\longrightarrow\mathbb{R}^{N}$
with components in $\mathcal{S}\left(X\right)$, such that $\overline{\psi\left(X\right)}$
is contained in the closed polydisk of $\mathbb{R}^{N}$ centered
at zero and of radius $1$. The members of the collection $\mathcal{S}_{c}\left(X\right)=\bigcup_{N\in\mathbb{N^{\times}}}\mathcal{S}_{c}^{N}\left(X\right)$
are called \emph{1-bounded} subanalytic maps defined on $X$.
\item (Strongly convergent series) Let $\mathcal{E}$ be the field of meromorphic
functions $\xi:\mathbb{C}\longrightarrow\mathbb{C}$ and denote by
$D^{N}$ the closed polydisk of radius $\frac{3}{2}$ and center $0\in\mathbb{R}^{N}$.
Given a formal power series $F=\sum_{I}\xi_{I}\left(s\right)Z^{I}\in\mathcal{E}\left\llbracket Z\right\rrbracket $
in $N$ variables $Z$ and with coefficients $\xi_{I}\in\mathcal{E}$,
we say that $F$ \emph{converges strongly} if there exists a closed
discrete set $P\left(F\right)\subseteq\mathbb{C}$ (called the \emph{set
of poles} of $F$) such that:
\begin{itemize}
\item for every $s_{0}\in\mathbb{C}\setminus P\left(F\right)$, the power
series $F\left(s_{0},Z\right)\in\mathbb{C}\left\llbracket Z\right\rrbracket $
converges in a neighbourhood of $D^{N}$ (thus $F$ defines a function
on $\left(\mathbb{C}\setminus P\left(F\right)\right)\times D^{N}$);
\item for every $s_{0}\in\mathbb{C}$ there exists $m=m\left(s_{0}\right)\in\mathbb{N}$
such that for all $z_{0}\in D^{N}$, the function $\left(s,z\right)\longmapsto\left(s-s_{0}\right)^{m}F\left(s,z\right)$
has a holomorphic extension on some complex neighbourhood of $\left(s_{0},z_{0}\right)$;
\item $P\left(F\right)$ is the set of all $s_{0}\in\mathbb{C}$ such that
the minimal such $m\left(s_{0}\right)$ is strictly positive.
\end{itemize}
\item (\textbf{Parametric strong functions}) Given a closed discrete set
$P\subseteq\mathbb{C}$, a function $\Phi:\left(\mathbb{C}\setminus P\right)\times X\longrightarrow\mathbb{C}$
is called a \emph{parametric strong function} on $X$ if there exist
a $1$-bounded subanalytic map $\psi\in\mathcal{S}_{c}^{N}\left(X\right)$
and a strongly convergent series $F=\sum_{I}\xi_{I}\left(s\right)Z^{I}\in\mathcal{E}\left\llbracket Z\right\rrbracket $
with $P\left(F\right)\subseteq P$ such that,
\[
\forall\left(s,x\right)\in\left(\mathbb{C}\setminus P\right)\times X,\ \Phi\left(s,x\right)=F\circ\left(s,\psi\left(x\right)\right)=\sum_{I}\xi_{I}\left(s\right)\left(\psi\left(x\right)\right)^{I}.
\]
Define $\mathcal{A}\left(X\right)$ as the collection of all parametric
strong functions on $X$. If $\Phi\in\mathcal{A}\left(X\right)$ has
no poles outside $P\subseteq\mathbb{C}$, then for all $s\in\mathbb{C}\setminus P,\ x\longmapsto\Phi\left(s,x\right)$
is bounded. If furthermore for all $s\in\mathbb{C}\setminus P,\ x\longmapsto\Phi\left(s,x\right)$
is bounded away from zero then we call $\Phi$ a \emph{parametric
strong unit}. A parametric strong function which happens not to depend
on the variable $s$ is called a \emph{subanalytic strong function}.
\item (Parametric powers) For $X\subseteq\mathbb{R}^{m}$ subanalytic, define
the \emph{parametric powers of $\mathcal{S}$ }on $X$ as the functions
in the collection
\[
\mathcal{P}\left(\mathcal{S}_{+}\left(X\right)\right)=\{P_{f}:\mathbb{C}\times X\longrightarrow\mathbb{C}\text{\ such\ that}\ P_{f}\left(s,x\right)=f\left(x\right)^{s},\ \text{for\ some\ }f\in\mathcal{S}_{+}\left(X\right)\}.
\]
\item (\textbf{Parametric power-constructible functions}) If $X\subseteq\mathbb{R}^{0}$,
then define $\mathcal{C}^{\mathcal{M}}\left(X\right)=\mathcal{E}$.
If $X\subseteq\mathbb{R}^{m}$, with $m>0$, then we let $\mathcal{C}^{\mathcal{M}}\left(X\right)$
be the $\mathcal{A}\left(X\right)$-algebra generated by $\mathcal{C}^{\mathbb{C}}\left(X\right)\cup\mathcal{P}\left(\mathcal{S}_{+}\left(X\right)\right)$.
The system $\mathcal{C}^{\mathcal{M}}$ is the collection of algebras
of \emph{parametric power-constructible functions}. Every function
$h\in\mathcal{C}^{\mathcal{M}}\left(X\right)$ can be written on $\left(\mathbb{C}\setminus P\right)\times X$
(for some closed discrete $P\subseteq\mathbb{C}$) as a closed discrete
sum of \emph{generators} of the form
\begin{equation}
\Phi\left(s,x\right)\cdot g\left(x\right)\cdot f\left(x\right)^{s},\label{eq:generator-C^M}
\end{equation}
where $g\in\mathcal{C}^{\mathbb{C}}\left(X\right),\ f\in\mathcal{S}_{+}\left(X\right)$
and $\Phi\in\mathcal{A}\left(X\right)$ has no poles outside $P$.
Here the word \textquotedblleft parametric\textquotedblright{} refers
to the variable $s\in\mathbb{C}$ seen as a new complex parameter
(alongside the real parameters $x\in X$).
\end{itemize}
\end{defn}
The functions in $\mathcal{C}^{\mathcal{M}}$ have a domain of the
form $\left(\mathbb{C}\setminus P\right)\times X$. We are interested
in studying functions defined on domains of the form $\left(\Sigma\setminus P\right)\times X$,
where $\Sigma$ is a strip. For this, we define $\mathcal{C}_{\Sigma}^{\mathcal{M}}\left(X\right)$
as the collection of all restrictions to $\Sigma\times X$ of functions
in $\mathcal{C}^{\mathcal{M}}\left(X\right)$ and thus form the systems
$\mathcal{A},\mathcal{P}\left(\mathcal{S}_{+}\right)$ and $\mathcal{C}^{\mathcal{M}}$,
proceeding as in Notation \ref{notation: class D} (note that, since
the functions in these collections are defined on the whole of $\mathbb{C}$
and not just on strips, the two definitions of $\mathcal{C}^{\mathcal{M}}$
coincide). With this notation we immediately derive from \cite{cmrs:mellin_constructible}
the following result.
\begin{thm}[{\cite[Theorem 2.16 and Corollary 2.18]{cmrs:mellin_constructible}}]
\label{thm: C_M stability}The system $\mathcal{C}^{\mathcal{M}}$
is stable under parametric integration. Moreover, $\mathcal{C}^{\mathcal{M}}$
is the smallest system of $\mathcal{A}$-algebras containing $\mathcal{\mathcal{C}}^{\mathbb{C}}$
and stable under the generalized parametric Mellin transform.
\end{thm}

\subsubsection{\label{subsec: Main results}Parametric Fourier transforms of parametric
power-constructible functions}

Our next goal is to define a system containing both the parametric
Fourier and the parametric Mellin transforms of power-constructible
functions.
\begin{defn}
\label{def: C^M,F-1}Let $\Sigma\subseteq\mathbb{C}$ be a strip.
Consider the fixed frequency parametric Fourier operator $\mathfrak{f}$
acting on $\mathcal{C}_{\Sigma}^{\mathcal{M}}$:
\[
\mathfrak{f}\left[h\right]\left(s,x\right)=\int_{\mathbb{R}}h\left(s,x,y\right)\text{e}^{\text{i}y}\text{dy}\ \ \ \left(h\in\mathcal{C}_{\Sigma}^{\mathcal{M}}\left(X\times\mathbb{R}\right)_{\text{int}}\right).
\]
If $h$ has no poles outside some closed discrete set $P\subseteq\mathbb{C}$,
then so does $\mathfrak{f}\left[h\right]$. Define
\[
\mathcal{C}_{\Sigma}^{\mathcal{M},\mathcal{F}}\left(X\right)=\left\{ \mathfrak{f}\left[h\right]:\ h\in\mathcal{C}_{\Sigma}^{\mathcal{M}}\left(X\times\mathbb{R}\right)_{\text{int}}\right\} .
\]
It is a $\mathcal{C}_{\Sigma}^{\mathcal{M}}\left(X\right)$-module.
\end{defn}
We will show in Section \ref{sec: Generators} that the functions
in the above collection can be extended to the whole complex plane,
in the sense of Notation \ref{notation: class D}:
\begin{prop}
\label{prop: C^M,F has extension}The collection $\left\{ \mathcal{C}_{\Sigma}^{\mathcal{M},\mathcal{F}}\left(X\right):\ \Sigma\ \text{strip}\right\} $
has the extension property.
\end{prop}
Thanks to the above proposition, we may define the system
\[
\mathcal{C}^{\mathcal{M},\mathcal{F}}=\left\{ \mathcal{C}^{\mathcal{M},\mathcal{F}}\left(X\right):\ X\subseteq\mathbb{R}^{m}\text{ subanalytic},\ m\in\mathbb{N}\right\} .
\]

Our main stability result is the following.
\begin{thm}
\label{thm: Stability of C^M,F}The system $\mathcal{C}^{\mathcal{M},\mathcal{F}}$
is stable under parametric integration. It is a system of $\mathbb{C}$-algebras,
containing $\mathcal{C}^{\mathbb{C}}\cup\mathrm{e}^{\mathrm{i}\mathcal{S}}$,
and stable under generalized parametric Mellin and Fourier transforms.
\end{thm}

\section{\label{sec: Generators}Generators of $\mathcal{C}^{\mathcal{M},\mathcal{F}}$ and proof of the extension result}

\subsection{Generators of $\mathcal{C}^{\mathcal{M},\mathcal{F}}$} 
In this section we choose a set of generators for $\mathcal{C}_{\Sigma}^{\mathcal{M},\mathcal{F}}\left(X\right)$
as an additive group, of a special form, which is suitable for proving
Proposition \ref{prop: C^M,F has extension} and Theorem \ref{thm: Stability of C^M,F}.

First, we recall some definitions from \cite{cmrs:mellin_constructible}.
\begin{defn}
\label{def: notation from Mellin, cells psi}Let $X\subseteq\mathbb{R}^{m}$
be a subanalytic cell and
\begin{equation}
B=\left\{ \left(x,y\right):\ x\in X,\ a\left(x\right)<y<b\left(x\right)\right\} ,\label{eq:A_theta=00003DB}
\end{equation}
where $a,b:X\longrightarrow\mathbb{R}$ are analytic subanalytic functions
with $1\leq a\left(x\right)<b\left(x\right)\ \text{for all }x\in X$,
and $b$ is allowed to be $\equiv+\infty$. We say that $B$ has\emph{
bounded $y$-fibers} if $b<+\infty$ and \emph{unbounded $y$-fibers}
if $b\equiv+\infty$.
\begin{itemize}
\item A $1$-bounded subanalytic map $\psi:B\longrightarrow\mathbb{R}^{M+2}\in\mathcal{S}_{c}^{M+2}\left(B\right)$
is \emph{$y$-prepared} if it has the form
\begin{equation}
\psi\left(x,y\right)=\left(c\left(x\right),\left(\frac{a\left(x\right)}{y}\right)^{\frac{1}{d}},\left(\frac{y}{b\left(x\right)}\right)^{\frac{1}{d}}\right),\label{eq: psi}
\end{equation}
where $d\in\mathbb{N}\setminus\left\{ 0\right\} $. \\
If $b\equiv+\infty$, then we will implicitly assume that the last
component is missing and hence $\psi:B\longrightarrow\mathbb{R}^{M+1}$.
\item A subanalytic strong unit $U\in\mathcal{S}\left(B\right)$ is \emph{$\psi$-prepared}
if there exists a strongly convergent series $F\in\mathbb{R}\left\llbracket Z\right\rrbracket $
such that $U=F\circ\psi$. If $B$ has unbounded $y$-fibers, then
the \emph{nested $\psi$-prepared form} of $U$ is
\begin{equation}
U\left(x,y\right)=\sum_{k}b_{k}\left(x\right)\left(\frac{a\left(x\right)}{y}\right)^{\frac{k}{d}},\label{eq: nester sub strong unit}
\end{equation}
where the subanalytic functions $b_{k}$ are bounded and $b_{0}$
does not vanish on $X$.
\item A parametric strong function $\Phi\in\mathcal{A}_{\Sigma}\left(B\right)$
is \emph{$\psi$-prepared} if there exists a strongly convergent series
$F=\sum\xi_{I}\left(s\right)Z^{I}\in\mathcal{E}_{\mathbb{}}\left\llbracket Z\right\rrbracket $
such that
\begin{equation}
\forall\left(s,x,y\right)\in\left(\Sigma\setminus P\left(F\right)\right)\times B,\ \Phi\left(s,x,y\right)=F\circ\left(s,\psi\left(x,y\right)\right).\label{eq: prep param strong}
\end{equation}
If $B$ has unbounded $y$-fibers, then the \emph{nested $\psi$-prepared
form} of $\Phi$ is
\begin{align}
 & \Phi\left(s,x,y\right)=\sum_{k}\xi_{k}\left(s,x\right)\left(\frac{a\left(x\right)}{y}\right)^{\frac{k}{d}},\text{\ where}\ \xi_{k}\left(s,x\right)\in\mathcal{A}_{\Sigma}\left(X\right).\label{eq:eq: nested prep form unb}
\end{align}
\item A subanalytic function $\varphi\in\mathcal{S}\left(B\right)$ is \emph{prepared}
if there are $\omega\in\mathbb{Z}$, an analytic function $\varphi_{0}\in\mathcal{S}\left(X\right)$
and a $\psi$-prepared subanalytic strong unit $U$ such that
\begin{equation}
\varphi\left(x,y\right)=\varphi_{0}\left(x\right)y^{\frac{\omega}{d}}U\left(x,y\right).\label{eq: prepared subanalytic}
\end{equation}
\end{itemize}
\end{defn}
In order to choose suitable generators for $\mathcal{C}^{\mathcal{M},\mathcal{F}}$,
we first need to introduce two additional classes of functions.
\begin{defn}
\label{def: naive and generators}Let $\Sigma\subseteq\mathbb{C}$
be a strip and $X\subseteq\mathbb{R}^{m}$ be a subanalytic set. Let
$B$ be as in \eqref{eq:A_theta=00003DB}.
\begin{itemize}
\item Let $\mathcal{C}_{\Sigma}^{\mathcal{M},\text{i}\mathcal{S}}\left(X\right)$
be the additive group generated by the functions of the form
\begin{equation}
g\text{e}^{\text{i}\varphi}\ \ \ \left(g\in\mathcal{C}_{\Sigma}^{\mathcal{M}}\left(X\right),\varphi\in\mathcal{S}\left(X\right)\right).\label{eq: naive gen}
\end{equation}
It is a $\mathbb{C}$-algebra.
\item A \emph{transcendental element} is a function of the form
\[
\left(\Sigma\setminus P\right)\times X\ni\left(s,x\right)\longmapsto\gamma\left(s,x\right)=\int_{\mathbb{R}}\chi_{B}\left(x,y\right)y^{\lambda\left(s\right)}\left(\log y\right)^{\mu}\Phi\left(s,x,y\right)\text{e}^{\sigma\text{i}y}\text{d}y,
\]
where $\sigma\in\left\{ +,-\right\} ,\mu\in\mathbb{N},\Phi$ is a
$\psi$-prepared parametric strong function (as in \eqref{eq: prep param strong},
with $\psi$ as in \eqref{eq: psi}) without poles outside the closed
discrete set $P\subseteq\mathbb{C}$ and $\lambda\left(s\right)={\displaystyle \frac{\ell s+\eta}{d}}$,
for some $\ell\in\mathbb{Z},\eta\in\mathbb{C}$ and the same $d$
appearing in \eqref{eq: psi}. If $B$ has unbounded $y$-fibers,
then we require that for all $s\in\Sigma,\ \Re\left(\lambda\left(s\right)\right)<-1$.
\\
We let $\Gamma_{\Sigma}\left(X\right)$ be the collection of all transcendental
elements on $\Sigma\times X$.\\
Thus, a \emph{generator} (as an additive group) of the $\mathcal{C}_{\Sigma}^{\mathcal{M},\text{i}\mathcal{S}}\left(X\right)$-module
generated by the set $\Gamma_{\Sigma}\left(X\right)$ is a function
of the form
\begin{equation}
T=g\text{e}^{\text{i}\varphi}\gamma\ \ \ \left(g\in\mathcal{C}_{\Sigma}^{\mathcal{M}}\left(X\right),\varphi\in\mathcal{S}\left(X\right),\gamma\in\Gamma_{\Sigma}\left(X\right)\right).\label{eq:generator}
\end{equation}
\end{itemize}
Notice that $1=\mathfrak{f}\left[\frac{\mathrm{i}}{2}\chi_{\left[\pi,2\pi\right]}\right]\in\Gamma_{\Sigma}\left(X\right)$.
In particular, \eqref{eq: naive gen} is an instance of \eqref{eq:generator}.
\end{defn}
\begin{lem}
\label{lem: gen contenu}Let $T$ be a generator as in \eqref{eq:generator},
without poles outside some closed discrete set $P\subseteq\mathbb{C}$.
There exists $h\in\mathcal{C}_{\Sigma}^{\mathcal{M}}\left(X\times\mathbb{R}\right)_{\mathrm{int}}$
without poles outside $P$ such that $T=\mathfrak{f}\left[h\right]$.
In particular,
\[
\mathcal{C}_{\Sigma}^{\mathcal{M},\mathrm{i}\mathcal{S}}\left(X\right),\Gamma_{\Sigma}\left(X\right)\subseteq\mathcal{C}_{\Sigma}^{\mathcal{M},\mathcal{F}}\left(X\right).
\]
\end{lem}
\begin{proof}
Let $B$ be as in \eqref{eq:A_theta=00003DB} and
\[
G\left(s,x,y\right)=\chi_{B}\left(x,y\right)g\left(s,x\right)y^{\lambda\left(s\right)}\left(\log y\right)^{\mu}\Phi\left(s,x,y\right).
\]
Then $G\in\mathcal{C}_{\Sigma}^{\mathcal{M}}\left(X\times\mathbb{R}\right)_{\text{int}}$
and
\[
T\left(s,x\right)=\int_{\mathbb{R}}G\left(s,x,y\right)\text{e}^{\text{i}\left(\varphi\left(x\right)+\sigma y\right)}\text{d}y.
\]
Thus, by a change of variables, $T=\mathfrak{f}\left[h\right]$ with
$h\left(s,x,y\right):=\sigma G\left(s,x,\sigma\left(y-\varphi\left(x\right)\right)\right)\in\mathcal{C}_{\Sigma}^{\mathcal{M}}\left(X\times\mathbb{R}\right)_{\text{int}}$.
Notice that $h$ has no poles outside $P$.
\end{proof}
\begin{lem}
\label{lem: gen contient}Let $h\in\mathcal{C}_{\Sigma}^{\mathcal{M},\mathcal{F}}\left(X\right)$.
There are a closed discrete set $P\subseteq\mathbb{C}$ and finitely
many generators $T_{1},\ldots,T_{m}$ as in \eqref{eq:generator}
such that $h,T_{1},\ldots,T_{m}$ have no poles outside $P$ and $h=\sum T_{j}$.
In particular, $\mathcal{C}_{\Sigma}^{\mathcal{M},\mathcal{F}}\left(X\right)$
can also be described as the $\mathcal{C}_{\Sigma}^{\mathcal{M},\mathrm{i}\mathcal{S}}\left(X\right)$-module
generated by the set $\Gamma_{\Sigma}\left(X\right)$ , and the functions
of the form \eqref{eq:generator} are generators of $\mathcal{C}_{\Sigma}^{\mathcal{M},\mathcal{F}}\left(X\right)$
as an additive group.
\end{lem}
Recall the notation established right after Definition 3.9 in \cite{cmrs:mellin_constructible}.
\begin{notation}
\label{notation: cell bijection B_A}Let $A\subseteq X\times\mathbb{R}$
be a subanalytic cell which is open over and projects onto $X$ (see
Notation \ref{notation: chi, open over}), let $\theta_{A}$ be its
center, so that the set $\left\{ y-\theta_{A}\left(x\right):\ \left(x,y\right)\in A\right\} $
is contained in one of the sets $\left(-\infty,-1\right),\left(-1,0\right),\left(0,1\right),\left(1,+\infty\right)$,
as in \cite[Definition 3.4]{ccmrs:integration-oscillatory}. There
are unique sign conditions $\sigma_{A},\tau_{A}\in\left\{ -1,1\right\} $
such that
\begin{equation}
A=\left\{ \left(x,y\right):\ x\in X,\ a_{A}\left(x\right)<\sigma_{A}\left(y-\theta_{A}\left(x\right)\right)^{\tau_{A}}<b_{A}\left(x\right)\right\} \label{eq:A}
\end{equation}
for some analytic subanalytic functions $a_{A},b_{A}$ such that $1\leq a_{A}\left(x\right)<b_{A}\left(x\right)\leq+\infty$.
Let
\begin{equation}
B_{A}=\left\{ \left(x,y\right):\ x\in X,\ a_{A}\left(x\right)<y<b_{A}\left(x\right)\right\} \label{eq: Atheta}
\end{equation}
and $\Pi_{A}:B_{A}\longrightarrow A$ be the bijection
\begin{equation}
\Pi_{A}\left(x,y\right)=\left(x,\sigma_{A}y^{\tau_{A}}+\theta_{A}\left(x\right)\right),\ \Pi_{A}^{-1}\left(x,y\right)=\left(x,\sigma_{A}\left(y-\theta_{A}\left(x\right)\right)^{\tau_{A}}\right).\label{eq:Ptheta}
\end{equation}
We will still denote by $\Pi_{A}$ the map $\mathbb{C}\times B_{A}\ni\left(s,x,y\right)\longmapsto\left(s,\Pi_{A}\left(x,y\right)\right)\in\mathbb{C}\times A$.
\end{notation}
\begin{rem}
\label{rem: cell at infty}By \cite[Definition 3.4(3)]{ccmrs:integration-oscillatory},
if $A$ is a cell of the form $A=\{\left(x,y\right):\ x\in X,\ y>f\left(x\right)\}$
with $f\in\mathcal{S}\left(X\right)$ and $f\geq1$, then $\sigma_{A}=\tau_{A}=1$
and $\theta_{A}=0$. Hence in this case $a_{A}=f,\ b_{A}=+\infty$
and $B_{A}=A$.
\end{rem}
\begin{proof}[Proof of Lemma \ref{lem: gen contient}]
Write $h=\mathfrak{f}\left[g\right]$, for some $g\in\mathcal{C}_{\Sigma}^{\mathcal{M}}\left(X\times\mathbb{R}\right)_{\text{int}}$ 
and apply the parametric power-constructible Preparation Theorem \cite[Proposition 4.7]{cmrs:mellin_constructible}
to $g$: this yields a cell decomposition of $X\times\mathbb{R}$
and by linearity of the integral we may concentrate on a cell $A$
which is open over $X$. 

Using Notation \ref{notation: cell bijection B_A}, if $\tau_{A}=-1$
then the set $\left\{ \left|y-\theta_{A}\left(x\right)\right|:\ \left(x,y\right)\in A\right\} $
is contained in $\left(0,1\right)$, so that $\left(x,y\right)\longmapsto\text{e}^{\text{i}\left(y-\theta_{A}\left(x\right)\right)}$
is a complex-valued subanalytic function. Hence, in this case we may
write
\[
\int_{A_{x}}g\restriction A\left(s,x,y\right)\text{e}^{\text{i}y}\text{d}y=\text{e}^{\text{i}\theta_{A}\left(x\right)}\int_{A_{x}}g\restriction A\left(s,x,y\right)\text{e}^{\text{i}\left(y-\theta_{A}\left(x\right)\right)}\text{d}y.
\]
As the integrand on the right hand side is a parametric power-constructible
function, by \cite[Theorem 2.16 and Remark 6.7]{cmrs:mellin_constructible}
there are a closed discrete set $P\subseteq\mathbb{C}$ (containing
the poles of $g$) and a parametric power-constructible function $G\in\mathcal{C}_{\Sigma}^{\mathcal{M}}\left(X\right)$
without poles outside $P$ such that $\mathfrak{f}\left[g\restriction A\right]=\text{e}^{\text{i}\theta_{A}}G$.

If $\tau_{A}=1$ then we apply the change of variables $\Pi_{A}$
under the sign of integral and, using \cite[Proposition 4.7]{cmrs:mellin_constructible},
we write $g\circ\Pi_{A}$ as a finite sum of prepared generators as
in \cite[Equation (4.8)]{cmrs:mellin_constructible}:
\begin{align*}
\int_{A_{x}}g\restriction A\left(s,x,y\right)\text{e}^{\text{i}y}\text{d}y & =\int_{a_{A}\left(x\right)}^{b_{A}\left(x\right)}g\circ\Pi_{A}\left(s,x,y\right)\frac{\partial\Pi_{A}}{\partial y}\left(x,y\right)\text{e}^{\text{i}\left(\sigma_{A}y+\theta_{A}\left(x\right)\right)}\text{d}y\\
 & =\text{e}^{\text{i}\theta_{A}\left(x\right)}\int_{a_{A}\left(x\right)}^{b_{A}\left(x\right)}\sum_{i}G_{i}\left(s,x\right)y^{\lambda_{i}\left(s\right)}\left(\log y\right)^{\mu_{i}}\Phi_{i}\left(s,x,y\right)\text{e}^{\text{i}\sigma_{A}y}\text{d}y\\
 & =\sum_{i}\text{e}^{\text{i}\theta_{A}\left(x\right)}G_{i}\left(s,x\right)\gamma_{i}\left(s,x\right).
\end{align*}
Summing up, we have written $h$ as a finite sum of generators without
poles outside some closed discrete set $P\subseteq\mathbb{C}$.
\end{proof}
Thus, from now on we will refer to the functions of the form \eqref{eq:generator}
as \emph{generators} of $\mathcal{C}_{\Sigma}^{\mathcal{M},\mathcal{F}}\left(X\right)$.

\subsection{Proof of Proposition \ref{prop: C^M,F has extension}}\label{subsec:proof2.20} 

Let $\Sigma,\Sigma'\subseteq\mathbb{C}$ be strips such that $\Sigma\subseteq\Sigma'$
and $h\in\mathcal{C}_{\Sigma}^{\mathcal{M},\mathcal{F}}\left(X\right)$
without poles outside some closed discrete set $P\subseteq\mathbb{C}$.
Write $h$ as a finite sum of generators of the form \eqref{eq:generator}, which is possible by Lemma \ref{lem: gen contient}. 
The problem is that the integrands of the transcendental elements
appearing in the generators are integrable on $\left(\Sigma\setminus P\right)\times X$
but might not be on the whole $\left(\Sigma'\setminus P\right)\times X$
(the only issue here is the integrability of a power of $y$ at $+\infty$,
as the parametric strong functions are bounded on the whole plane
$\mathbb{C}$). In this case, we need to rewrite the transcendental
elements as sums of generators on $\left(\Sigma'\setminus P\right)\times X$.
With this in mind, we may suppose that $h$ itself is a transcendental
element with unbounded $y$-fibers of the form
\[
h\left(s,x\right)=\int_{a\left(x\right)}^{+\infty}y^{\lambda\left(s\right)}\left(\log y\right)^{\mu}\Phi\left(s,x,y\right)\text{e}^{\sigma\text{i}y}\text{d}y,
\]
with $a\in\mathcal{S}\left(X\right)$ such that for all $x\in X,\ a\left(x\right)\geq1$,
and that the set $S_{0}=\{s\in\Sigma':\ \Re\left(\lambda\left(s\right)\right)<-1\}$
is a proper subset of $\Sigma'$. It follows that the above integral
is not finite on $\left(\Sigma'\setminus S_{0}\right)\times X$. Using
the strong convergence of the series defining $\Phi$ and \eqref{eq:eq: nested prep form unb},
we may rewrite, for some $k_{0}\geq0$,
\begin{align*}
h\left(s,x\right) & =\sum_{k\geq k_{0}}\xi_{k}^{c}\left(s,x\right)\left(a\left(x\right)\right)^{\frac{k}{d}}\int_{a\left(x\right)}^{+\infty}y^{\frac{\ell s+\eta-k}{d}}\left(\log y\right)^{\mu}\text{e}^{\sigma\text{i}y}\text{d}y\\
 & =\sum_{k\geq k_{0}}g_{k}\left(s,x\right)\int_{a\left(x\right)}^{+\infty}y^{\lambda_{k}\left(s\right)}\left(\log y\right)^{\mu}\text{e}^{\sigma\text{i}y}\text{d}y
\end{align*}
As the real part of the exponent $\lambda_{k}\left(s\right)$ decreases
as $k$ increases and as $\Sigma'$ has bounded width, there are only
finitely many power-log monomials which are not integrable for all
$s\in\Sigma'$. Let us concentrate on one such critical power-log
monomial and use integration by parts (where we integrate the exponential
and derive the power-log monomial):
\begin{align*}
\forall\left(s,x\right) & \in S_{0}\times X,\ \int_{a\left(x\right)}^{+\infty}y^{\lambda_{k}\left(s\right)}\left(\log y\right)^{\mu}\text{e}^{\sigma\text{i}y}\text{d}y\\
 & =\sigma\text{i}\text{e}^{\sigma\text{i}a\left(x\right)}\left(a\left(x\right)\right)^{\lambda_{k}\left(s\right)}\left(\log\left(a\left(x\right)\right)\right)^{\mu}-\int_{a\left(x\right)}^{+\infty}y^{\lambda_{k}\left(s\right)-1}\left(\log y\right)^{\mu-1}\left(\lambda_{k}\left(s\right)\log y+\mu\right)\text{e}^{\sigma\text{i}y}\text{d}y.
\end{align*}
Note that the right-hand side of the above equality is actually defined
on a strictly larger set than $S_{0}\times X$, namely on the set
$S_{1}\times X$, where $S_{1}=\left\{ s\in\Sigma':\ \Re\left(\lambda_{k}\left(s\right)\right)<0\right\} $.
Since $\Sigma'$ has bounded width, there is an integer $N_{k}\in\mathbb{N}$
such that $\Sigma'=\left\{ s\in\Sigma':\ \Re\left(\lambda_{k}\left(s\right)\right)<N_{k}-1\right\} $
and if we repeat the above procedure $N_{k_{0}}$ times for each critical
monomial, then we rewrite $h$ as a sum of generators such that the
transcendental elements are well defined on the whole $\left(\Sigma'\setminus P\right)\times X$.

\hfill{}$\qed$

\section{\label{sec: Strategy}Strongly integrable and monomial generators}

The aim of this section is to prove Theorems \ref{thm: stability of C^F}
and \ref{thm: Stability of C^M,F}, assuming a central result, Theorem \ref{thm: interpolation}, the proof of which requires extensive work carried out in the next two sections. We start by dealing with integrals of some specific functions in our class $\mathcal{C}^{\mathcal{M},\mathcal{F}}$.

\subsection{Special generators of $\mathcal{C}^{\mathcal{M},\mathcal{F}}$}\label{subsec.4.1}
We lay the foundations for the proofs of Theorems \ref{thm: stability of C^F}
and \ref{thm: Stability of C^M,F} by treating some special cases to which we will reduce
later (in Sections \ref{sec: Preparation} and \ref{sec: stability}). More precisely, we identify two special types of generators for $\mathcal{C}^{\mathcal{M},\mathcal{F}}$, strongly integrable generators and monomial generators, for which we show that their parametric integrals lie in $\mathcal{C}^{\mathcal{M},\mathcal{F}}$. In Section \ref{sec: stability}, Proposition \ref{prop: Splitting} will provide a reduction to such special generators. 

To illustrate the main ideas of this section, we start with two examples of explicit integration of very simple
generators.  
\begin{examples}
\label{example 1}Let $\Sigma=\left\{ s:\ -2<\Re\left(s\right)<1\right\} $
and $B=\left\{ \left(x,y\right):\ x\in X,\ y>a\left(x\right)\right\} $,
for some analytic $a\in\mathcal{S}\left(X\right)$ such that for all
$x\in X,\ a\left(x\right)\geq1$.
\begin{enumerate}
\item Let $D=\left\{ \left(x,y,t\right):\ \left(x,y\right)\in B,\ t>\widetilde{a}\left(x,y\right)\right\} $,
for some analytic $\widetilde{a}\in\mathcal{S}\left(B\right)$ such
that for all $\left(x,y\right)\in B,\ \widetilde{a}\left(x,y\right)\geq1$,
and $\Phi\in\mathcal{A}_{\Sigma}\left(D\right)$ be a parametric strong
function without poles outside some closed discrete set $P\subseteq\mathbb{C}$.
If
\[
g\left(s,x,y,t\right)=y^{-s-3}t^{s-2}\Phi\left(s,x,y,t\right)\chi_{D}\left(x,y,t\right),
\]
then for all $\left(s,x,y\right)\in\left(\Sigma\setminus P\right)\times X\times\mathbb{R},\ t\longmapsto g\left(s,x,y,t\right)\in L^{1}\left(\mathbb{R}\right)$,
so $h=\mathfrak{f}\left[g\right]\in\mathcal{C}_{\Sigma}^{\mathcal{M},\mathcal{F}}\left(X\times\mathbb{R}\right)$
is well defined on $\left(\Sigma\setminus P\right)\times X\times\mathbb{R}$.\\
We claim that $g\in\mathcal{C}_{\Sigma}^{\mathcal{M}}\left(X\times\mathbb{R}^{2}\right)_{\text{int}}$
and that there exist a closed discrete set $P'\supseteq P$ and a
function $H\in\mathcal{C}_{\Sigma}^{\mathcal{M},\mathcal{F}}\left(X\right)$
without poles outside $P'$ such that
\[
\forall\left(s,x\right)\in\left(\Sigma\setminus P'\right)\times X,\ H\left(s,x\right)=\int_{\mathbb{R}}h\left(s,x,y\right)\text{d}y.
\]
To see this, note that, since $\Phi$ is bounded, there is a constant
$C>0$ such that
\[
\forall\left(s,x,y\right)\in\left(\Sigma\setminus P\right)\times B,\ \int_{1}^{+\infty}t^{s-2}\left|\Phi\left(s,x,y,t\right)\right|\text{d}t<C.
\]
It follows that
\[
\forall\left(s,x,y\right)\in\left(\Sigma\setminus P\right)\times B,\ \mathfrak{f}\left[\left|g\right|\right]\leq y^{-s-3}\int_{1}^{+\infty}t^{s-2}\left|\Phi\left(s,x,y,t\right)\right|\text{d}t\leq Cy^{-s-3},
\]
so $y\longmapsto\mathfrak{f}\left[\left|g\right|\right]\left(s,x,y\right)\in L^{1}\left(\mathbb{R}\right)$
and by Tonelli's Theorem, $g\in\mathcal{C}_{\Sigma}^{\mathcal{M}}\left(X\times\mathbb{R}^{2}\right)_{\text{int}}$.
Hence, by Fubini's Theorem
\[
\int_{\mathbb{R}}h\left(s,x,y\right)\text{d}y=\int_{\mathbb{R}}t^{s-2}\text{e}^{\text{i}t}\left[\int_{\mathbb{R}}y^{-s-3}\Phi\left(s,x,y,t\right)\chi_{D}\left(x,y,t\right)\text{d}y\right]\text{d}t
\]
and the integrand $\tilde{g}$ in the inner integral belongs to $\mathcal{C}_{\Sigma}^{\mathcal{M}}\left(D\right)$
and is integrable with respect to $y$. By Theorem \ref{thm: C_M stability},
there are a closed discrete set $P'\subseteq\mathbb{C}$ containing
$P$ and a function $G\in\mathcal{C}_{\Sigma}^{\mathcal{M}}\left(X\times\mathbb{R}\right)$
without poles outside $P'$ such that
\[
\forall\left(s,x\right)\in\left(\Sigma\setminus P'\right)\times X,\ G\left(s,x,t\right)=t^{s-2}\int_{\mathbb{R}}y^{-s-3}\Phi\left(s,x,y,t\right)\chi_{D}\left(x,y,t\right)\text{d}y.
\]
Notice also that $G\in\mathcal{C}_{\Sigma}^{\mathcal{M}}\left(X\times\mathbb{R}\right)_{\text{int}}$
and that $H=\mathfrak{f}\left[G\right]$ proves the claim.
\item Consider $g\left(s,x,y\right)=y^{s}\chi_{B}\left(x,y\right)\in\mathcal{C}_{\Sigma}^{\mathcal{M}}\left(X\times\mathbb{R}\right)$
and $T\left(s,x,y\right)=g\left(s,x,y\right)\text{e}^{\text{i}y}\in\mathcal{C}_{\Sigma}^{\mathcal{M},\mathcal{F}}\left(X\times\mathbb{R}\right)$.
As
\[
\mathrm{Int}\left(g;\Sigma\times X\right)=\left\{ s\in\Sigma:\ \Re\left(s\right)<-1\right\} \times X\not=\Sigma\times X,
\]
we cannot apply the operator $\mathfrak{f}$ to $g$ in order to express
the integral of $T$ with respect to $y$. However, we claim that
there exists a function $H\in\mathcal{C}_{\Sigma}^{\mathcal{M},\mathcal{F}}\left(X\right)$
such that
\begin{equation}
\forall\left(s,x\right)\in\mathrm{Int}\left(T;\Sigma\times X\right),\ H\left(s,x\right)=\int_{a\left(x\right)}^{+\infty}y^{s}\mathrm{e}^{\mathrm{i}y}\mathrm{d}y.\label{eq:1}
\end{equation}
To show this, let us integrate by parts $y^{s}\mathrm{e}^{\mathrm{i}y}$
twice, where we integrate the exponential and derive the parametric
power:
\begin{align*}
\int y^{s}\mathrm{e}^{\mathrm{i}y}\mathrm{d}y & =y^{s}\frac{\mathrm{e}^{\mathrm{i}y}}{\mathrm{i}}-\frac{1}{\mathrm{i}}\int sy^{s-1}\mathrm{e}^{\mathrm{i}y}\mathrm{d}y\\
 & =\mathrm{i}y^{s}\mathrm{e}^{\mathrm{i}y}+sy^{s-1}\mathrm{e}^{\mathrm{i}y}-s\left(s-1\right)\int y^{s-2}\mathrm{e}^{\mathrm{i}y}\mathrm{d}y.
\end{align*}
Define
\[
H\left(s,x\right)=-\mathrm{i}\left(a\left(x\right)\right)^{s}\mathrm{e}^{\mathrm{i}a\left(x\right)}-s\left(a\left(x\right)\right)^{s-1}\mathrm{e}^{\mathrm{i}a\left(x\right)}-s\left(s-1\right)\int_{a\left(x\right)}^{+\infty}y^{s-2}\mathrm{e}^{\mathrm{i}y}\mathrm{d}y.
\]
Notice that $H$ is well defined on $\Sigma\times X$, because the
real part of the exponent in the integrand is always $<-1$ on $\Sigma$,
and that $H\in\mathcal{C}_{\Sigma}^{\mathcal{M},\mathcal{F}}\left(X\right)$,
because the last term is obtained by applying the operator $\mathfrak{f}$
to $\tilde{g}\left(s,x,y\right)=s\left(s-1\right)y^{s-2}\chi_{B}\left(s,y\right)\in\mathcal{C}_{\Sigma}^{\mathcal{M}}\left(X\times\mathbb{R}\right)_{\text{int}}$.
Note also that $H$ satisfies \eqref{eq:1}, since $\mathrm{Int}\left(T;\Sigma\times X\right)=\{s\in\Sigma:\ \Re\left(s\right)<-1\}\times X,$
and on this part of the space the exponents of the parametric powers
$y^{s},y^{s-1}$ have negative real part.
\end{enumerate}
\end{examples}
The techniques illustrated in these two examples can be generalized
and used to integrate generators of $\mathcal{C}_{\Sigma}^{\mathcal{M},\mathcal{F}}\left(X\times\mathbb{R}\right)$
of a rather simple form.

Recall that, by Lemma \ref{lem: gen contenu}, every generator of
$\mathcal{C}_{\Sigma}^{\mathcal{M},\mathcal{F}}\left(X\times\mathbb{R}\right)$
can be written as $\mathfrak{f}\left[h\right]$, for some $h\in\mathcal{C}_{\Sigma}^{\mathcal{M}}\left(\left(X\times\mathbb{R}\right)\times\mathbb{R}\right)_{\text{int}}$.
\begin{defn}
\label{def: superintegrable and monomial genrators}Let $T\in\mathcal{C}_{\Sigma}^{\mathcal{M},\mathcal{F}}\left(X\times\mathbb{R}\right)$
be a generator.
\begin{itemize}
\item $T$ is \emph{strongly integrable} if $T$ can be written as $\mathfrak{f}\left[h\right]$,
for some $h\in\mathcal{C}_{\Sigma}^{\mathcal{M}}\left(X\times\mathbb{R}^{2}\right)_{\text{int}}$.
If $B$ is a cell as in \eqref{eq:A_theta=00003DB}, we say that $T$
is strongly integrable \emph{on }$B$ if $T\chi_{B}$ is strongly
integrable.
\item $T$ is \emph{monomial} \emph{in (its last variable)} $y$ if $T$
has the form
\begin{equation}
T\left(s,x,y\right)=f\left(s,x\right)y^{\lambda\left(s\right)}\left(\log y\right)^{\mu}\text{e}^{\text{i}Q\left(x,y\right)},\label{eq: monomial generator}
\end{equation}
where $f\in\mathcal{C}_{\Sigma}^{\mathcal{M},\mathcal{F}}\left(X\right),\mu\in\mathbb{N},\lambda\left(s\right)=\frac{\ell s+\eta}{d}$
for some $\ell\in\mathbb{Z},\ \eta\in\mathbb{C},\ d\in\mathbb{N}\setminus\left\{ 0\right\} $
and $Q\in\mathcal{S}\left(X\right)\left[y^{\frac{1}{d}}\right]$ is
a polynomial in the variable $y^{\frac{1}{d}}$ with coefficients
subanalytic functions of $x$. The tuple $\left(d,\ell,\eta,\mu,Q\right)$
is called the \emph{monomial data }of $T$.
\end{itemize}
\end{defn}
\begin{rem}
\label{rem: superintegrable old fashion}Let $T=g\text{e}^{\text{i}\varphi}\gamma\in\mathcal{C}_{\Sigma}^{\mathcal{M},\mathcal{F}}\left(X\times\mathbb{R}\right)$
be a generator and write $\gamma$ as $\int_{\mathbb{R}}G\left(s,x,y,t\right)\text{e}^{\text{i}t}\text{d}t$
for some appropriate $G\in\mathcal{C}_{\Sigma}^{\mathcal{M}}\left(\left(X\times\mathbb{R}\right)\times\mathbb{R}\right)_{\text{int}}$.
Suppose that $gG\in\mathcal{C}_{\Sigma}^{\mathcal{M}}\left(X\times\mathbb{R}^{2}\right)_{\text{int}}$.
Then $T$ is strongly integrable. To see this, proceed as in Lemma
\ref{lem: gen contenu} and write $T$ as $\mathfrak{f}\left[h\right]$.
It is clear that $h\in\mathcal{C}_{\Sigma}^{\mathcal{M}}\left(X\times\mathbb{R}^{2}\right)_{\text{int}}$.
\end{rem}
Our next aim is to integrate a single generator which is of either
of the forms in Definition \ref{def: superintegrable and monomial genrators}.
\begin{prop}
\label{prop: integrability of naive integrable gen}Let $g\in\mathcal{C}_{\Sigma}^{\mathcal{M}}\left(X\times\mathbb{R}^{2}\right)_{\mathrm{int}}$
be without poles outside some closed discrete set $P\subseteq\mathbb{C}$
and $\varphi\in\mathcal{S}\left(X\times\mathbb{R}^{2}\right)$. There
exist a closed discrete set $P'\subseteq\mathbb{C}$ containing $P$
and a function $H\in\mathcal{C}_{\Sigma}^{\mathcal{M},\mathcal{F}}\left(X\right)$
without poles outside $P'$ such that
\[
\forall\left(s,x\right)\in\left(\Sigma\setminus P'\right)\times X,\ H\left(s,x\right)=\int_{\mathbb{R}^{2}}g\left(s,x,y,t\right)\mathrm{e}^{\mathrm{i}\varphi\left(x,y,t\right)}\mathrm{d}y\wedge\mathrm{d}t.
\]
Moreover, the set $P'\setminus P$ is contained in a finitely generated
$\mathbb{Z}$-lattice.
\end{prop}
\begin{proof}
Up to decomposing $X\times\mathbb{R}^{2}$ into subanalytic cells,
we may suppose that, on each cell $A$ of base $X$ and open over
$\mathbb{R}^{m}$, either $\varphi$ does not depend on $\left(y,t\right)$,
or for one of these two variables (say, $y$), the function $y\longmapsto\varphi\left(x,y,t\right)$
is $C^{1}$ and strictly monotonic.

In the first case we factor the exponential out of the integral and
we apply Theorem \ref{thm: C_M stability} to $g$. In the second
case, up to applying the subanalytic change of variables $\left(x,y,t\right)\longmapsto\left(x,\varphi\left(x,y,t\right),t\right)$
and multiplying by the Jacobian of its inverse, we may suppose that
$\varphi\left(x,y,t\right)=y$ on $A$. Hence, by Fubini's Theorem,
\[
\int_{A}g\left(s,x,y,t\right)\mathrm{e}^{\mathrm{i}\varphi\left(x,y,t\right)}\text{d}y\wedge\text{d}t=\int_{\mathbb{R}}\text{e}^{\text{i}y}\left[\int_{\mathbb{R}}\chi_{A}\left(x,y,t\right)g\left(s,x,y,t\right)\text{d}t\right]\text{d}y.
\]
Again by Theorem \ref{thm: C_M stability}, applied to the integrand
inside the square brackets, the right-hand side of the above equation
is of the form $\mathfrak{f}\left[\tilde{g}\right]$, for some suitable
$\tilde{g}\in\mathcal{C}_{\Sigma}^{\mathcal{M}}\left(X\times\mathbb{R}\right)$
without poles outside some closed discrete set $P'\supseteq P$. We
conclude by linearity of the integral, taking the sum over the cells
of the decomposition.
\end{proof}
\begin{cor}
\label{cor: algebra}$\mathcal{C}_{\Sigma}^{\mathcal{M},\mathcal{F}}\left(X\right)$
is a $\mathbb{C}$-algebra.
\end{cor}
\begin{proof}
Is suffices to show that if $h_{1},h_{2}\in\mathcal{C}_{\Sigma}^{\mathcal{M},\mathcal{F}}\left(X\right)$
then $h_{1}\cdot h_{2}\in\mathcal{C}_{\Sigma}^{\mathcal{M},\mathcal{F}}\left(X\right)$.
Write $h_{i}=\mathfrak{f}\left[g_{i}\right]$, for some $g_{i}\in\mathcal{C}_{\Sigma}^{\mathcal{M}}\left(X\times\mathbb{R}\right)_{\text{int}}$.
By Fubini's Theorem, $\left(s,x,y,t\right)\longmapsto g_{1}\left(s,x,y\right)\cdot g_{2}\left(s,x,t\right)\in\mathcal{C}_{\Sigma}^{\mathcal{M}}\left(X\times\mathbb{R}^{2}\right)_{\text{int}}$
so Proposition \ref{prop: integrability of naive integrable gen}
applies.
\end{proof}
Proposition \ref{prop: integrability of naive integrable gen} allows
us to integrate strongly integrable generators.
\begin{cor}
\label{cor: integration of strongly integrable}Let $T\in\mathcal{C}_{\Sigma}^{\mathcal{M},\mathcal{F}}\left(X\times\mathbb{R}\right)$
be a strongly integrable generator without poles outside some closed
discrete set $P\subseteq\mathbb{C}$. There exist a closed discrete
set $P'\subseteq\mathbb{C}$ containing $P$, such that $P'\setminus P$
is contained in a finitely generated $\mathbb{Z}$-lattice, and a
function $H\in\mathcal{C}_{\Sigma}^{\mathcal{M},\mathcal{F}}\left(X\right)$
without poles outside $P'$ such that
\[
\forall\left(s,x\right)\in\left(\Sigma\setminus P'\right)\times X,\ H\left(s,x\right)=\int_{\mathbb{R}}T\left(s,x,y\right)\text{d}y.
\]
\end{cor}
\begin{proof}
Immediate from Fubini's Theorem and Proposition \ref{prop: integrability of naive integrable gen}.
\end{proof}
Next, we consider a monomial generator and interpolate its integral
on a given cell by a function of the class $\mathcal{C}_{\Sigma}^{\mathcal{M},\mathcal{F}}$.
\begin{lem}
\label{lem: monomial gen on bounded cells}Let $B$ be a cell as in
\eqref{eq:A_theta=00003DB} with bounded $y$-fibers and let $T\in\mathcal{C}_{\Sigma}^{\mathcal{M},\mathcal{F}}\left(X\times\mathbb{R}\right)$
be a generator which is monomial in $y$ (as in \eqref{eq: monomial generator}),
without poles outside some closed discrete set $P\subseteq\mathbb{C}$.
Then $T$ is strongly integrable on $B$.
\end{lem}
\begin{proof}
For all $\left(s,x\right)\in\left(\Sigma\setminus P\right)\times X,\ y\longmapsto\left|T\left(s,x,y\right)\chi_{B}\left(x,y\right)\right|$
extends to a continuous function on $\left[a\left(x\right),b\left(x\right)\right]$.
Hence, by Remark \ref{rem: superintegrable old fashion} we are done.
\end{proof}
\begin{prop}
\label{prop: integration of monomial generators}Let $B$ be a cell
as in \eqref{eq:A_theta=00003DB} with unbounded $y$-fibers and let
$T\in\mathcal{C}_{\Sigma}^{\mathcal{M},\mathcal{F}}\left(X\times\mathbb{R}\right)$
be a generator which is monomial in $y$ (as in \eqref{eq: monomial generator}),
without poles outside some closed discrete set $P\subseteq\mathbb{C}$.
Then
\[
\mathrm{Int}\left(T\chi_{B};\left(\Sigma\setminus P\right)\times X\right)=\left\{ \left(s,x\right):\ \Re\left(\lambda\left(s\right)\right)<-1\vee\left(\Re\left(\lambda\left(s\right)\right)\geq-1\wedge f\left(s,x\right)=0\right)\right\} .
\]
Moreover there are a closed discrete set $P'\subseteq\mathbb{C}$
containing $P$, such that $P'\setminus P$ is contained in a finitely
generated $\mathbb{Z}$-lattice, and a function $H\in\mathcal{C}_{\Sigma}^{\mathcal{M},\mathcal{F}}\left(X\right)$
without poles outside $P'$ such that
\[
\forall\left(s,x\right)\in\mathrm{Int}\left(T\chi_{B};\left(\Sigma\setminus P'\right)\times X\right),\ H\left(s,x\right)=\int_{a\left(x\right)}^{+\infty}T\left(s,x,y\right)\text{d}y.
\]
\end{prop}
\begin{proof}
The statement on the integration locus is immediate.

Write $Q\left(x,y\right)=\sum_{i\leq n}b_{i}\left(x\right)y^{\frac{i}{d}}$.
We may suppose that $n>0$, because otherwise we are done by Theorem
\ref{thm: C_M stability}. By o-minimality, we may suppose that for
all $x\in X,\ b_{n}\left(x\right)\not=0$. By Lemma \ref{lem: monomial gen on bounded cells}
and definable choice, we may suppose that for all $x\in X,\ y\longmapsto Q\left(x,y\right)$
is monotonic (say, strictly increasing) on $\left(a\left(x\right),+\infty\right)$.
Hence we may write $Q\left(x,y\right)=b_{n}\left(x\right)y^{\frac{n}{d}}\left(1+\varepsilon_{1}\left(x,y\right)\right)$,
where $\varepsilon_{1}\in\mathbb{R}\left\{ x\right\} \left[y^{-\frac{1}{d}}\right]$
with $\varepsilon_{1}\left(x,0\right)=0$, and the compositional inverse
has the form $\phi\left(x,z\right)=c\left(x\right)z^{\frac{d}{n}}\left(1+\varepsilon_{2}\left(x,z\right)\right)$,
for some analytic $c\in\mathcal{S}\left(X\right)$ and $\varepsilon_{2}\in\mathbb{R}\left\{ x,z^{-\frac{1}{n}}\right\} $
with $\varepsilon_{2}\left(x,0\right)=0$. Note that ${\displaystyle \frac{\partial\phi}{\partial z}\left(x,z\right)=\frac{d}{n}c\left(x\right)z^{\frac{d}{n}-1}\left(1+\varepsilon_{3}\left(x,z\right)\right)}$,
for some $\varepsilon_{3}\in\mathbb{R}\left\{ x,z^{-\frac{1}{n}}\right\} $
with $\varepsilon_{3}\left(x,0\right)=0$. Hence, on $\mathrm{Int}\left(T\chi_{B};\left(\Sigma\setminus P\right)\times X\right)$
we may write
\[
\int_{a\left(x\right)}^{+\infty}T\left(s,x,y\right)\text{d}y=\tilde{f}\left(s,x\right)\int_{Q\left(x,a\left(x\right)\right)}^{+\infty}z^{\tilde{\lambda}\left(s\right)}\left[\frac{d}{n}\log z+G\left(x,z\right)\right]^{\mu}u\left(s,x,z\right)\text{e}^{\text{i}z}\text{d}z,
\]
where $\tilde{f}\left(s,x\right)=f\left(s,x\right)\frac{d}{n}\left(c\left(x\right)\right)^{\lambda\left(s\right)+1},\ \tilde{\lambda}\left(s\right)=\frac{d}{n}\lambda\left(s\right)+\frac{d}{n}-1,\ G\left(x,z\right)=\log\left(c\left(x\right)\right)+\log\left(1+\varepsilon_{2}\left(x,z\right)\right)$
and $u\left(s,x,z\right)=\left(1+\varepsilon_{2}\left(x,z\right)\right)^{\lambda\left(s\right)}\left(1+\varepsilon_{3}\left(x,z\right)\right)$
is a parametric strong unit. Note that, for all $x\in X$, $G$ can
be expanded as a power series (with nonzero constant term) in the
variable $z^{-\frac{1}{n}}$. By expanding the power $\mu$ of the
square bracket, we may rewrite the above integral as a finite sum
of terms where the integrand has the form $z^{\tilde{\lambda}\left(s\right)}\text{e}^{\text{i}z}\cdot\left(\log z\right)^{\nu}U_{\nu}\left(s,x,z\right)$,
for some $\nu\leq\mu$ and parametric strong unit $U_{\nu}$ which
can be expanded as a series in the variable $z^{-\frac{1}{n}}$. It
follows that there are finitely many monomials in the integrand of
the form $z^{\Re\left(\tilde{\lambda}\left(s\right)\right)-\frac{k}{n}}\left(\log z\right)^{\nu}$
for which the integral is not finite. Argueing as in the proof of
Proposition \ref{prop: C^M,F has extension}, we find the function
$H$ in the statement by integration by parts.
\end{proof}

\subsection{\label{subsec: Proof strategy}Overview of the proofs of Theorems \ref{thm: stability of C^F}
and \ref{thm: Stability of C^M,F}}

The proof of Theorem \ref{thm: Stability of C^M,F}, which will be
completed in Section \ref{sec: stability}, is organized as follows:
we show that, given $h\in\mathcal{C}_{\Sigma}^{\mathcal{M},\mathcal{F}}\left(X\times\mathbb{R}\right)$
without poles outside some closed discrete set $P\subseteq\mathbb{C}$,
the domain $X\times\mathbb{R}$ can be partitioned into subanalytic
cells such that on each cell $A$, up to a subanalytic change of variables,
$h$ can be written as a finite sum of generators which are either
strongly integrable or monomial in the last variable (Proposition
\ref{prop: Splitting}). The results of the current section provide,
for each such generator $T_{i}$, a description of the integration
locus and a function $H_{i}\in\mathcal{C}_{\Sigma}^{\mathcal{M},\mathcal{F}}\left(X\right)$
without poles outside some closed discrete set $P'\supseteq P$, which
coincides with the integral of $T_{i}$ on its integration locus.
Next, we show that, up to possibly enlarging the closed discrete set
$P'$, the integration locus of $h\restriction A$ is the intersection
of the integration loci of the generators $T_{i}$ (this is done using
a non-compensation argument proven in \cite[Proposition 3.4]{cmrs:mellin_constructible}).
Thus the sum of the functions $H_{i}$ interpolates the integral of
$h\restriction A$ on its integration locus (Theorem \ref{thm: interpolation}).
Theorem \ref{thm: Stability of C^M,F} follows from Theorem \ref{thm: interpolation}
by Fubini's Theorem (with an argument spelled out in detail in \cite[pp. 31-32]{cmrs:mellin_constructible}), 
which shows that we can iterate the argument above integrating with
respect to one variable at the time.

The proof of Theorem \ref{thm: stability of C^F} is just a special
case of that of Theorem \ref{thm: Stability of C^M,F}, where all
the functions involved happen not to depend on the variable $s$.
In particular, that $\mathcal{C}^{\mathbb{C},\mathcal{F}}$ is a system
of $\mathbb{C}$-algebras containing $\text{e}^{\text{i}\mathcal{S}}$
follows from Corollary \ref{cor: algebra} and Lemma \ref{lem: gen contenu}.
It follows from stability under parametric integration and the definition
of $\mathcal{C}^{\mathbb{C},\mathcal{F}}$ that it is the smallest
system of $\mathbb{C}$-algebras containing $\mathcal{S}_{+}^{\mathbb{C}}\cup\mathrm{e}^{\mathrm{i}\mathcal{S}}$
and stable under parametric integration, and the smallest such system
containing $\mathcal{C}^{\mathbb{C}}$ and stable under the parametric
Fourier transform and right-composition with subanalytic maps.

\section{\label{sec: Preparation}Preparation}

With the aim of proving Proposition \ref{prop: Splitting}, in this
section we write every function of the class $\mathcal{C}_{\Sigma}^{\mathcal{M},\mathcal{F}}\left(X\times\mathbb{R}\right)$, piecewise and up to a subanalytic change of variables, as a finite
sum of generators of a special \emph{prepared form} which gives some
information on their integration locus. This builds further (and relies) on preparation results for subanalytic functions from \cite{paru:lip1,parusinski:preparation,lr:int,miller_dan:preparation_theorem_weierstrass_systems}
\begin{notation}
\label{notation: t-cell}Let $B\subseteq X\times\mathbb{R}$ be as
in \eqref{eq:A_theta=00003DB} Consider a cell $D\subseteq B\times\mathbb{R}$
of the form
\begin{equation}
D=\left\{ \left(x,y,t\right):\ \left(x,y\right)\in B,\ \widetilde{a}\left(x,y\right)<t<\widetilde{b}\left(x,y\right)\right\} ,\label{eq: D}
\end{equation}
where $\widetilde{a},\widetilde{b}:B\longrightarrow\mathbb{R}$ are
analytic subanalytic functions with $1\leq\widetilde{a}\left(x,y\right)<\widetilde{b}\left(x,y\right)\ \text{for all }\left(x,y\right)\in B$,
and $\widetilde{b}$ is allowed to be $\equiv+\infty$. We say that
$D$ has\emph{ bounded $t$-fibers} if $\widetilde{b}<+\infty$ and
\emph{unbounded $t$-fibers} if $\widetilde{b}\equiv+\infty$.

Suppose furthermore that $\widetilde{a},\widetilde{b},\widetilde{b}-\widetilde{a}$
have the following prepared form:
\begin{equation}
\begin{aligned} & \widetilde{a}\left(x,y\right)=a_{0}\left(x\right)y^{\frac{\alpha}{d}}u_{a}\left(x,y\right), & \widetilde{b}\left(x,y\right)=b_{0}\left(x\right)y^{\frac{\beta}{d}}u_{b}\left(x,y\right),\\
 & \widetilde{b}\left(x,y\right)-\widetilde{a}\left(x,y\right)=d_{0}\left(x\right)y^{\frac{\Delta}{d}}u_{d}\left(x,y\right),
\end{aligned}
\label{eq: prepared t-bounds}
\end{equation}
where $\alpha,\beta,\Delta\in\mathbb{N},\ a_{0},b_{0},d_{0}\in\mathcal{S}\left(X\right)$
are analytic and $u_{a},u_{b},u_{d}\in\mathcal{S}\left(B\right)$
are $\psi$-prepared subanalytic strong units (for $\psi$ as in \eqref{eq: psi}).
If $\widetilde{b}=+\infty$ we stipulate that $b_{0}=d_{0}=+\infty,\ u_{b}=u_{d}=1$
and $\beta=\Delta=0$.

Define
\begin{equation}
\varPsi\left(x,y,t\right)=\left(\psi\left(x,y\right),\left(\frac{a_{0}\left(x\right)y^{\frac{\alpha}{d}}}{t}\right)^{\frac{1}{d}},\left(\frac{t}{b_{0}\left(x\right)y^{\frac{\beta}{d}}}\right)^{\frac{1}{d}}\right),\label{eq: big psi}
\end{equation}
where if $D$ has unbounded $t$-fibers we omit the last component.
Note that $\varPsi$ is a $1$-bounded subanalytic map on $D$.
\end{notation}
\begin{defn}
\label{def: prepared generator}A generator $T\in\mathcal{C}_{\Sigma}^{\mathcal{M},\mathcal{F}}\left(X\times\mathbb{R}\right)$
without poles outside some closed discrete set $P\subseteq\mathbb{C}$
is \emph{prepared on }$B$ if for all $\left(s,x,y\right)\in\left(\Sigma\setminus P\right)\times B,$
\[
T\left(s,x,y\right)=g\left(s,x\right)y^{\lambda\left(s\right)}\left(\log y\right)^{\mu}\text{e}^{\text{i}\varphi\left(x,y\right)}\gamma\left(s,x,y\right),
\]
where $g\in\mathcal{C}_{\Sigma}^{\mathcal{M}}\left(X\right),\ \mu\in\mathbb{N},\ \varphi\in\mathcal{S}\left(B\right)$
is prepared as in \eqref{eq: prepared subanalytic} with respect to
$\psi$ as in \eqref{eq: psi}, $\lambda\left(s\right)={\displaystyle \frac{\ell s+\eta}{d}}$,
for some $\ell\in\mathbb{Z},\eta\in\mathbb{C}$ and the same $d$
appearing in \eqref{eq: psi}, and the transcendental element $\gamma\in\Gamma_{\Sigma}\left(B\right)$
has the form
\[
\gamma\left(s,x,y\right)=\int_{\mathbb{R}}\chi_{D}\left(x,y,t\right)t^{\varrho\left(s\right)}\left(\log t\right)^{\nu}\Phi\left(s,x,y,t\right)\text{e}^{\sigma\text{i}t}\text{d}t,
\]
where $D$ is as in Notation \ref{notation: t-cell}, $\sigma\in\left\{ +,-\right\} ,\ \nu\in\mathbb{N},\ \varrho\left(s\right)={\displaystyle \frac{\widetilde{\ell}s+\widetilde{\eta}}{d}}$
for some $\widetilde{\ell}\in\mathbb{Z},\widetilde{\text{\ensuremath{\eta}}}\in\mathbb{C}$
and the same $d$ appearing in \eqref{eq: psi} and $\Phi$ is a $\varPsi$-prepared
parametric strong function (with $\varPsi$ as in \eqref{eq: big psi}).
\end{defn}
Recall Notation \ref{notation: cell bijection B_A}.
\begin{prop}[Preparation]
\label{prop: preparation, 1st version} Let $h\in\mathcal{C}_{\Sigma}^{\mathcal{M},\mathcal{F}}\left(X\times\mathbb{R}\right)$.
There is a cell decomposition of $\mathbb{R}^{m+1}$ compatible with
$X$ such that for each cell $A$ that is open over $\mathbb{R}^{m}$
(which we may suppose to be of the form \eqref{eq:A}), $h\circ\Pi_{A}$
is a finite sum of prepared generators on $B_{A}$.
\end{prop}
\begin{proof}
Write $h$ as a sum of generators of the form \eqref{eq:generator}
and apply \cite[Proposition 4.7]{cmrs:mellin_constructible} simultaneously
to all the parametric power-constructible data appearing in the generators:
this produces a cell decomposition of $X\times\mathbb{R}^{2}$ compatible
with $X$, and on each cell $D$ which is open over $X\times\mathbb{R}$,
a prepared form of the data with respect to the last variable $t$,
where the coefficient functions are parametric power-constructible
functions depending on the variables $\left(x,y\right)\in X\times\mathbb{R}$.
Now apply \cite[Proposition 4.7]{cmrs:mellin_constructible} again
simultaneously to all the coefficient functions, in order to prepare
them with respect to the variable $y$ on suitable cells $B\subseteq X\times\mathbb{R}$,
thus refining the cell decomposition. This gives the wanted result,
up to trivial manipulations to adjust the definition of $\psi$ and
$\varPsi$ (see \cite[pp. 1268-70]{ccmrs:integration-oscillatory}
for the details).
\end{proof}
\begin{rem}
\label{rem: d}The proof of \cite[Proposition 4.7]{cmrs:mellin_constructible}
(and indeed that of all preparation results based on the Subanalytic
Preparation Theorem in \cite{lr:prep}) shows that it is possible
to choose the same integer $d$ appearing in Definition \ref{def: prepared generator}
for all prepared generators on all cells. Thus $d$ is a data of the
cell decomposition and not of a single prepared generator on a single
cell. We will hence call a $d$\emph{-cell decomposition }a cell-decomposition
with data $d\in\mathbb{N}\setminus\left\{ 0\right\} $ and we will
say that a generator is $d$\emph{-prepared} on one of the cells of
the composition. Similar easy manipulations show that if a generator
$T$ is $d$-prepared on a cell of a $d$-cell decomposition, then
$T$ is also $d^{2}$-prepared and the decomposition can also be considered
as a $d^{2}$-cell decomposition.
\end{rem}
Our aim is to refine the previous preparation statement so as to write
$h\circ\Pi_{A}$ as a finite sum of generators which are either strongly
integrable or monomial in $y$.

The first remark is that if we consider a cell $A$ such that $B_{A}$
has bounded $y$-fibers, then the prepared generators which appear
in $h\circ\Pi_{A}$ are strongly integrable.
\begin{prop}
\label{prop: strongly int on bounded cells}Suppose $B$ as in \eqref{eq:A_theta=00003DB}
has bounded $y$-fibers and let $T\in\mathcal{C}_{\Sigma}^{\mathcal{M},\mathcal{F}}\left(X\times\mathbb{R}\right)$
be a generator which is prepared on $B.$ Then $T$ is strongly integrable.
\end{prop}
\begin{proof}
Let
\[
G\left(s,x,y,t\right)=g\left(s,x\right)y^{\lambda\left(s\right)}\left(\log y\right)^{\mu}\chi_{D}\left(x,y,t\right)t^{\varrho\left(s\right)}\left(\log t\right)^{\nu}\Phi\left(s,x,y,t\right)\in\mathcal{C}_{\Sigma}^{\mathcal{M}}\left(X\times\mathbb{R}^{2}\right),
\]
so that
\begin{equation}
T\left(s,x,y\right)=\int_{\mathbb{R}}G\left(s,x,y,t\right)\text{e}^{\text{i}\left(\varphi\left(x,y\right)+\sigma t\right)}\text{d}t.\label{eq: prepared gen as integral of G}
\end{equation}
Since $\Phi$ is bounded and extends continuously to $\overline{D}$,
for all $\left(s,x\right)\in\left(\Sigma\setminus P\right)\times X,$
the function $y\longmapsto\int_{\mathbb{R}}\left|G\right|\text{d}t$
extends continuously to the closed and bounded interval $\left[a\left(x\right),b\left(x\right)\right]$
and is hence integrable on this interval. By Tonelli's Theorem, for
all $\left(s,x\right)\in\left(\Sigma\setminus P\right)\times X,\ \left(y,t\right)\longmapsto\left|G\left(s,x,y,t\right)\right|\in L^{1}\left(\mathbb{R}^{2}\right)$,
so that $G\in\mathcal{C}_{\Sigma}^{\mathcal{M}}\left(X\times\mathbb{R}^{2}\right)_{\text{int}}$
and, by Remark \ref{rem: superintegrable old fashion}, $T$ is strongly
integrable.
\end{proof}
Next, we refine the statement of Proposition \ref{prop: preparation, 1st version}
for the subclass of $\mathcal{C}_{\Sigma}^{\mathcal{M},\mathcal{F}}\left(X\times\mathbb{R}\right)$
of those functions which are \emph{naive in the last variable $y$}
(see Definition \ref{def: naive in y} below). On cells with unbounded
$y$-fibers, the functions in this subclass have easily readable asymptotics
in $y$ (see Section \ref{sec:Asymptotic-expansions-and}) and can
be written as finite sums of generators which are either strongly
integrable or monomial in $y$.
\begin{defn}
\label{def: naive in y} Let $B\subseteq X\times\mathbb{R}$ be a
subanalytic cell which is open over $X$. A generator $T\in\mathcal{C}_{\Sigma}^{\mathcal{M},\mathcal{F}}\left(B\right)$
as in \eqref{eq:generator} is \emph{naive in }$y$ if the transcendental
element $\gamma$ does not depend on $y$. Hence
\[
T=\gamma g\text{e}^{\text{i}\varphi}
\]
with $\gamma\in\mathcal{C}_{\Sigma}^{\mathcal{M},\mathcal{F}}\left(X\right),\ g\in\mathcal{C}_{\Sigma}^{\mathcal{M}}\left(B\right)$
and $\varphi\in\mathcal{S}\left(B\right)$.
\end{defn}
\begin{prop}
\label{prop: naive in y are sums of mon + ssint}Let $h\in\mathcal{C}_{\Sigma}^{\mathcal{M},\mathcal{F}}\left(X\times\mathbb{R}\right)$
be a finite sum of generators which are naive in $y$. Then Proposition
\ref{prop: preparation, 1st version} holds for $h$ with the additional
property that the prepared generators on $B_{A}$ are either monomial
in $y$ or strongly integrable.
\end{prop}
\begin{proof}
Apply Proposition \ref{prop: preparation, 1st version} to $h$: the
proof shows that this produces a cell decomposition and, for each
cell $A$, a presentation of $h\circ\Pi_{A}$ as a finite sum of prepared
generators which are themselves naive in $y$. By Proposition \ref{prop: strongly int on bounded cells},
on cells with bounded $y$-fibers the generators are strongly integrable.
Hence we may concentrate on a cell $B=B_{A}$ of the form \eqref{eq:A_theta=00003DB}
with unbounded $y$-fibers and on a generator $T$ which is prepared
on $B$ and naive in $y$. Thus $T$ has the form
\begin{equation}
T\left(s,x,y\right)=f\left(s,x\right)y^{\lambda\left(s\right)}\left(\log y\right)^{\mu}\text{e}^{\text{i}\varphi\left(x,y\right)}\Phi\left(s,x,y\right),\label{eq: naive in y in 5.7}
\end{equation}
where $f\in\mathcal{C}_{\Sigma}^{\mathcal{M},\mathcal{F}}\left(X\right),\ \lambda\left(s\right),\mu$
are as in Definition \ref{def: prepared generator} and the prepared
forms (with respect to $\psi$ as in \eqref{eq: psi}) of $\varphi\in\mathcal{S}\left(B\right)$
and $\Phi\in\mathcal{A}_{\Sigma}\left(B\right)$ are as in \eqref{eq: prepared subanalytic}
and \eqref{eq:eq: nested prep form unb}, respectively. Up to partitioning
$X$ into subanalytic cells, we may suppose that $\left|\varphi_{0}\right|$
is either bounded from above or bounded away from zero. If $\left|\varphi_{0}\right|$
is bounded and $\omega<0$ then $\text{e}^{\text{i}\varphi}$ is a
complex-valued $\psi$-prepared subanalytic strong function. If $\omega\geq0$
then write
\[
\varphi\left(x,y\right)=Q\left(x,y\right)+\varphi_{>}\left(x,y\right),
\]
where $Q\in\mathcal{S}\left(X\right)\left[y^{\frac{1}{d}}\right]$
is a polynomial in the variable $y^{\frac{1}{d}}$ with coefficients
subanalytic functions of $x$ and
\[
\varphi_{>}\left(x,y\right)=\varphi_{0}\left(x\right)\left(a\left(x\right)\right)^{\frac{\omega}{d}}\left(\frac{a\left(x\right)}{y}\right)^{\frac{1}{d}}\sum_{k\geq0}b_{k+\omega+1}\left(x\right)\left(\frac{a\left(x\right)}{y}\right)^{\frac{k}{d}}.
\]
If $\left|\varphi_{0}\right|$ is bounded from above then clearly
$\text{e}^{\text{i}\varphi_{>}}$ is a complex-valued $\psi$-prepared
subanalytic strong function. If $\left|\varphi_{0}\right|$ is bounded
away from zero, then we write
\[
\varphi_{>}\left(x,y\right)=\left(a\left(x\right)\right)^{\frac{\omega}{d}}\left(\frac{a\left(x\right)\varphi_{0}\left(x\right)^{d}}{y}\right)^{\frac{1}{d}}\sum_{k\geq0}b_{k+\omega+1}\left(x\right)\varphi_{0}\left(x\right)^{-k}\left(\frac{a\left(x\right)\varphi_{0}\left(x\right)^{d}}{y}\right)^{\frac{k}{d}}.
\]
At the price of creating a new cell with bounded $y$-fibers (which
can be dealt with by Proposition \ref{prop: strongly int on bounded cells}),
we may suppose that $y>a\left(x\right)\varphi_{0}\left(x\right)^{d}$
on $B$, so that, up to modifying the definition of $\psi$, it is
clear that $\varphi_{>}$ is bounded on $B$ and hence $\text{e}^{\text{i}\varphi_{>}}$
is a complex-valued $\psi$-prepared subanalytic strong function (we
deal in the same way with the case when $\left|\varphi_{0}\right|$
is bounded away from zero and $\omega<0$).

It follows that $\text{e}^{\text{i}\varphi_{>}}$ can be absorbed
into $\Phi$ and hence
\[
T\left(s,x,y\right)=y^{\lambda\left(s\right)}\left(\log y\right)^{\mu}\text{e}^{\text{i}Q\left(x,y\right)}\sum_{k}f_{k}\left(s,x\right)y^{-\frac{k}{d}},
\]
where $f_{k}\left(s,x\right)=f\left(s,x\right)\xi_{k}\left(s,x\right)a\left(x\right)^{\frac{k}{d}}\in\mathcal{C}_{\Sigma}^{\mathcal{M},\mathcal{F}}\left(X\right)$,
in the notation of \eqref{eq:eq: nested prep form unb}. Since $\Sigma$
has bounded width, there exists $k_{0}\in\mathbb{N}$ such that, setting
$\lambda_{k}\left(s\right)=\lambda\left(s\right)-\frac{k}{d}$, there
exists $s\in\Sigma$ such that $\Re\left(\lambda_{k}\left(s\right)\right)\geq-1$
if and only if $k\leq k_{0}$. It follows that
\[
T\left(s,x,y\right)=\sum_{k\leq k_{0}}f_{k}\left(s,x\right)y^{\lambda_{k}\left(s\right)}\left(\log y\right)^{\mu}\text{e}^{\text{i}Q\left(x,y\right)}+R\left(s,x,y\right),
\]
where the terms of the sum are monomial in $y$ and $R$ is strongly
integrable.
\end{proof}
In order to extend Proposition \ref{prop: naive in y are sums of mon + ssint}
to generators which are not naive in $y$ we need some preparatory
work.
\begin{lem}
\label{lem: cond suff ssint}Let $T$ be a prepared generator (as
in Definition \ref{def: prepared generator}) on a cell $B$ with
unbounded $y$-fibers. Suppose that, either
\begin{enumerate}
\item $D$ has unbounded $t$-fibers and $\forall s\in\Sigma,\ \Re\left(\lambda\left(s\right)\right)<-1$,
or
\item $D$ has bounded $t$-fibers and $\forall s\in\Sigma,\ \Re\left(\lambda\left(s\right)\right)+\frac{\Delta}{d}<-1\text{ and }\Re\left(\varrho\left(s\right)\right)\leq0$.
\end{enumerate}
Then $T$ is strongly integrable.
\end{lem}
\begin{proof}
Suppose that $T$ has no poles outside some closed discrete set $P\subseteq\mathbb{C}$.
If we write $T$ as in \eqref{eq: prepared gen as integral of G},
then it suffices to prove that $G\in\mathcal{C}_{\Sigma}^{\mathcal{M}}\left(X\times\mathbb{R}^{2}\right)_{\text{int}}$.
By Tonelli's Theorem, it is enough to prove that for all $\left(s,x\right)\in\left(\Sigma\setminus P\right)\ni X$,
the function
\[
f_{\left(s,x\right)}:y\longmapsto\int_{\mathbb{R}}\left|G\left(s,x,y,t\right)\right|\text{d}t
\]
is in $L^{1}\left(\mathbb{R}\right)$.
\begin{enumerate}
\item Let $\widetilde{b}=+\infty$. Since $\forall s\in\Sigma,\ \Re\left(\varrho\left(s\right)\right)<-1$,
there is a positive constant $M$ such that
\[
\int_{1}^{+\infty}\left|t^{\varrho\left(s\right)}\left(\log t\right)^{\nu}\Phi\left(s,x,y,t\right)\right|\mathrm{d}t\leq M.
\]
Therefore,
\[
\left|f_{\left(s,x\right)}\left(y\right)\right|\leq M\left|g\left(s,x\right)\right|\chi_{B}\left(x,y\right)y^{\lambda\left(s\right)}\left(\log y\right)^{\mu}
\]
and, as $\forall s\in\Sigma,\ \Re\left(\lambda\left(s\right)\right)<-1$,
we have that $f_{\left(s,x\right)}\in L^{1}\left(\mathbb{R}\right)$.
\item Let now $\widetilde{b}<+\infty$. Since $\forall s\in\Sigma,\ \Re\left(\varrho\left(s\right)\right)\leq0$
and $\Phi$ is bounded, there is a positive constant $M\left(x\right)$
such that
\[
\int_{\widetilde{a}\left(x,y\right)}^{\widetilde{b}\left(x,y\right)}\left|t^{\varrho\left(s\right)}\left(\log t\right)^{\nu}\Phi\left(s,x,y,t\right)\right|\mathrm{d}t\leq M\left(x\right)y^{\frac{\Delta}{d}}\left(\log y\right)^{\nu}.
\]
Therefore,
\[
\left|f_{\left(s,x\right)}\left(y\right)\right|\leq M\left(x\right)\left|g\left(s,x\right)\right|\chi_{B}\left(x,y\right)y^{\lambda\left(s\right)+\frac{\Delta}{d}}\left(\log y\right)^{\mu+\nu}
\]
and, as $\forall s\in\Sigma,\ \Re\left(\lambda\left(s\right)\right)+\frac{\Delta}{d}<-1$,
we have that $f_{\left(s,x\right)}\in L^{1}\left(\mathbb{R}\right)$.
\end{enumerate}
\end{proof}
\begin{lem}
\label{lem: int by parts}Let $T$ be a prepared generator (as in
Definition \ref{def: prepared generator}) on a cell $B$ with unbounded
$y$-fibers and let $k_{0}\in\mathbb{N}.$ Then $T$ can be rewritten
as a finite sum of prepared generators which are either naive in $y$
or such that for all $s\in\Sigma,\ \Re\left(\varrho\left(s\right)\right)<-k_{0}$.
\end{lem}
\begin{proof}
Suppose that $T$ has no poles outside some closed discrete set $P\subseteq\mathbb{C}$.
Write $T$ on $B$ as
\[
g\left(s,x\right)y^{\lambda\left(s\right)}\left(\log y\right)^{\mu}\text{e}^{\text{i}\varphi\left(x,y\right)}\gamma\left(s,x,y\right),
\]
where
\[
\gamma\left(s,x,y\right)=\int_{\widetilde{a}\left(x,y\right)}^{\widetilde{b}\left(x,y\right)}t^{\varrho\left(s\right)}\left(\log t\right)^{\nu}\Phi\left(s,x,y,t\right)\text{e}^{\sigma\text{i}t}\text{d}t.
\]
Since $B$ has unbounded $y$-fibers, $\Phi$ has the following $\varPsi$-prepared
nested form (see \cite[Remark 3.7]{cmrs:mellin_constructible}) with
respect to the last three components of $\varPsi$:
\[
\Phi\left(s,x,y,t\right)=F\circ\left(s,x,\Xi\left(x,y,t\right)\right),
\]
where
\[
F\left(s,x,Y_{0},Y_{1},Y_{2}\right)=\sum_{k,m,n}\xi_{k,m,n}\left(s,x\right)Y_{0}^{k}Y_{1}^{m}Y_{2}^{n}\in\mathcal{A}_{\Sigma}\left(x\right)\left\llbracket Y_{0},Y_{1},Y_{2}\right\rrbracket
\]
is strongly convergent and
\[
\Xi\left(x,y,t\right)=\left(\left(\frac{a\left(x\right)}{y}\right)^{\frac{1}{d}},\left(\frac{a_{0}\left(x\right)y^{\frac{\alpha}{d}}}{t}\right)^{\frac{1}{d}},\left(\frac{t}{b_{0}\left(x\right)y^{\frac{\beta}{d}}}\right)^{\frac{1}{d}}\right)
\]
(the variable $Y_{2}$ and the last component of $\Xi$ are missing
if $D$ has unbounded $t$-fibers).

Fix $\left(s,x,y\right)\in\left(\Sigma\setminus P\right)\times B$
and apply integration by parts to the transcendental element $\gamma$,
where we integrate $\mathrm{e}^{\mathrm{i}\sigma t}$ and derivate
$f\left(t\right):=t^{\varrho\left(s\right)}\left(\log t\right)^{\nu}\Phi\left(s,x,y,t\right)$.
For this, write
\[
f'\left(t\right)=t^{-1}\left[\varrho\left(s\right)t^{\varrho\left(s\right)}\left(\log t\right)^{\nu}\Phi+\nu t^{\varrho\left(s\right)}\left(\log t\right)^{\nu-1}\Phi+t^{\varrho\left(s\right)}\left(\log t\right)^{\nu}\widetilde{\Phi}\right],
\]
where
\[
\widetilde{\Phi}\left(s,x,y,t\right)=\widetilde{F}\circ\left(s,x,\Xi\left(x,y,t\right)\right)\text{ with }\widetilde{F}=-\frac{1}{d}Y_{1}\frac{\partial F}{\partial Y_{1}}+\frac{1}{d}Y_{2}\frac{\partial F}{\partial Y_{2}}.
\]
In particular, $\widetilde{F}$ is strongly convergent and $\widetilde{\Phi}$
is a $\varPsi$-prepared parametric strong function. Notice that each
of the terms of $f'\left(t\right)\frac{\mathrm{e}^{\mathrm{i}\sigma t}}{\mathrm{i}\sigma}$
gives rise to a prepared generator such that the exponent of $t$
in the transcendental element is $\varrho\left(s\right)-1$. The other
terms produced by integration by parts are of the form
\begin{equation}
-\sigma\text{i}\widetilde{c}\left(x,y\right)^{\varrho\left(s\right)}\left(\log\widetilde{c}\left(x,y\right)\right)^{\nu}\mathrm{e}^{\sigma\mathrm{i}\widetilde{c}\left(x,y\right)}\Phi\left(s,x,y,\widetilde{c}\left(x,y\right)\right),\label{eq:naive terms in integration by parts}
\end{equation}
where $\widetilde{c}$ is either $\widetilde{a}$ or $\widetilde{b}$
(or the whole term is replaced by zero, if $\widetilde{b}=+\infty$,
since then in this case for all $s\in\Sigma,\ \Re\left(\varrho\left(s\right)\right)<-1<0$).
Now, if $\widetilde{c}=\widetilde{b}$, then
\[
\Xi\left(x,y,\widetilde{b}\left(x,y\right)\right)=\left(\left(\frac{a\left(x\right)}{y}\right)^{\frac{1}{d}},\left(\frac{a_{0}\left(x\right)y^{\frac{\alpha}{d}}}{b_{0}\left(x\right)y^{\frac{\beta}{d}}u_{b}\left(x,y\right)}\right)^{\frac{1}{d}},\left(u_{b}\left(x,y\right)\right)^{\frac{1}{d}}\right)
\]
and
\[
\frac{a_{0}\left(x\right)y^{\frac{\alpha}{d}}}{b_{0}\left(x\right)y^{\frac{\beta}{d}}u_{b}\left(x,y\right)}=\left[\frac{a_{0}\left(x\right)}{b_{0}\left(x\right)a\left(x\right)^{\frac{\beta-\alpha}{d}}}\right]\left(u_{b}\left(x,y\right)\right)^{-1}\left(\frac{a\left(x\right)}{y}\right)^{\frac{\beta-\alpha}{d}}.
\]
The term between square brackets is bounded on $B$, because all the
other terms are, so we can add it to the list of functions $c\left(x\right)$
in the definition of $\psi$ (see \ref{eq: psi}). The unit $u_{b}$
is $\psi$-prepared hence $\Phi\left(s,x,y,\widetilde{b}\left(x,y\right)\right)$
is a $\psi$-prepared parametric strong function. A similar calculation
shows that so is $\Phi\left(s,x,y,\widetilde{a}\left(x,y\right)\right)$.
Thus the terms \ref{eq:naive terms in integration by parts} are generators
which are naive in $y$ and prepared with respect to $\psi$.

We iterate the process to further reduce the exponent of $t$: since
$\Sigma$ has bounded width, ${\displaystyle \sup_{s\in\Sigma}}\left(\Re\left(\varrho\left(s\right)\right)\right)\in\mathbb{R}$.
Let $M:={\displaystyle \left\lfloor {\displaystyle \sup_{s\in\Sigma}}\left(\Re\left(\varrho\left(s\right)\right)\right)\right\rfloor }$.
By integrating by parts $M+k_{0}+1$ times, we can rewrite $T$ as
a finite sum of generators which are either naive in $y$ and prepared
with respect to $\psi$, or such that the exponent of $t$ in the
transcendental element $\gamma$ is $\varrho\left(s\right)-M-k_{0}-1$,
whose real part is $<-k_{0}$.
\end{proof}
\begin{prop}
\label{prop: generator sum of naive and ssint}Let $T$ be a prepared
generator (as in Definition \ref{def: prepared generator}) on a cell
$B$ with unbounded $y$-fibers and suppose that $T$ has no poles
outside some closed discrete set $P\subseteq\mathbb{C}$. Then there
exist a closed discrete set $P'\subseteq\mathbb{C}$ containing $P$
and such that $P'\setminus P$ is contained in a finitely generated
$\mathbb{Z}$-lattice, a finite partition of $B$ into subcells and,
on each subcell $B'$ which is open over $X$, finitely many generators
$T_{i}$ which are either naive in $y$ or strongly integrable on
$B'$, such that for all $\left(s,x,y\right)\in\left(\Sigma\setminus P'\right)\times B',\ T\left(s,x,y\right)=\sum T_{i}\left(s,x,y,\right)$.
\end{prop}
\begin{proof}
Recall Notation \ref{notation: t-cell}. Note that $0\leq\alpha\leq\beta$.
There are three cases:
\begin{enumerate}
\item $\alpha=\beta=0$
\item $\alpha>0$
\item $\alpha=0,\beta>0$
\end{enumerate}
Define $C:={\displaystyle \sup_{s\in\Sigma}}\left(\Re\left(\lambda\left(s\right)\right)\right)$.
\begin{enumerate}
\item The case $\beta=0$ also includes the case $\widetilde{b}\left(x,y\right)=b_{0}\left(x\right)=+\infty$.
We claim that, at the price of creating a new cell with bounded $y$-fibers
(which can be handled using Proposition \ref{prop: strongly int on bounded cells}),
we may suppose that for all $\left(x,y\right)\in B$\\
$\bullet$ $\widetilde{a}\left(x,y\right)\leq a_{0}\left(x\right)$
and $b_{0}\left(x\right)\leq\widetilde{b}\left(x,y\right)$;\\
$\bullet$ $\left|\widetilde{a}\left(x,y\right)-a_{0}\left(x\right)\right|\leq1$
and $\left|\widetilde{b}\left(x,y\right)-b_{0}\left(x\right)\right|\leq1$
\\
The proof of the claim can be found in \cite[p.1277]{ccmrs:integration-oscillatory}
and only uses basic o-minimal properties of subanalytic sets and functions.
\\
Therefore, we may write the transcendental element $\gamma$ as the
sum of three integrals, with integration bounds, respectively, $\left(\widetilde{a}\left(x,y\right),a_{0}\left(x\right)\right)$,
$\left(a_{0}\left(x\right),b_{0}\left(x\right)\right)$ and $\left(b_{0}\left(x\right),\widetilde{b}\left(x,y\right)\right)$
(if $\widetilde{b}=+\infty$, then the second integral has $+\infty$
as upper integration bound and the third integral is missing). The
integral with bounds $\left(\widetilde{a}\left(x,y\right),a_{0}\left(x\right)\right)$
can be written as
\begin{equation}
\mathrm{e}^{\mathrm{i}\sigma a_{0}\left(x\right)}\int_{a_{0}\left(x\right)}^{a_{0}\left(x\right)u_{a}\left(x,y\right)}t^{\varrho\left(s\right)}\left(\log t\right)^{\nu}\Phi\left(s,x,y,t\right)\mathrm{e}^{\sigma\mathrm{i}\left(t-a_{0}\left(x\right)\right)}\text{d}t,\label{eq: 1}
\end{equation}
where, thanks to the claim, $\mathrm{e}^{\sigma\mathrm{i}\left(t-a_{0}\left(x\right)\right)}$
is a complex-valued subanalytic function on $B$. Hence the integrand
is in $\mathcal{C}_{\Sigma}^{\mathcal{M}}\left(B\right)$ and we can
invoke Theorem \ref{thm: C_M stability} to obtain that the term containing
\eqref{eq: 1} can be written, outside some closed discrete set $P'$
containing $P$, as a generator of $\mathcal{C}_{\Sigma}^{\mathcal{M},\mathcal{F}}\left(X\times\mathbb{R}\right)$
which is naive in $y$. The integral with bounds $\left(b_{0}\left(x\right),\widetilde{b}\left(x,y\right)\right)$,
if present, is handled similarly. For the integral with bounds $\left(a_{0}\left(x\right),b_{0}\left(x\right)\right)$,
notice that now the variable $y$ only appears in the parametric strong
function $\Phi$, which we can write in nested form with respect to
the last component of $\psi$ as
\begin{equation}
\sum_{k\geq0}\widetilde{\xi}_{k}\left(s,x,t\right)\left(\frac{a\left(x\right)}{y}\right)^{\frac{k}{d}},\label{eq: series expansion in y}
\end{equation}
for some $\widetilde{\xi}_{k}\in\mathcal{A}_{\Sigma}\left(E\right)$,
where $E=\left\{ \left(x,t\right):\ x\in X,\ a_{0}\left(x\right)<t<b_{0}\left(x\right)\right\} $.
Let $k_{0}=\lceil d\left(C+1\right)\rceil+1$, so that for all $s\in\Sigma,\ \Re\left(\lambda\left(s\right)\right)-\frac{k_{0}}{d}<-1$,
and let
\[
\Phi_{>k_{0}}=\Phi-\sum_{k\leq k_{0}}\widetilde{\xi}_{k}\left(s,x,t\right)\left(\frac{a\left(x\right)}{y}\right)^{\frac{k}{d}}.
\]
Setting
\[
f_{k}\left(s,x\right)=g\left(s,x\right)\left(a\left(x\right)\right)^{\frac{k}{d}}\int_{a_{0}\left(x\right)}^{b_{0}\left(x\right)}t^{\varrho\left(s\right)}\left(\log t\right)^{\nu}\widetilde{\xi}_{k}\left(s,x,t\right)\mathrm{e}^{\sigma\mathrm{i}t}\text{d}t\in\mathcal{C}_{\Sigma}^{\mathcal{M},\mathcal{F}}\left(X\right),
\]
write the term
\[
g\left(s,x\right)y^{\lambda\left(s\right)}\left(\log y\right)^{\mu}\mathrm{e}^{\mathrm{i}\varphi\left(x,y\right)}\int_{a_{0}\left(x\right)}^{b_{0}\left(x\right)}t^{\varrho\left(s\right)}\left(\log t\right)^{\nu}\mathrm{e}^{\sigma\mathrm{i}t}\Phi\left(s,x,y,t\right)\text{d}t
\]
as the following sum of generators which are naive in $y$
\[
\sum_{k\leq k_{0}}f_{k}\left(s,x\right)y^{\lambda\left(s\right)-\frac{k}{d}}\left(\log y\right)^{\mu}\mathrm{e}^{\mathrm{i}\varphi\left(x,y\right)},
\]
plus the term
\begin{equation}
g\left(s,x\right)\left(a\left(x\right)\right)^{\frac{k_{0}}{d}}y^{\lambda\left(s\right)-\frac{k_{0}}{d}}\left(\log y\right)^{\mu}\mathrm{e}^{\mathrm{i}\varphi\left(x,y\right)}\int_{a_{0}\left(x\right)}^{b_{0}\left(x\right)}t^{\varrho\left(s\right)}\left(\log t\right)^{\nu}\Phi_{>k_{0}}\left(s,x,y,t\right)\mathrm{e}^{\sigma\mathrm{i}t}\text{d}t,\label{eq: last term}
\end{equation}
which is strongly integrable on $B$ by definition of $k_{0}$ and
since $\Phi_{>k_{0}}$ is bounded.
\item Let $N_{0}=\left\lceil \frac{dC+\Delta+d}{\alpha}\right\rceil +1$.
By Lemma \ref{lem: int by parts}, we may suppose that for all $s\in\Sigma,\ \Re\left(\varrho\left(s\right)\right)<-N_{0}$.
Write
\[
t^{\varrho\left(s\right)}=t^{\varrho\left(s\right)+N_{0}}\left(a_{0}\left(x\right)y^{\frac{\alpha}{d}}\right)^{-N_{0}}\left[\left(\frac{a_{0}\left(x\right)y^{\frac{\alpha}{d}}}{t}\right)^{\frac{1}{d}}\right]^{dN_{0}}.
\]
The rightmost term in the above formula can be absorbed into $\Phi$
and the central term can be factored out of the integral defining
the transcendental element $\gamma$. By the choice of $N_{0}$, for
all $s\in\Sigma,$
\begin{equation}
\Re\left(\lambda\left(s\right)\right)-\frac{N_{0}\alpha}{d}+\frac{\Delta}{d}<-1,\label{eq:cond on exp 0}
\end{equation}
so by Lemma \ref{lem: cond suff ssint} (either of the two conditions,
depending on the nature of the $t$-fibers of the cell $D$) $T$
is strongly integrable on $B$.
\item Let $N_{0}=\left\lceil \frac{dC+\Delta+d}{\beta}\right\rceil +1$
and $k_{0}=d\left(N_{0}-1\right)$. By Lemma \ref{lem: int by parts},
we may suppose that for all $s\in\Sigma,\ \Re\left(\varrho\left(s\right)\right)<-N_{0}$.
This implies in particular that for all $s\in\Sigma$,
\begin{align}
 & \Re\left(\varrho\left(s\right)\right)+\frac{k_{0}}{d}<-1,\label{eq: condition on exponents}\\
 & \Re\left(\varrho\left(s\right)\right)+\frac{k_{0}+1}{d}\leq0,\label{eq: cond on exp 2}\\
 & \Re\left(\lambda\left(s\right)\right)-\frac{\beta\left(k_{0}+1\right)}{d^{2}}+\frac{\Delta}{d}<-1.\label{eq:cond on exp 3}
\end{align}
First, we split $\Phi\left(s,x,y,t\right)$ into the sum of two series,
by separating the positive and the negative powers of $t$:
\[
\Phi=\sum_{k>0}\xi_{k}^{<}\left(s,x,y\right)\left(\frac{a_{0}\left(x\right)}{t}\right)^{\frac{k}{d}}+\sum_{k\geq0}\xi_{k}^{>}\left(s,x,y\right)\left(\frac{t}{b_{0}\left(x\right)y^{\frac{\beta}{d}}}\right)^{\frac{k}{d}}.
\]
Next, write
\begin{align*}
\gamma_{\leq k_{0}}\left(s,x,y\right) & =\int_{a_{0}\left(x\right)u_{a}\left(x,y\right)}^{b_{0}\left(x\right)y^{\frac{\beta}{d}}u_{b}\left(x,y\right)}t^{\varrho\left(s\right)}\left(\log t\right)^{\nu}\Phi_{\leq k_{0}}\left(s,x,y,t\right)\mathrm{e}^{\sigma\mathrm{i}t}\text{d}t,\\
\\
\gamma_{>k_{0}}\left(s,x,y\right) & =\int_{a_{0}\left(x\right)u_{a}\left(x,y\right)}^{b_{0}\left(x\right)y^{\frac{\beta}{d}}u_{b}\left(x,y\right)}t^{\varrho\left(s\right)}\left(\log t\right)^{\nu}\left(\frac{t}{b_{0}\left(x\right)y^{\frac{\beta}{d}}}\right)^{\frac{k_{0}+1}{d}}\Phi_{>k_{0}}\left(s,x,y,t\right)\mathrm{e}^{\sigma\mathrm{i}t}\text{d}t,
\end{align*}
where
\begin{align*}
\Phi_{\leq k_{0}}= & \sum_{k>0}\xi_{k}^{<}\left(s,x,y\right)\left(\frac{a_{0}\left(x\right)}{t}\right)^{\frac{k}{d}}+\sum_{k=0}^{k_{0}}\xi_{k}^{>}\left(s,x,y\right)\left(\frac{t}{b_{0}\left(x\right)y^{\frac{\beta}{d}}}\right)^{\frac{k}{d}},\\
\Phi_{>k_{0}}= & \sum_{k\geq0}\xi_{k+k_{0}+1}^{>}\left(s,x,y\right)\left(\frac{t}{b_{0}\left(x\right)y^{\frac{\beta}{d}}}\right)^{\frac{k}{d}}.
\end{align*}
\\
By \eqref{eq: condition on exponents} and linearity, we may write
$\gamma_{\leq k_{0}}$ as the sum of two integrals with upper integration
bound equal to $+\infty$ and the lower integration bounds equal to,
respectively, $a_{0}\left(x\right)u_{a}\left(x,y\right)$ and $b_{0}\left(x\right)y^{\frac{\beta}{d}}u_{b}\left(x,y\right)$.
The first integral falls within the scope of the first part of this
proof, whereas the second integral falls within the scope of the second
part of this proof.\\
It remains to consider
\begin{align}
 & T_{>k_{0}}\left(s,x,y\right):=g\left(s,x\right)y^{\lambda\left(s\right)}\left(\log y\right)^{\mu}\mathrm{e}^{\mathrm{i}\varphi\left(x,y\right)}\gamma_{>k_{0}}\left(s,x,y\right)\label{eq: last term-1}\\
 & =g\left(s,x\right)y^{\lambda\left(s\right)-\frac{\beta\left(k_{0}+1\right)}{d^{2}}}\left(\log y\right)^{\mu}\mathrm{e}^{\mathrm{i}\varphi\left(x,y\right)}\left(b_{0}\left(x\right)\right)^{-\frac{\beta\left(k_{0}+1\right)}{d^{2}}}\cdot\nonumber \\
 & \cdot\int_{a_{0}\left(x\right)u_{a}\left(x,y\right)}^{b_{0}\left(x\right)y^{\frac{\beta}{d}}u_{b}\left(x,y\right)}t^{\varrho\left(s\right)+\frac{k_{0}+1}{d}}\left(\log t\right)^{\nu}\Phi_{>k_{0}}\left(s,x,y,t\right)\mathrm{e}^{\sigma\mathrm{i}t}\text{d}t.\nonumber
\end{align}
By \eqref{eq: cond on exp 2} and \eqref{eq:cond on exp 3}, $T_{\geq k_{0}}$
satisfies the second condition in Lemma \ref{lem: cond suff ssint}.
\end{enumerate}
\end{proof}

\section{\label{sec: stability}Interpolation and stability under integration}

In this section we finish the proof of Theorem \ref{thm: stability of C^F}.
For this, it suffices to consider the $1$-dimensional case, Theorem
\ref{thm: interpolation} below (the general $n$-dimensional case
follows from Fubini's Theorem, see the end of Section \ref{subsec: Proof strategy}),
the proof of which requires an analysis of the integration locus.

With this in mind, we adapt \cite[Definition 6.1]{cmrs:mellin_constructible}
to the current setting.
\begin{defn}
\label{def: grid}Given $N,d\in\mathbb{N}\setminus\left\{ 0\right\} $
and $\left\{ \left(\ell_{i},r_{i}\right):\ 1\leq i\leq N\right\} \subseteq\mathbb{R}^{2}$,
define
\begin{align*}
\Xi_{i,0,-} & =\emptyset,\\
\Xi_{i,0,\circ} & =\left\{ s\in\Sigma:\ \ell_{i}\Re\left(s\right)+r_{i}+d<0\right\} \quad\left(1\leq i\leq N\right),\\
\Xi_{i,j,-} & =\left\{ s\in\Sigma:\ \ell_{i}\Re\left(s\right)+r_{i}+d=j-1\right\} \quad\left(1\leq i\leq N,\ j\in\mathbb{N}\setminus\left\{ 0\right\} \right),\\
\Xi_{i,j,\circ} & =\left\{ s\in\Sigma:\ j-1<\ell_{i}\Re\left(s\right)+r_{i}+d<j\right\} \quad\left(1\leq i\leq N,\ j\in\mathbb{N}\setminus\left\{ 0\right\} \right).
\end{align*}
The collection
\[
\mathcal{G}=\left\{ \Xi_{i,j,\star}:\ 1\leq i\leq N,\ j\in\mathbb{N},\ \star\in\left\{ -,\circ\right\} \right\}
\]
is called the \emph{grid} of \emph{denominator} $d$ and \emph{data}
$\left(d,\left\{ \left(\ell_{i},r_{i}\right):\ 1\leq i\leq N\right\} \right)$.
A \emph{$\mathcal{G}$-cell} is a nonempty subset $S\subseteq\Sigma$
such that
\[
\forall\Xi\in\mathcal{G},\ \Xi\cap S=\emptyset\text{\ or }S\subseteq\Xi\text{, and }S=\bigcap\left\{ \Xi\in\mathcal{G}:\ S\subseteq\Xi\right\} .
\]
Finally, given a prepared generator $T$ as in Definition \ref{def: prepared generator},
we call the tuple
\[
\left(d^{2},\left\{ \left(d\ell+\delta\widetilde{\ell},d\Re\left(\eta\right)+\delta\Re\left(\widetilde{\eta}\right)\right):\ \delta\in\left\{ 0,\alpha,\beta\right\} \right\} \right)
\]
the \emph{grid data} of $T$.
\end{defn}
\begin{rems}
\label{rems: grids}$\ $
\begin{enumerate}
\item Since $\Sigma$ has bounded width, a grid $\mathcal{G}$ induces a
finite partition $\mathcal{R}\left(\mathcal{G}\right)$ of $\Sigma$
into $\mathcal{G}$-cells, and each $\mathcal{G}$-cell is either
an open vertical substrip of $\Sigma$ or a vertical line.
\item A prepared generator $T$ generates a grid of data the grid data of
$T$. In this case, if $T$ is monomial in $y$ then on each $\mathcal{G}$-cell
$S$ the real part of the exponent of $y$ is either always $<-1$
or always $\geq-1$.
\end{enumerate}
\end{rems}
\begin{notation}
\label{not : different polynomials}Given a subanalytic set $X\subseteq\mathbb{R}^{m}$
and functions $Q_{1},Q_{2}$$\in\mathcal{S}\left(X\right)\left[y^{\frac{1}{d}}\right]$
which are polynomials in $y^{\frac{1}{d}}$ with coefficients subanalytic
functions of $x$, it is clearly possible to partition $X$ into finitely
many subanalytic cells such that for each cell $X'$, either for all
$x\in X',\ Q_{1}\left(x,\cdot\right)$ and $Q_{2}\left(x,\cdot\right)$
define the same polynomial function, or for all $x\in X',\ Q_{1}\left(x,\cdot\right)$
and $Q_{2}\left(x,\cdot\right)$ define different polynomial functions.
In this case, we will say, respectively, that $Q_{1}=Q_{2}$ on $X'$
or $Q_{1}\not=Q_{2}$ on $X'$.
\end{notation}
\begin{prop}[Splitting]
\label{prop: Splitting} Let $h\in\mathcal{C}_{\Sigma}^{\mathcal{M},\mathcal{F}}\left(X\times\mathbb{R}\right)$
be without poles outside some closed discrete set $P\subseteq\mathbb{C}$.
There are a closed discrete set $P'\subseteq\mathbb{C}$ containing
$P$ and such that $P'\setminus P$ is contained in a finitely generated
$\mathbb{Z}$-lattice, $d\in\mathbb{N}\setminus\left\{ 0\right\} $,
finite sets $J_{\mathrm{int}},J_{\mathrm{mon}}\subseteq\mathbb{N}$
and a $d$-cell decomposition (see Remark \ref{rem: d}) of $\mathbb{R}^{m+1}$
compatible with $X$ such that for each cell $A$ that is open over
$\mathbb{R}^{m}$ (which we may suppose to be of the form \eqref{eq:A}),
\begin{equation}
h\circ\Pi_{A}=\sum_{j\in J_{\mathrm{int}}}T_{j}+\sum_{j\in J_{\mathrm{mon}}}T_{j},\label{eq: splitting}
\end{equation}
 where each $T_{j}$ is a $d$-prepared generator without poles outside
$P'$ (see Definition \ref{def: prepared generator} and Remark \ref{rem: d}).
Moreover, using the notation in Definition \ref{def: superintegrable and monomial genrators},
\begin{enumerate}
\item For every $j\in J_{\mathrm{int}},\ T_{j}$ is strongly integrable
on $B_{A}$ and if $B_{A}$ has unbounded $y$-fibers, then, in the
notation of Definition \ref{def: prepared generator}, for all $s\in\Sigma,\ \Re\left(\lambda\left(s\right)\right)<-1$.
\item For every $j\in J_{\mathrm{mon}},\ T_{j}$ is monomial in $y$, with
monomial data $\left(d,\ell_{j},\eta_{j},\mu_{j},Q_{j}\right)$, where:
\begin{enumerate}
\item for all $x\in X,\ Q_{j}\left(x,0\right)=0$ and for all $i,j\in J_{\mathrm{mon}}$,
either $Q_{i}=Q_{j}$ on $X$ or $Q_{i}\not=Q_{j}$ on $X$ (see Notation
\ref{not : different polynomials});
\item the tuples $\left(\ell_{j},\eta_{j},\mu_{j},Q_{j}\right)\in\mathbb{Z}\times\mathbb{C}\times\mathbb{N}\times\mathcal{S}\left(X\right)\left[y^{\frac{1}{d}}\right]$
$\left(j\in J_{\mathrm{mon}}\right)$ are pairwise distinct;
\item there is a grid $\mathcal{G}$ such that for all $\mathcal{G}$-cell
$S$, for all $j\in J_{\mathrm{mon}}$, either $\Re\left(\frac{\ell_{j}s+\eta_{j}}{d}\right)<-1$
for all $s\in S$, or $\Re\left(\frac{\ell_{j}s+\eta_{j}}{d}\right)\geq-1$
for all $s\in S$.
\end{enumerate}
\end{enumerate}
\end{prop}
\begin{proof}
Apply Proposition \ref{prop: preparation, 1st version} to $h$. This
produces $d$ and a $d$-cell decomposition of $\mathbb{R}^{m+1}$
such that on each cell $A$ open over $X$, $h\circ\Pi_{A}$ is a
finite sum of prepared generators $T$. Collect the grid data of all
the prepared generators and generate the corresponding grid $\mathcal{G}$
with denominator $d^{2}$. For each cell $A$, apply Propositions
\ref{prop: generator sum of naive and ssint} and \ref{prop: naive in y are sums of mon + ssint}
to each prepared generator $T$ on $B_{A}$. This produces a refinement
of the $d$-cell decomposition and rewrites $T$ on each cell as a
finite sum of prepared generators $T'$ which are either strongly
integrable (and satisfying condition (1)) or monomial in $y$. Up
to absorbing $\text{e}^{\text{i}Q_{j}\left(x,0\right)}$ into $f_{j}\left(s,x\right)$
and up to partitioning $X$ into subanalytic cells, we may suppose
that item (2.a) in the statement of the proposition is satisfied.
Summing like terms we may also suppose that item (2.b) is satisfied.
Revisiting the proofs of Propositions \ref{prop: generator sum of naive and ssint}
and \ref{prop: naive in y are sums of mon + ssint}, which are based
on integration by parts of the transcendental elements and series
expansion of parametric strong functions on cells with unbounded $y$-fibers,
we see that if the exponents of $y$ and $t$ in the original prepared
generator $T$ are $\lambda\left(s\right)$ and $\varrho\left(s\right)$
respectively, then the exponents of $y$ in the newly created monomial
generators $T'$ have the form $\lambda\left(s\right)-\frac{k}{d}+\frac{\delta}{d}\left(\varrho\left(s\right)-k'\right)$,
for some $k,k'\in\mathbb{N}$ and $\delta\in\left\{ 0,\alpha,\beta\right\} $.
In particular, the grid generated by the grid data of the new monomial
generators does not create any new cell. By Remark \ref{rem: d} we
may rename $d^{2}$ as $d$ and adapt accordingly the definitions
of $\ell_{j},\eta_{j},\alpha_{j},\beta_{j}$, so that, by Remark \ref{rems: grids}
(2), item (2.c) in the statement of the proposition is also satisfied.
\end{proof}
\begin{thm}[Interpolation and integration locus]
\label{thm: interpolation} Let $h\in\mathcal{C}_{\Sigma}^{\mathcal{M},\mathcal{F}}\left(X\times\mathbb{R}\right)$
be without poles outside some closed discrete set $P\subseteq\mathbb{C}$.
There are a closed discrete set $P'\subseteq\mathbb{C}$ containing
$P$ and such that $P'\setminus P$ is contained in a finitely generated
$\mathbb{Z}$-lattice and a function $H\in\mathcal{C}_{\Sigma}^{\mathcal{M},\mathcal{F}}\left(X\right)$
without poles outside $P'$ such that
\[
\forall\left(s,x\right)\in\mathrm{Int}\left(h;\left(\Sigma\setminus P'\right)\times X\right),\quad\int_{\mathbb{R}}h\left(s,x,y\right)\mathrm{d}y=H\left(s,x\right).
\]
Moreover, there exists a grid $\mathcal{G}$ such that
\begin{equation}
\mathrm{Int}\left(h;\left(\Sigma\setminus P'\right)\times X\right)=\bigcup_{S\in\mathcal{\mathcal{R}\left(\mathcal{G}\right)}}\left\{ \left(s,x\right):\ s\in S\setminus P',\ \bigwedge_{j\in J_{S}}f_{j}\left(s,x\right)=0\right\} ,\label{eq:int locus}
\end{equation}
for a suitable finite set $J_{S}$ and suitable $f_{j}\in\mathcal{C}_{\Sigma}^{\mathcal{M},\mathcal{F}}\left(X\right)$
without poles outside $P'$.
\end{thm}
\begin{proof}
Apply Proposition \ref{prop: Splitting} to $h$: this produces a
closed discrete set $P'\subseteq\mathbb{C}$ containing $P$ and such
that $P'\setminus P$ is contained in a finitely generated $\mathbb{Z}$-lattice,
$d\in\mathbb{N}\setminus\left\{ 0\right\} $, finite sets $J_{\mathrm{int}},J_{\mathrm{mon}}\subseteq\mathbb{N}$,
a grid $\mathcal{G}$ and a $d$-cell decomposition, such that the
conclusion of the proposition holds. By linearity of the integral,
it suffices to prove the statement of the theorem for $h\restriction A$,
where $A$ is a cell of the decomposition which is open over $X$.
Recall Notation \ref{notation: cell bijection B_A} and note that
\[
\frac{\partial\Pi_{A}}{\partial y}\left(x,y\right)=\sigma_{A}\tau_{A}y^{\tau_{A}-1}.
\]
Thus, up to multiplying each $T_{j}$ in \eqref{eq: splitting} by
$\frac{\partial\Pi_{A}}{\partial y}\left(x,y\right)$, we may write
that for all $\left(s,x\right)\in\text{Int}\left(h\restriction A;\left(\Sigma\setminus P'\right)\times X\right),$
\[
\int_{A_{x}}h\left(s,x,y\right)\text{d}y=\int_{a_{A}\left(x\right)}^{b_{A}\left(x\right)}\left(\sum_{j\in J_{\mathrm{int}}}T_{j}\left(s,x,y\right)+\sum_{j\in J_{\mathrm{mon}}}T_{j}\left(s,x,y\right)\right)\text{d}y.
\]
If $B_{A}$ has bounded $y$-fibers, then we are done by Proposition
\ref{prop: strongly int on bounded cells} and Corollary \ref{cor: integration of strongly integrable}.

If $B_{A}$ has unbounded $y$-fibers, then for all $j\in J_{\text{mon}},\ T_{j}$
has the form $f_{j}\left(s,x\right)y^{\lambda_{j}\left(s\right)}\left(\log y\right)^{\mu_{j}}\text{e}^{\text{i}Q_{j}\left(x,y\right)}$,
with $\lambda_{j}\left(s\right)=\frac{\ell_{j}s+\eta_{j}}{d}$, and
for all $\mathcal{G}$-cell $S$ there is a set $J_{S}\subseteq J_{\text{mon}}$
such that for all $j\in J_{S},\ \text{Int}\left(T_{j}\chi_{B_{A}};\left(S\setminus P'\right)\times X\right)=\left\{ \left(s,x\right)\in\left(S\setminus P'\right)\times X:\ f_{j}\left(s,x\right)=0\right\} $
whereas for all $j\in J_{\text{mon}}\setminus J_{S},\ \text{Int}\left(T_{j}\chi_{B_{A}};\left(S\setminus P'\right)\times X\right)=\left(S\setminus P'\right)\times X$.
Thus, the set
\[
E:=\bigcap_{j\in J_{\text{int}}\cup J_{\text{mon}}}\text{Int}\left(T_{j}\chi_{B_{A}};\left(\Sigma\setminus P'\right)\times X\right)
\]
is of the form of the right hand side of \eqref{eq:int locus} and,
applying either Corollary \ref{cor: integration of strongly integrable}
or Proposition \ref{prop: integration of monomial generators} to
$T_{j}\chi_{B_{A}}$ and possibly enlarging $P'$, we find $H\in\mathcal{C}_{\Sigma}^{\mathcal{M},\mathcal{F}}\left(X\right)$
without poles outside $P'$ which interpolates the integral of $h\restriction A$
for all $\left(s,x\right)\in E$.

Note that $E\subseteq\mathrm{Int}\left(h\restriction A;\left(\Sigma\setminus P'\right)\times X\right)$.
It remains to show that, up to possibly enlarging $P'$, the set $E$
coincides with $\mathrm{Int}\left(h\restriction A;\left(\Sigma\setminus P'\right)\times X\right)$.
Let
\[
P_{A}=P'\cup\left\{ s\in\mathbb{C}:\ \exists i,j\in J_{\text{mon}}\text{ such that}\ i\not=j,\ \lambda_{i}\left(s\right)=\lambda_{j}\left(s\right),\ \mu_{i}=\mu_{j},\ Q_{i}=Q_{j}\right\} .
\]
By item (2.b) in Proposition \ref{prop: Splitting}, if $s\in\Sigma$
is such that $\lambda_{i}\left(s\right)=\lambda_{j}\left(s\right)$
for some $i\not=j$ such that $\mu_{i}=\mu_{j},\ Q_{i}=Q_{j}$, then
necessarily $\ell_{i}\not=\ell_{j}$, so $P_{A}\setminus P'$ is finite.

By definition of $P_{A}$, if $s\in\Sigma\setminus P_{A}$ and $i,j\in J_{\text{mon}}$
are such that $i\not=j,\ \mu_{i}=\mu_{j}$ and $\Re\left(\lambda_{i}\left(s\right)\right)=\Re\left(\lambda_{j}\left(s\right)\right)$,
then $Q_{i}=Q_{j}\Longrightarrow\Im\left(\lambda_{i}\left(s\right)\right)\not=\Im\left(\lambda_{j}\left(s\right)\right)$.

Let $\left(s_{0},x_{0}\right)\in\mathrm{Int}\left(h\restriction A;\left(\Sigma\setminus P_{A}\right)\times X\right)$
and let $S$ be the $\mathcal{G}$-cell to which $s_{0}$ belongs.
Define
\[
\rho_{j}=\Re\left(\lambda_{j}\left(s_{0}\right)\right),\ \sigma_{j}=\Im\left(\lambda_{j}\left(s_{0}\right)\right),\ p_{j}\left(y\right)=Q_{j}\left(x_{0},y\right)\in\mathbb{R}\left[y^{\frac{1}{d}}\right].
\]
Let $\left(r_{0},\nu_{0}\right)$ be the lexicographic maximum of
the set $\left\{ \left(\rho_{j},\mu_{j}\right):\ j\in J_{S}\right\} $
and let $J_{0}=\{j\in J_{S}:\ \left(\rho_{j},\mu_{j}\right)=\left(r_{0},\nu_{0}\right)\}$.
Then
\[
\sum_{j\in J_{S}}T_{j}\left(s_{0},x_{0},y\right)=y^{r_{0}}\left(\log y\right)^{\nu_{0}}\sum_{j\in J_{0}}f_{j}\left(s_{0},x_{0}\right)y^{\text{i}\sigma_{j}}\mathrm{e}^{\mathrm{i}p_{j}\left(y\right)}+\sum_{j\in J_{S}\setminus J_{0}}f_{j}\left(s_{0},x_{0}\right)y^{\rho_{j}+\mathrm{i}\sigma_{j}}\left(\log y\right)^{\mu_{j}}\mathrm{e}^{\mathrm{i}p_{j}\left(y\right)}.
\]
Since $\left(s_{0},x_{0}\right)\in\mathrm{Int}\left(h\restriction A;\left(\Sigma\setminus P_{A}\right)\times X\right)$,
it follows from \cite[Proposition 3.4 (1)]{cmrs:mellin_constructible}
that $\bigwedge_{j\in J_{0}}f_{j}\left(s_{0},x_{0}\right)=0$. By
repeating this procedure with the index set $J_{S}\setminus J_{0}$,
we end up proving that $\left(s_{0},x_{0}\right)\in E$.
\end{proof}
\begin{rem}
\label{rem: extension of the prep}Given $h\in\mathcal{C}_{\Sigma}^{\mathcal{M},\mathcal{F}}\left(X\times\mathbb{R}\right)$
and a strip $\Sigma'\supseteq\Sigma$, apply Proposition \ref{prop: Splitting}
to $h$ and consider the extension $h'$ of $h$ to $\Sigma'$. The
proof shows that Proposition \ref{prop: Splitting} applies to $h'$
with different generators $T_{j}'$ but with the same $d,\mathcal{G}$
and $P'$, by integrating by parts some of the transcendental elements
appearing in the strongly integrable generators $T_{j}$. For the
same reason, Theorem \ref{thm: interpolation} applies to $h'$ with
a different $H$ but the same $P',\mathcal{G}$.
\end{rem}

\section{Asymptotic expansions and limits\label{sec:Asymptotic-expansions-and}}

\subsection{Asymptotic expansions}\label{subsec:asympt}
In this section we study the behaviour of a function $h$, in $\mathcal{C}_{\Sigma}^{\mathcal{M},\mathcal{F}}\left(X\times\mathbb{R}\right)$,
and, in $\mathcal{C}^{\mathbb{C},\mathcal{F}}\left(X\times\mathbb{R}\right)$, 
seen as a function of the last variable $y$ with parameters ($s\in\Sigma$
and) $x\in X$. We are interested in \textquotedblleft the germ at
$+\infty$ in $y$\textquotedblright{} of $h$, hence we will work
in restriction to cells of the form \eqref{eq:A_theta=00003DB} with
unbounded $y$-fibers. As we are only interested in the behaviour
at $+\infty$ in $y$, we will often replace the cell $B$ by some
smaller cell $B'$, still of base $X$ and with unbounded $y$-fibers,
but whose lower boundary function is some analytic subanalytic function
$a'$ which satisfies that for all $x\in X,\ a\left(x\right)\leq a'\left(x\right)$.
As $X$ serves as a space of parameters, we will also often partition
$X$ into finitely many subanalytic cells and suppose, as we did in
the previous sections, that $X$ itself is one of the cells of the
partition. Finally, if $h\in\mathcal{C}_{\Sigma}^{\mathcal{M},\mathcal{F}}\left(X\times\mathbb{R}\right)$
has no poles outside some closed discrete set $P\subseteq\mathbb{C}$,
as $\Sigma$ also serves as a space of parameters, we will often replace
$P$ by some bigger closed discrete set $P'\subseteq\mathbb{C}$ such
that $P'\setminus P$ is contained in a finitely generated $\mathbb{Z}$-lattice.

Summing up, the sentence \textquotedblleft if $B$ is a cell of base
$X$ with unbounded $y$-fibers and $h\in\mathcal{C}_{\Sigma}^{\mathcal{M},\mathcal{F}}\left(X\times\mathbb{R}\right)$
has no poles outside some closed discrete set $P\subseteq\mathbb{C}$,
then, \emph{up to partitioning $X,$ shrinking $B$ and enlarging
$P$}, Property ({*}) holds for $h$\textquotedblright{} will be used
as a shorthand for the following: \emph{there are a finite partition
of $X$ into subanalytic cells $X'$ and a closed discrete set $P'\subseteq\mathbb{C}$
such that $P'\setminus P$ is contained in a finitely generated $\mathbb{Z}$-lattice,
and for each cell $X'$ there is a cell $B'\subseteq B$ of base $X'$
and with unbounded $y$-fibers such that Property ({*}) holds for
$h\restriction\left(\Sigma\setminus P'\right)\times B'$.}

\medskip{}

Our first result concerns the class $\mathcal{C}_{\Sigma}^{\mathcal{M},\mathcal{F}}\left(X\times\mathbb{R}\right)$.
\begin{thm}
\label{thm: summable}Let $B$ be as in \eqref{eq:A_theta=00003DB}
with unbounded $y$-fibers and $h\in\mathcal{C}_{\Sigma}^{\mathcal{M},\mathcal{F}}\left(B\right)$
be without poles outside some closed discrete set $P\subseteq\mathbb{C}$.
Up to partitioning $X$, shrinking $B$ and enlarging $P$, there
is a sequence $\left(T_{n}\right)_{n\in\mathbb{N}}\subseteq\mathcal{C}_{\Sigma}^{\mathcal{M},\mathcal{F}}\left(B\right)$
of generators which are monomial in $y$ such that:
\begin{enumerate}
\item For all $N\in\mathbb{N}$ there are $j_{N}\in\mathbb{N}$ and a function
$C_{N}:\left(\Sigma\setminus P\right)\times X\longrightarrow\left(0,+\infty\right)$
such that
\[
\forall\left(s,x,y\right)\in\left(\Sigma\setminus P\right)\times B,\ \left|h\left(s,x,y\right)-\sum_{j\leq j_{N}}T_{j}\left(s,x,y\right)\right|\leq C_{N}\left(s,x\right)y^{-N}.
\]
\item If $h$ is a finite sum of generators which are naive in $y$ then
we can choose the sequence $\left(T_{n}\right)_{n\in\mathbb{N}}$
so that the series $\sum_{j\in\mathbb{N}}T_{j}$ converges absolutely
to $h$.
\end{enumerate}
\end{thm}
\begin{proof}
We first prove the two statements for a function $h$ which is a finite
sum of generators which are naive in $y$: write $h=\sum_{i\in I}T_{i}$,
where $I$ is a finite index set and $T_{i}$ has the form \eqref{eq: naive in y in 5.7}.
Argueing as in the proof of Proposition \ref{prop: naive in y are sums of mon + ssint}
and using Remark \ref{rem: cell at infty}, up to partitioning $X$
and shrinking $B$, we may suppose that
\[
T_{i}\left(s,x,y\right)=f_{i}\left(s,x\right)y^{\lambda_{i}\left(s\right)}\left(\log y\right)^{\mu_{i}}\text{e}^{\text{i}Q_{i}\left(x,y\right)}\Phi_{i}\left(s,x,y\right),
\]
where the $Q_{i}\in\mathcal{S}\left(X\right)\left[y^{\frac{1}{d}}\right]$
satisfy items (2.a) and (2.b) of Proposition \ref{prop: Splitting}
and $\Phi_{i}$ is as in \eqref{eq:eq: nested prep form unb}. Using
the (absolutely convergent) series expansion of $\Phi_{i}$, we write
\begin{align}
h\left(s,x,y\right) & =\sum_{i\in I}f_{i}\left(s,x\right)\left(\log y\right)^{\mu_{i}}\text{e}^{\text{i}Q_{i}\left(x,y\right)}\sum_{k\in\mathbb{N}}\xi_{i,k}\left(s,x\right)\left(a\left(x\right)\right)^{\frac{k}{d}}y^{\lambda_{i}\left(s\right)-\frac{k}{d}}\nonumber \\
 & =\sum_{i\in I,k\in\mathbb{N}}f_{i,k}\left(s,x\right)y^{\lambda_{i,k}\left(s\right)}\left(\log y\right)^{\mu_{i}}\text{e}^{\text{i}Q_{i}\left(x,y\right)},\label{eq: convergent series}
\end{align}
where $f_{i,k}\left(s,x\right)=f_{i}\left(s,x\right)\xi_{i,k}\left(s,x\right)\left(a\left(x\right)\right)^{\frac{k}{d}}$
and $\lambda_{i,k}\left(s\right)=\lambda_{i}\left(s\right)-\frac{k}{d}$.
This proves the second statement of the theorem for $h$.

Fix $N\in\mathbb{N}$, let $\mu:=\max_{i\in I}\mu_{i},\ K:=\sup_{i\in I,s\in\Sigma}\left|\Re\left(\lambda_{i}\left(s\right)\right)\right|$
and choose $k_{0}\in\mathbb{N}$ such that $k_{0}\geq d\left(K+N+1\right)$.
Let
\begin{align*}
h_{\geq k_{0}}\left(s,x,y\right) & :=h\left(s,x,y\right)-\sum_{i\in I}f_{i}\left(s,x\right)\left(\log y\right)^{\mu_{i}}\text{e}^{\text{i}Q_{i}\left(x,y\right)}\sum_{k<k_{0}}\xi_{i,k}\left(s,x\right)\left(a\left(x\right)\right)^{\frac{k}{d}}y^{\lambda_{i}\left(s\right)-\frac{k}{d}}\\
 & =\sum_{i\in I}f_{i}\left(s,x\right)y^{\lambda_{i}\left(s\right)-\frac{k_{0}}{d}}\left(\log y\right)^{\mu_{i}}\text{e}^{\text{i}Q_{i}\left(x,y\right)}\sum_{k\geq0}\xi_{i,k+k_{0}}\left(s,x\right)\left(a\left(x\right)\right)^{\frac{k}{d}}y^{-\frac{k}{d}}.
\end{align*}
Setting $C_{N}\left(s,x\right)=\left(\frac{\left(\log a\left(x\right)\right)^{\mu}}{a\left(x\right)}+\frac{1}{\mathrm{e}}\right)\sum_{i\in I}\left|f_{i}\left(s,x\right)\right|$$\sum_{k\geq0}\left|\xi_{i,k+k_{0}}\left(s,x\right)\right|$,
by the choice of $k_{0}$ we have
\[
\left|h_{\geq k_{0}}\left(s,x,y\right)\right|\leq C_{N}\left(s,x\right)y^{-N},
\]
which proves the first statement of the theorem for $h$.

Suppose now that $h$ is not a finite sum of generators which are
naive in $y$. Apply Proposition \ref{prop: preparation, 1st version}
and Remark \ref{rem: cell at infty} to $h$: up to shrinking $B$,
this writes $h$ as a finite sum of prepared generators as in Definition
\ref{def: prepared generator}. Let $T$ be one such generator: for
our aim it is enough to show that, given $N\in\mathbb{N}$, we can
rewrite $T$ as a finite sum of generators which are either naive
in $y$ or such that we can control their asymptotics by $y^{-N}$.
For this, we revisit the proof of Proposition \ref{prop: generator sum of naive and ssint}
and argue according to the nature of the integration bounds in the
transcendental element of $T$.

Recall Notation \ref{notation: t-cell}.

If $\alpha=\beta=0$, then, up to partitioning $X,$ shrinking $B$
and enlarging $P$, we may rewrite $T$ as a a finite sum of generators
which are naive in $y$ plus a term of the form \eqref{eq: last term},
where we can expand $\Phi_{>k_{0}}$ as an absolutely convergent series
in the variable $y$ as in \eqref{eq: series expansion in y}. Permuting
integral and summation, we obtain that this last term can be written
as an absolutely convergent series of the form \eqref{eq: convergent series}.
Hence we can apply the first part of the proof to this last term.

If $\alpha>0$, then chose $\ell_{0}\in\mathbb{N}$ such that for
all $s\in\Sigma,\ \Re\left(\lambda\left(s\right)\right)+\frac{\alpha}{d}\left(\Re\left(\varrho\left(s\right)\right)-\ell_{0}+1\right)+\frac{\Delta}{d}<-\left(N+1\right)$.
If we integrate by parts as in Lemma \ref{lem: int by parts} $\ell_{0}$
times, then we create finitely many terms which are naive in $y$
and an integral rest of the form
\begin{equation}
R\left(s,x,y\right)=g\left(s,x\right)y^{\lambda\left(s\right)}\left(\log y\right)^{\mu}\text{e}^{\text{i}\varphi\left(x,y\right)}\int_{\widetilde{a}\left(x,y\right)}^{\widetilde{b}\left(x,y\right)}t^{\varrho\left(s\right)-\ell_{0}}\left(\log t\right)^{\nu}\Phi\left(s,x,y,t\right)\text{e}^{\sigma\text{i}t}\text{d}t,\label{eq:rest}
\end{equation}
where \eqref{eq:cond on exp 0} is satisfied.

If $\widetilde{b}=+\infty$, then
\begin{align*}
\left|R\left(s,x,y\right)\right| & \leq\widetilde{C}_{N}\left(s,x\right)y^{\Re\left(\lambda\left(s\right)\right)+\frac{\alpha}{d}\left(\Re\left(\varrho\left(s\right)\right)-\ell_{0}+1\right)}\left(\log y\right)^{\mu+\nu}\\
 & \leq C_{N}\left(s,x\right)y^{-N},
\end{align*}
for suitable positive functions $\widetilde{C}_{N},C_{N}$.

If $\widetilde{b}<+\infty$, then
\begin{align*}
\left|R\left(s,x,y\right)\right| & \leq\widetilde{C}_{N}\left(s,x\right)y^{\Re\left(\lambda\left(s\right)\right)+\frac{\alpha}{d}\left(\Re\left(\varrho\left(s\right)\right)-\ell_{0}\right)+\frac{\Delta}{d}}\left(\log y\right)^{\mu+\nu}\\
 & \leq C_{N}\left(s,x\right)y^{-N},
\end{align*}
for suitable positive functions $\widetilde{C}_{N},C_{N}$.

If $\alpha=0$ and $\beta>0$ then choose $k_{0}\in\mathbb{N}$ such
that for all $s\in\Sigma,\ \Re\left(\lambda\left(s\right)\right)-\frac{\beta}{d}\left(\frac{k_{0}+1}{d}\right)+\frac{\Delta}{d}<-\left(N+1\right)$
and $\ell_{0}\in\mathbb{N}$ such that $\ell_{0}>\Re\left(\varrho\left(s\right)\right)+\frac{k_{0}}{d}+1$.
Then \eqref{eq: condition on exponents} and \eqref{eq: cond on exp 2}
are satisfied if we replace $\varrho\left(s\right)$ by $\varrho\left(s\right)-\ell_{0}$.
If we integrate by parts $\ell_{0}$ times, then we create finitely
many terms which are naive in $y$ and an integral rest of the form
\eqref{eq:rest}. Proceeding as in the third part of the proof of
Proposition \ref{prop: generator sum of naive and ssint}, we are
left to deal with a term $R_{>k_{0}}$ of the form \eqref{eq: last term-1}
which satisfies
\begin{align*}
\left|R_{>k_{0}}\left(s,x,y\right)\right| & \leq\widetilde{C}_{N}\left(s,x\right)y^{\Re\left(\lambda\left(s\right)\right)-\frac{\beta}{d}\left(\frac{k_{0}+1}{d}\right)+\frac{\Delta}{d}}\left(\log y\right)^{\mu+\nu}\\
 & \leq C_{N}\left(s,x\right)y^{-N},
\end{align*}
for suitable positive functions $\widetilde{C}_{N},C_{N}$.
\end{proof}
Our next goal is to concentrate on the subclass $\mathcal{C}^{\mathbb{C},\mathcal{F}}\left(X\times\mathbb{R}\right)$
and deduce from Theorem \ref{thm: summable} a more precise result
on the asymptotic behaviour of $h$ in $y$, in the sense of \cite[Definition 7.1]{ccmrs:integration-oscillatory}
but uniformly in the variables $x\in X$. 

First, we restate and improve Theorem \ref{thm: summable} for functions
in the class $\mathcal{C}^{\mathbb{C},\mathcal{F}}\left(X\times\mathbb{R}\right)$.
\begin{defn}
\label{def: space of coefficients}Let $\mathscr{E}\subseteq\mathcal{C}^{\mathbb{C},\mathcal{F}}\left(X\times\left(0,+\infty\right)\right)$
be the $\mathbb{C}$-vector space of all functions of the form
\[
E\left(x,y\right)=\sum_{j\in J}f_{j}\left(x\right)\text{e}^{\text{i}\zeta_{j}\left(x,y\right)},
\]
where $J$ is a finite index set, $f_{j}\in\mathcal{C}^{\mathbb{C},\mathcal{F}}\left(X\right),\ \zeta_{j}\left(x,y\right)=\sigma_{j}\log y+Q_{j}\left(x,y\right)$
with $\sigma_{j}\in\mathbb{R},\ Q_{j}\in\mathcal{S}\left(X\right)\left[y^{\frac{1}{d}}\right]$.
We require moreover that for all $j\in J,\text{\ for all }x\in X,\ Q_{j}\left(x,0\right)=0,$
for all $i,j\in J$ , either $Q_{i}=Q_{j}$ on $X$ or $Q_{i}\not=Q_{j}$
on $X$ and if $i\not=j$ then for all $x\in X$, the functions $y\longmapsto\zeta_{i}\left(x,y\right)$
and $y\longmapsto\zeta_{j}\left(x,y\right)$ are distinct.
\end{defn}
\begin{rem}
\label{rem: space of coeff}By \cite[Proposition 3.4 (2)]{cmrs:mellin_constructible},
if $E\in\mathscr{E}\setminus\left\{ 0\right\} $ then for all $x\in X$,
either $y\longmapsto E\left(x,y\right)$ is identically zero or there
exist $\varepsilon\left(x\right)>0$ and a sequence $\left(y_{n}\right)_{n\in\mathbb{N}}$
such that $\lim_{n\longrightarrow+\infty}y_{n}=+\infty$ and for all
for all $n\in\mathbb{N},\ \left|E\left(x,y_{n}\right)\right|>\varepsilon\left(x\right)$.
\end{rem}
\begin{defn}
\label{def: asymptotic exp}A function $h\in\mathcal{C}^{\mathbb{C},\mathcal{F}}\left(X\times\mathbb{R}\right)$
has a \emph{power-log asymptotic expansion with coefficients in $\mathscr{E}$
}if there are a collection $\left\{ E_{n}:\ n\in\mathbb{N}\right\} \subseteq\mathscr{E}$,
a sequence $\left(r_{n},\nu_{n}\right)_{n\in\mathbb{N}}\subseteq\mathbb{R}\times\mathbb{N}$
which is strictly decreasing with respect to the lexicographic order,
a cell $B$ as in \eqref{eq:A_theta=00003DB} with unbounded $y$-fibers
and for all $N\in\mathbb{N}$, a function $C_{N}:X\longrightarrow\left(0,+\infty\right)$
such that, for all $\left(x,y\right)\in B$,
\begin{equation}
\left|h\left(x,y\right)-\sum_{n<N}E_{n}\left(x,y\right)y^{r_{n}}\left(\log y\right)^{\nu_{n}}\right|\leq C_{N}\left(x\right)y^{r_{N}}\left(\log y\right)^{\nu_{N}}.\label{eq:asymp exp}
\end{equation}
If moreover the series $\sum_{n\in\mathbb{N}}E_{n}\left(x,y\right)y^{r_{n}}\left(\log y\right)^{\nu_{n}}$
converges absolutely to $h$, then we say that $h$ has a \emph{convergent
}power-log asymptotic expansion with coefficients in $\mathscr{E}$.
\end{defn}
Note that the sequence of real functions $\left(g_{n}\left(y\right)\right)_{n\in\mathbb{N}}=\left(y^{r_{n}}\left(\log y\right)^{\nu_{n}}\right)_{n\in\mathbb{N}}$
forms an \emph{asymptotic scale} at $+\infty$ in the sense that,
for all $n\in\mathbb{N},$ $\lim_{y\longrightarrow+\infty}\frac{g_{n+1}\left(y\right)}{g_{n}\left(y\right)}=0$.

Recall the definition of the system $\mathcal{C}^{\mathbb{C}}$ of
power-constructible functions and that of the system $\mathcal{C}^{\mathbb{C},\text{i}\mathcal{S}}$,
given in Section \ref{subsec: Fourier transf}.
\begin{defn}
\label{def: naive in 7}Let $\mathcal{C}_{\text{naive}}^{\mathbb{C},\mathcal{F}}\left(X\times\mathbb{R}\right)$
be the additive group generated by the generators which are naive
in $y$, i.e. of the form
\[
\gamma g\text{e}^{\text{i}\varphi}\ \ \ \left(\gamma\in\mathcal{C}^{\mathbb{C},\mathcal{F}}\left(X\right),\ g\in\mathcal{C}^{\mathbb{C}}\left(X\times\mathbb{R}\right),\ \varphi\in\mathcal{S}\left(X\times\mathbb{R}\right)\right).
\]
\end{defn}
Note that $\mathcal{C}_{\text{naive}}^{\mathbb{C},\mathcal{F}}\left(X\times\mathbb{R}\right)$
is a $\mathbb{C}$-algebra and
\begin{equation}
\mathcal{C}^{\mathbb{C}}\left(X\times\mathbb{R}\right)\subseteq\mathcal{C}^{\mathbb{C},\text{i}\mathcal{S}}\left(X\times\mathbb{R}\right)\subseteq\mathcal{C}_{\text{naive}}^{\mathbb{C},\mathcal{F}}\left(X\times\mathbb{R}\right)\subseteq\mathcal{C}^{\mathbb{C},\mathcal{F}}\left(X\times\mathbb{R}\right).\label{eq: inclusions}
\end{equation}

\begin{thm}
\label{thm: asympt exp}Every $h\in\mathcal{C}^{\mathbb{C},\mathcal{F}}\left(X\times\mathbb{R}\right)$
has, up to partitioning $X$, a power-log asymptotic expansion with
coefficients in $\mathscr{E}$. If moreover $h\in\mathcal{C}_{\mathrm{naive}}^{\mathbb{C},\mathcal{F}}\left(X\times\mathbb{R}\right)$,
then such an asymptotic expansion is convergent.
\end{thm}
\begin{proof}
Suppose first that $h\in\mathcal{C}_{\text{naive}}^{\mathbb{C},\mathcal{F}}\left(X\times\mathbb{R}\right)$,
so that, up to partitioning $X$ and on some cell $B$ with unbounded
$y$-fibers, $h$ can be written as in \eqref{eq: convergent series},
where the functions $f_{i,k}$ only depend on the variables $x$ and
$\lambda_{i,k}=\lambda_{i}-\frac{k}{d}\in\mathbb{C}$. Let $\rho_{i,k}=\Re\left(\lambda_{i}\right)-\frac{k}{d}$,
$\sigma_{i}=\Im\left(\lambda_{i}\right)$ and define $\zeta_{i}\left(x,y\right)=\sigma_{i}\log y+Q_{i}\left(x,y\right)$.
Hence we can write $h$ as the sum of the absolutely convergent series
of functions
\[
\sum_{\left(i,k\right)\in I\times\mathbb{N}}f_{i,k}\left(x\right)y^{\rho_{i,k}}\left(\log y\right)^{\mu_{i}}\text{e}^{\text{i}\zeta_{i}\left(x,y\right)}.
\]
The set $\left\{ \rho_{i,k}:\ i\in I,k\in\mathbb{N}\right\} $ is
contained in a finitely generated $\mathbb{Z}$-lattice and, since
$I$ is finite, so is the set $\left\{ \mu_{i}:\ i\in I\right\} $.
Hence the set
\[
J=\left\{ \left(r,\nu\right):\ \exists\left(i,k\right)\in I\times\mathbb{N}\text{ s.t. }\left(\rho_{i,k},\mu_{i}\right)=\left(r,\nu\right)\right\}
\]
 is countable and, for $\left(r,\nu\right)\in J$, the set $J_{\left(r,\nu\right)}=\left\{ \left(i,k\right)\in I\times\mathbb{N}:\ \rho_{i,k}=r,\ \mu_{i}=\nu\right\} $
is finite. Fix a bijection
\[
\mathbb{N}\ni n\longmapsto\left(r_{n},\nu_{n}\right)\in J
\]
which is decreasing with respect to the lexicographic order and define
\[
E_{n}\left(x,y\right)=\sum_{\left(i,k\right)\in J_{\left(r_{n},\nu_{n}\right)}}f_{i,k}\left(x\right)\text{e}^{\text{i}\zeta_{i}\left(x,y\right)}.
\]
These are the coefficients of a convergent power-log asymptotic expansion
of $h$ in the asymptotic scale $\left\{ y^{r_{n}}\left(\log y\right)^{\nu_{n}}:\ n\in\mathbb{N}\right\} $.

Suppose now that $h\not\in\mathcal{C}_{\text{naive}}^{\mathbb{C},\mathcal{F}}\left(X\times\mathbb{R}\right)$.
Revisiting the proof of Theorem \ref{thm: summable}, given $N\in\mathbb{N}$,
we may write $h$ as a finite sum of generators which are either naive
in $y$ (and hence have a convergent power-log asymptotic expansion
in some common asymptotic scale $\left(y^{r_{n}}\left(\log y\right)^{\nu_{n}}\right)_{n\in\mathbb{N}}$
with coefficients $\left\{ E_{n}:\ n\in\mathbb{N}\right\} \subseteq\mathscr{E}$)
or whose module is bounded $C_{N}\left(x\right)y^{\lfloor r_{N}\rfloor-1}$,
where $C_{N}$ is some positive function in $\mathcal{C}^{\mathbb{C},\mathcal{F}}\left(X\right)$.
In particular, $h$ has a (not necessarily convergent) power-log asymptotic
expansion as in \eqref{eq:asymp exp}.
\end{proof}
\begin{rem}
\label{rem: unique asympt exp}Argueing as in \cite[Lemma 7.2]{ccmrs:integration-oscillatory}
and using Remark \ref{rem: space of coeff}, one sees that if $h$
has a power-log asymptotic expansion in a certain power-log asymptotic
scale and with coefficients in $\mathscr{E}$, then its coefficients
are uniquely determined. Note that the proof of Theorem \ref{thm: asympt exp}
shows that the power-log asymptotic scales $\left(y^{r_{n}}\left(\log y\right)^{\nu_{n}}\right)_{n\in\mathbb{N}}$
appearing in the asymptotic expansions of functions in $\mathcal{C}^{\mathbb{C},\mathcal{F}}$
have the property that the sequence $\left(r_{n},\nu_{n}\right)_{n\in\mathbb{N}}$
has the same order type as $\omega$. In particular, the union of
two such asymptotic scales is again an asymptotic scale of the same
type, so a function in $\mathcal{C}^{\mathbb{C},\mathcal{F}}$ cannot
have two different asymptotic expansions in two different power-log
asymptotic scales.
\end{rem}
\begin{cor}
\label{cor: naive not enough}The systems $\mathcal{C}^{\mathbb{C},\mathrm{i}\mathcal{S}}$
and $\mathcal{C}^{\mathcal{M},\mathrm{i}\mathcal{S}}$ are not stable
under parametric integration.
\end{cor}
\begin{proof}
We give two examples of functions which are in $\mathcal{C}^{\mathbb{C},\mathcal{F}}\left(\mathbb{R}\right)$
but not in $\mathcal{C}^{\mathcal{M},\text{i}\mathcal{S}}\left(\mathbb{R}\right)$.

The function $f:y\longmapsto\text{e}^{-\left|y\right|}$ belongs to
$\mathcal{C}^{\mathbb{C},\mathcal{F}}\left(\mathbb{R}\right)$, since
it can be obtained as a parametric integral of a function in $\mathcal{C}^{\mathbb{C},\mathrm{i}\mathcal{S}}\left(\mathbb{R}^{2}\right)$
(it is the inverse Fourier transform of the semialgebraic function
$t\longmapsto\frac{2}{1+4\pi^{2}t^{2}}$, see for example \cite{gasquet-witomski:fourier-analysis}).
If $f$ were in $\mathcal{C}^{\mathcal{M},\text{i}\mathcal{S}}\left(\mathbb{R}\right)$
then it would also be in $\mathcal{C}^{\mathbb{C},\text{i}\mathcal{S}}\left(\mathbb{R}\right)$
and by Theorem \ref{thm: asympt exp} $f$ would have a convergent
power-log asymptotic expansion with coefficients in $\mathscr{E}$.
Now, argueing as in \cite[Example 7.4]{ccmrs:integration-oscillatory}
and using Remark \ref{rem: space of coeff}, one sees that no exponentially
flat function can have such a convergent power-log asymptotic expansion.

Now consider the function
\[
\text{Si}\left(y\right)=\begin{cases}
\int_{0}^{y}\frac{\text{e}^{\text{i}t}-\text{e}^{-\text{i}t}}{2\text{i}t}\text{d}t & y>0\\
0 & y\leq0
\end{cases}
\]
which is obtained as a parametric integral of a function in $\mathcal{C}^{\mathbb{C},\mathrm{i}\mathcal{S}}\left(\mathbb{R}^{2}\right)$.
It is well-known that $\text{Si}$ has a divergent power-log asymptotic
expansion with coefficients in $\mathcal{E}$ (see \cite{abramowitz:handbook_mathematical_functions}
and \cite[Example 7.5]{ccmrs:integration-oscillatory}. By Theorem
\ref{thm: asympt exp} and Remark \ref{rem: unique asympt exp}, $\text{Si}\notin\mathcal{C}^{\mathcal{M},\text{i}\mathcal{S}}\left(\mathbb{R}\right)$.
\end{proof}
\begin{rem}
\label{rem:flat}Let $X\subseteq\mathbb{R}^{m}$ be a subanalytic
open set and $K\subseteq X$ be a compact subanalytic subset. It is
possible to construct a $C^{\infty}$ function $\eta\in\mathcal{C}^{\mathbb{C},\mathcal{F}}(X)$
such that $\eta\left(X\right)\subseteq\left[0,1\right]$ and $\eta\equiv1$
on a neighbourhood of $K$ in $X$ (in particular $\mathcal{C}^{\mathbb{C},\mathcal{F}}(X)$
contains smooth functions with compact support). One way to do this
is to consider the function
\[
\nu:x\mapsto
f(1-\|x\|^{2}),
\]
where $\|\cdot\|$ is the Euclidean norm in $\mathbb{R}^{m}$ and
\[
f:t\mapsto\left\{ \begin{array}{ll}
\mathrm{e}^{-\frac{1}{t}} & \text{if \ensuremath{t>0}}\\
0 & \text{if \ensuremath{t\leq0}}
\end{array}\right..
\]
Note that, considering the first example in the proof of Corollary
\ref{cor: naive not enough} and using the stability under right-composition
with subanalytic functions observed in Remark \ref{rem: 1 is a transcendental},
we obtain that $\nu\in\mathcal{C}^{\mathbb{C},\mathcal{F}}(X)$. We
can then define $\eta$ as the convolution of $\nu$ with the subanalytic
characteristic function of a sufficiently small tubular neighbourhood
of $K$ in $X$ (see for instance \cite[Theorem 1.4.1]{Hor1}), and
thus obtain that $\nu\in\mathcal{C}^{\mathbb{C},\mathcal{F}}(X)$
by Theorem \ref{thm: stability of C^F}.
\end{rem}
We have at our disposal several results concerning the asymptotics
at infinity of integral transforms, and in particular of Fourier and
Mellin transforms, of functions with support in $[0,+\infty)$ having
an asymptotic expansions at the origin in the scale $\{x^{\alpha}\log^{\beta}:\alpha,\beta\in\mathbb{R}\}$
(see for instance \cite{BleiHand,Wong,WongLin}). In this situation,
the integral transforms have an asymptotic expansion at $+\infty$
in the same power-log scale. On the other hand, to our knowledge very
little known beyond this scale, in particular with respect to asymptotic
scales detecting exponentially small terms (see \cite{Lombardi}),
a question that is relevant to the class $\mathcal{C}^{\mathrm{exp}}$
of \cite{ccmrs:integration-oscillatory} and to our class $\mathcal{C}^{\mathbb{C},\mathcal{F}}$
by Remark \ref{rem:flat}, but which seems to require new tools.

\subsection{Pointwise limits}

In this section we prove the stability of the class $\mathbf{\mathcal{C}^{\mathbb{C},\mathcal{F}}}$
under pointwise limits.
\begin{notation}
\label{not: limits}For $X\subseteq{\mathbb{R}}^{m}$ and $h\colon X\times{\mathbb{R}}\to{\mathbb{C}}$,
let
\[
\text{Lim}\left(h,X\right):=\{x\in X:\ \text{\ensuremath{\lim_{y\to+\infty}h\left(x,y\right)\ }exists}\}.
\]
\end{notation}
\begin{thm}
\label{prop:limits} Let $h\in\mathcal{C}^{\mathbb{C},\mathcal{F}}\left(X\times{\mathbb{R}}\right)$. There exist $f,g\in\mathcal{C}^{\mathbb{C},\mathcal{F}}\left(X\right)$
such that
\[
\mathrm{Lim}\left(h,X\right)=\{x\in X:\ f\left(x\right)=0\}
\]
and such that for all $x\in\mathrm{Lim}\left(h,X\right)$,
\[
\lim_{y\to+\infty}h\left(x,y\right)=g\left(x\right).
\]
\end{thm}
\begin{proof}
Apply Proposition \ref{prop: Splitting} to $h$ and concentrate on
a cell $A$ with unbounded $y$-fibers (so that, by Remark \ref{rem: cell at infty},
$A=B_{A}$ and $\Pi_{A}$ is the identity map). By condition (1),
the prepared generators $T_{j}$ which are strongly integrable tend
indeed to a limit and this limit is zero. Hence we may suppose that
$h=\sum_{i\in I}T_{i}$ for some finite index set $I$, where each
$T_{i}$ is a monomial generator of the form
\[
T_{i}\left(x,y\right)=f_{i}\left(x\right)y^{\lambda_{i}}\left(\log y\right)^{\mu_{i}}\text{e}^{\text{i}Q_{i}\left(x,y\right)},
\]
where $\lambda_{i}\in\mathbb{C}$ with $\Re\left(\lambda_{i}\right)\geq0$.
Write the finite set
\[
J=\left\{ \left(r,\nu\right)\in[0,+\infty)\times\mathbb{N}:\ \exists i\in I\text{ s.t. }\Re\left(\lambda_{i}\right)=r,\ \mu_{i}=\nu\right\}
\]
as
\[
J=\left\{ \left(r_{0},\nu_{0}\right),\ldots,\left(r_{N},\nu_{N}\right)\right\}
\]
for some $N\in\mathbb{N}$, and suppose that $\left(r_{0},\nu_{0}\right)>\ldots>\left(r_{N},\nu_{N}\right)$
with respect to the lexicographic order. For $j=0,\ldots,N$, define
\[
J_{j}=\left\{ i\in I:\ \Re\left(\lambda_{i}\right)=r_{j},\ \mu_{i}=\nu_{j}\right\} .
\]
Writing $\zeta_{i}\left(x,y\right)=\Im\left(\lambda_{i}\right)\log y+Q_{i}\left(x,y\right)$
and $E_{j}\left(x,y\right)=\sum_{i\in J_{j}}f_{i}\left(x\right)\text{e}^{\text{i}\zeta_{i}\left(x,y\right)}$,
we obtain that
\[
h\left(x,y\right)=\sum_{j\leq N}E_{j}\left(x,y\right)y^{r_{j}}\left(\log y\right)^{\nu_{j}}.
\]

Let $x\in\mathrm{Lim}\left(h,X\right)$. Suppose that there exists
$i\in J_{0}$ such that $f_{i}\left(x\right)=0$. Then, by Condition
(2.a) of Proposition \ref{prop: Splitting} and by \cite[Proposition 3.4(2)]{cmrs:mellin_constructible}
we have necessarily that $r_{0}=\nu_{0}=0$. Hence we may suppose
that $N=0$ and $h\left(x,y\right)=\sum_{i\in J_{0}}f_{i}\left(x\right)\text{e}^{\text{i}\zeta_{i}\left(x,y\right)}$.
If there exists $i\in J_{0}$ such that either $\Im\left(\lambda_{i}\right)\not=0$
or $Q_{i}\left(x,y\right)\not=0$, then by \cite[Proposition 3.4(3)]{cmrs:mellin_constructible}
we obtain that $f_{i}\left(x\right)=0$. Notice that there is at most
one index $i_{0}\in J_{0}$ such that $\zeta_{i_{0}}\left(x,y\right)=0$.
To conclude, we define and
\[
f\left(x\right)=\sum_{i\in\widetilde{I}}\left|f_{i}\left(x\right)\right|^{2}.
\]
As the class $\mathcal{C}^{\mathbb{C},\mathcal{F}}$ is clearly stable
under complex conjugation, $f$ belongs to $\mathcal{C}^{\mathbb{C},\mathcal{F}}$.
Finally, define $g\left(x\right)=f_{i_{0}}\left(x\right)$, if there
exists a (necessarily unique) index $i_{0}\in I$ such that $\Re\left(\lambda_{i_{0}}\right)=\mu_{i_{0}}=0\text{ and }\zeta_{i_{0}}\left(x,y\right)=0$,
and $g=0$ otherwise.
\end{proof}

\section{The  Fourier-Plancherel transform and $\text{L}^p$-limits\label{sec:Plancherel}}
 
We deal here with the question of parametric families of functions
of $\mathcal{C}^{\mathbb{C},\mathcal{F}}$, to provide non-compensation
arguments in this framework, useful for $\text{L}^{p}$-completeness
and the $\text{L}^{2}$-Fourier transform, also known as the Plancherel transform, or the Fourier-Plancherel transform.  In \cite[Section 8]{ccmrs:integration-oscillatory}, this is treated in
the case of the system $\mathcal{C}^{\mathrm{exp}}$, which we generalize to our setting of $\mathcal{C}^{\mathbb{C},\mathcal{F}}$. 

We recall from \cite{ccmrs:integration-oscillatory} what it means for a \emph{family} of functions to be \emph{continuously
uniformly distributed modulo 1}, which extends notions from \cite{Weyl:Gleichverteilung,kuipers_niederreiter:uniform_distributions_sequences}.

Let $X$ be a nonempty subset of $\mathbb{R}^{m}$, $N\in\mathbb{N}\setminus\{0\}$
and $\rho=\left(\rho_{1},\ldots,\rho_{N}\right):X\times[0,+\infty)\to\mathbb{R}^{N}$
be a map. If $I_{1},\ldots,I_{N}\subseteq\mathbb{R}$ are bounded
intervals with nonempty interior, we denote by $I$ the box $\prod_{j=1}^{N}I_{j}$
and, for $T\ge0$ and $x\in X$, we let
\[
W_{\rho,I,T}^{x}:=\{t\in[0,T]:\{\rho\left(x,t\right)\}\in I\},
\]
where $\{\rho\left(x,t\right)\}$ denotes the vector of fractional
parts $\left(\{\rho_{1}\left(x,t\right)\},\ldots,\{\rho_{N}\left(x,t\right)\}\right)$
of the components of $\rho$, that is to say for $x\in\mathbb{R}$,
${x}=x-\lfloor x\rfloor$.
\begin{defn}
\label{defn:CUD:param} With this notation, we say that the map $\rho$
is \emph{continuously uniformly distributed modulo $1$ on $X$} (abbreviated
as \emph{c.u.d.\ mod $1$ on $X$}) if for every box $I\subseteq[0,1)^{N}$,
\[
\lim_{T\to+\infty}\sup_{x\in X}\frac{\mathrm{vol}_{1}\left(W_{\rho,I,T}^{x}\right)}{T}=\mathrm{vol}_{N}\left(I\right).
\]
\end{defn}
We will use the c.u.d. mod 1 property in Lemma \ref{lem:Param:Dovetail}.
In our context we have to deal with sums of complex exponential functions,
with phase of type $\varphi(x,y)=\sigma\log y+p(x,y)$, where $p$
is a polynomial in $y$, or more exactly in $y^{\frac{1}{d}}$, for
some positive integer $d$, and with coefficients some functions of
the variable $x$. We cannot directly use the c.u.d. mod 1 property
for those phases, since $\log y$ is not a c.u.d. mod 1 function (although
$\varphi$ turns out to be c.u.d. mod 1 when $p$ is not constant).
To overcome this technical difficulty, we compose $\varphi$ with
$(x,y)=(x,\mathrm{e}^{t})$ to obtain a phase of type $\phi(x,t)=\sigma t+p(x,\mathrm{e}^{t})$.
Now we can use the c.u.d. mod 1 property, the change of variables
$y=\mathrm{e}^{t}$ being harmless in view of the conclusion of Lemma
\ref{lem:Param:Dovetail}.
\begin{prop}
\label{prop:WeylToCUD} Let $\ell,p\in\mathbb{N}$ and $X$ a compact
subset of $\mathbb{R}^{m}$. Consider a map $\rho=\left(\phi_{1},\cdots,\phi_{\ell},\rho_{1},\ldots,\rho_{p}\right):X\times[0,+\infty)\to\mathbb{R}^{\ell+p}$,
where for each $i\in\{1,\ldots,\ell\}$, for each $j\in\{1,\ldots,p\}$,
\[
\phi_{i}(x,t)=g_{i}\left(x\right)\mathrm{e}^{\frac{\delta_{i}}{d}t},\ \ \rho_{j}\left(x,t\right)=\sigma_{j}t
\]
for some continuous (nonzero) functions $g_{i}:X\to\mathbb{R}$, positive
integers $d$ and $\delta_{i}$, and for $\sigma_{j}$ real numbers.
Assume that for each $x\in X$, the functions $t\mapsto\phi_{1}(x,t),\ldots,t\mapsto\phi_{\ell}(x,t),t\mapsto\rho_{1}\left(x,t\right),\ldots,t\mapsto\rho_{p}\left(x,t\right)$
are linearly independent over $\mathbb{Q}$. Then $\rho$ is c.u.d.\ mod
$1$ on $X$.
\end{prop}
Before proving the proposition, we make a remark.
\begin{rem}
\label{rem:independence through J_k} In the notation of Proposition
\ref{prop:WeylToCUD}, let $\delta=\max\{\delta_{1},\ldots,\delta_{\ell}\},$
and for each $k\in\{1,\ldots,\delta\}$, let $I_{k}=\{i\in\{1,\ldots,\ell\}:\delta_{i}=k\}.$
The assumption that $t\mapsto\phi_{1}(x,t),\ldots,t\mapsto\phi_{\ell}(x,t),t\mapsto\rho_{1}\left(x,t\right),\ldots,t\mapsto\rho_{p}\left(x,t\right)$
are linearly independent over $\mathbb{Q}$ for each $x\in X$ is
equivalent to saying that for each $k\in\{1,\ldots,\delta\}$ and
$x\in X$, the family of real numbers $\left(g_{i}\left(x\right)\right)_{i\in I_{k}}$
is linearly independent over $\mathbb{Q}$, and that the family of
real numbers $(\sigma_{1},\ldots,\sigma_{m})$ is linearly independent
over $\mathbb{Q}$.
\end{rem}
\begin{proof}[Proof of Proposition \ref{prop:WeylToCUD}]
We may assume that $\ell\ge1$, since if $\ell=0$ and the family
of linear maps $(t\mapsto\sigma_{1}t,\ldots,t\mapsto\sigma_{p}t)$
is linearly independent over $\mathbb{Q}$, then the map $\rho$ is
well-known to be c.u.d. mod 1 (see \cite[Exercise 9.27]{kuipers_niederreiter:uniform_distributions_sequences}).

Assuming $\ell\ge1$, the proof consists in satisfying the version
in families of the criterion \eqref{eq.uniform Weyl criterion} (see
\cite[Theorem 9.9]{kuipers_niederreiter:uniform_distributions_sequences}
for the basic case, and \cite[Proposition 8.7]{ccmrs:integration-oscillatory}
for the version in families): for any $h=(\alpha_{1},\ldots,\alpha_{\ell},\beta_{1},\ldots,\beta_{p})\in\mathbb{Z}^{\ell+p}$,
$h\not=0$, 
\begin{equation}
\lim_{T\to+\infty}\dfrac{1}{T}\int_{1}^{T}\mathrm{e}^{2\pi\mathrm{i}\langle h,\rho(x,t)\rangle}\ \mathrm{d}t=0,\label{eq.uniform Weyl criterion}
\end{equation}
uniformly in $x\in X$. We prove in fact that for some $T_{0}\ge1$,
$J(T)={\displaystyle \int_{T_{0}}^{T}\mathrm{e}^{2\pi\mathrm{i}\langle h,\rho(x,t)\rangle}\ \mathrm{d}t}$
is bounded from above by a constant not depending on $x\in X$. To
do this, we follow the proof of \cite[Proposition 3.4]{cmrs:mellin_constructible}:
we fix $h\in\mathbb{Z}^{\ell+p}$, define, in the notation of Remark
\ref{rem:independence through J_k}, $G(x)=\sum_{i\in I_{\delta}}\alpha_{i}g_{i}(x)$,
$G_{k}\left(x\right)=\sum_{j\in I_{k}}\alpha_{i}g_{i}(x)$, for $k\in\{1,\ldots,\delta-1\}$,
and $\sigma=\sum_{j=1}^{m}\beta_{j}\sigma_{j}$, and we write
\begin{equation}
H(x,t)=\frac{\langle h,\rho(x,t)\rangle}{G(x)}=\mathrm{e}^{\frac{\delta}{d}t}+\frac{G_{\delta-1}(x)}{G(x)}\mathrm{e}^{\frac{\delta-1}{d}t}+\cdots+\frac{G_{1}(x)}{G(x)}\mathrm{e}^{\frac{t}{d}}+\frac{\sigma t}{G(x)}.\label{eq.prod}
\end{equation}
For simplicity we assume that $I_{k}\not=\emptyset$, for $k=1,\ldots,\delta$,
which is harmless. Note that the continuous functions $G_{1},\ldots,G_{\delta-1}$
are bounded from above on $X$. By Remark \ref{rem:independence through J_k}
the function $G$ has no zero in $X$, since for each $x\in X$, the
components of $\rho$ are linearly independent over $\mathbb{Q}$,
and therefore, again by continuity on $X$, $\vert G\vert$ is bounded
below by a constant $C>0$ on $X$. It follows that we can fix $T_{0}$
sufficiently large so that, for each $x\in X$, $t\mapsto H(x,t)$
and $t\mapsto\dfrac{\partial H}{\partial t}(x,t)$ are strictly increasing
(to $+\infty$) on $[T_{0},+\infty)$, and we can assume that for
all $x\in X$, $\dfrac{\partial H}{\partial t}(x,T_{0})\ge1$.

Denoting for each $x\in X$, $t=V(x,u)$ the inverse of $u=H(x,t)$,
we perform the change of variables $u=H(x,t)$ in $J(T)$ to obtain
\[
J(T)=\int_{T_{0}}^{T}\mathrm{e}^{2\pi\mathrm{i}G(x)H(x,t)}\ \mathrm{d}t=\int_{H(x,T_{0})}^{H(x,T)}\dfrac{\mathrm{e}^{2\pi\mathrm{i}G(x)u}}{\frac{\partial H}{\partial t}(x,V(x,u))}\ \mathrm{d}u.
\]
Now, since $u\mapsto\dfrac{1}{\frac{\partial H}{\partial t}(x,V(x,u))}$
is monotonically decreasing on $[H(x,T_{0}),+\infty)$, by the Second
Mean Value Theorem for integrals applied to the real part of $J(T)$,
we have
\[
\Re(J(T))=\dfrac{1}{\frac{\partial H}{\partial t}(x,T_{0})}\int_{H(x,T_{0})}^{\tau}\cos(2\pi G(x)u)\ \mathrm{d}u,
\]
for some $\tau\in(H(x,T_{0}),H(x,T)]$. Since $u\mapsto\cos\left(2\pi G(x)u\right)$
has an antiderivative with period $\frac{1}{\left|G(x)\right|}$,
and since $\frac{1}{\left|G(x)\right|}\leq\frac{1}{C}$, the integral
on the right side may be replaced with an integral over an interval
of length at most $\frac{1}{C}$. From the fact that $\dfrac{\partial H}{\partial t}(x,T_{0})\ge1$,
for all $x\in X$, it follows that the real part of $J(T)$ is uniformly
bounded from above with respect to $x\in X$, and so is the imaginary
part of $J(T)$ by the same computation.
\end{proof}
We now introduce some notation for Lemma \ref{lem:Param:OscPrep}.
Consider a cell
\[
A=\{\left(x,t\right):x\in A_{0},t>a\left(x\right)\},
\]
where $A_{0}$ is connected and open in $\mathbb{R}^{m}$. Let $f:A\to\mathbb{C}$
be defined by
\[
f\left(x,t\right)=\sum_{j=1}^{n}f_{j}\left(x\right)\mathrm{e}^{\mathrm{i}(\sigma_{j}t+p_{j}(x,\mathrm{e}^{t}))},
\]
where $\sigma_{1},\ldots,\sigma_{n}$ are real numbers, $\left(f_{1},\ldots,f_{n}\right)$
is a family of (nonzero) analytic functions in $\mathcal{C}^{\mathbb{C},\mathcal{F}}\left(A_{0}\right)$,
$p_{1}(x,T),\ldots,p_{n}(x,T)$ are polynomials (in $T^{\frac{1}{d}}$,
for some positive integer $d$) of $\mathcal{S}(A_{0})\left[T^{\frac{1}{d}}\right]$,
with analytic coefficients in $\mathcal{S}\left(A_{0}\right)$, and
$p_{j}\left(x,0\right)=0$ for all $j\in\{1,\ldots,n\}$ and all $x\in A_{0}$.
We furthermore assume that for $j\not=j'$ in $\{1,\ldots,n\}$, $\sigma_{j}t+p_{j}(x,t)\not=\sigma_{j'}t+p_{j'}(x,t)$
(as functions).

\begin{lem}
\label{lem:Param:OscPrep} In above notation, we may express $f$
on $A$ as
\[
f\left(x,t\right)=F\left(x,\rho\left(x,t\right)\right)
\]
where $\rho=\left(\phi_{1},\cdots,\phi_{\ell},\rho_{1},\ldots,\rho_{p}\right)$
for some $\ell,p\in\mathbb{N}$, and where for each $i\in\{1,\ldots,\ell\}$
and for each $j\in\{1,\ldots,p\}$,
\[
\phi_{i}(x,t)=g_{i}\left(x\right)e^{\frac{\delta_{i}}{d}t},\ \ \rho_{j}\left(x,t\right)=\sigma_{j}t
\]
for some analytic functions $g_{i}$ in $\mathcal{S}\left(A_{0}\right)$,
$\delta_{i}\in\mathbb{N}$, $\sigma_{j}\in\mathbb{R}$, and where
$F\left(x,z_{1},\ldots,z_{\ell+p}\right)$ is a Laurent polynomial
in the variables $\mathrm{e}^{2\pi\mathrm{i}z_{1}},\ldots,\mathrm{e}^{2\pi\mathrm{i}z_{\ell+p}}$
with analytic coefficients in $\mathcal{C}^{\mathbb{C},\mathcal{F}}\left(A_{0}\right)$.
If $n=1$ and if $\sigma_{1}=0$, $p_{1}=0$, then $\ell+p=0$ and
$F\left(x\right)=f_{1}\left(x\right)$. Otherwise we have $\ell+p>0$,
and
\begin{enumerate}
\item \label{F par non nul} there exists a set $A_{0}'\subseteq A_{0}$
such that $\mathrm{vol}_{m}\left(A_{0}\setminus A_{0}'\right)=0$
and for every $x\in A_{0}'$, $z\mapsto F\left(x,z\right)$ is nonconstant,
\item \label{Psi CUD} for every open set $\Omega\subseteq A_{0}$ and every
real number $\lambda<\mathrm{vol}_{m}\left(\Omega\right)$, there
exists a real number $T_{0}$ and a compact set $K\subseteq\Omega\cap A_{0}'$
such that $K\times[T_{0},+\infty)\subseteq A$, $\lambda\le\mathrm{vol}_{m}\left(K\right)\le\mathrm{vol}_{m}(\Omega)$,
and $\rho\restriction{K\times[T_{0},+\infty)}$ is c.u.d. mod 1 on
$K$.
\end{enumerate}
\end{lem}
\begin{proof}
The case $n=1$, $\sigma_{1}=0$ and $p_{1}=0$ being trivial, we
may assume that $\sigma_{1}\not=0$ or $p_{1}\not=0$. 
For each $j\in\{1,\ldots,n\}$, we write
\[
\sigma_{j}t+p_{j}\left(x,\mathrm{e}^{t}\right)=\sigma_{j}t+\sum_{k=1}^{D}h_{j,k}\left(x\right)\mathrm{e}^{\frac{k}{d}t}
\]
with $D\in\mathbb{N}$ and $h_{j,k}\in\mathcal{S}\left(A_{0}\right)$.
For each $k\in\{1,\ldots,D\}$, fix $I_{k}\subseteq\{1,\ldots,n\}$
such that $\left(h_{i,k}\right)_{i\in I_{k}}$ is a basis of the $\mathbb{Q}$-vector
space generated by the family $\left(h_{j,k}\right)_{j\in\{1,\ldots,n\}}$
(as functions of $x$), and fix $Q\subseteq\{1,\ldots,n\}$ such that
$\left(\sigma_{q}\right)_{q\in Q}$ is a basis of the $\mathbb{Q}$-vector
space generated by the family $\left(\sigma_{j}\right)_{j\in\{1,\ldots,n\}}$.
We then set
\[
I=\{\left(i,k\right):k\in\{1,\ldots,D\},i\in I_{k}\}\subseteq\{1,\ldots,D\}\times\{1,\ldots,n\}.
\]
We fix a positive integer $\eta$ such that for each $\left(j,k\right)\in\{1,\ldots,n\}\times\{1,\ldots,D\}$,
\[
h_{j,k}=\sum_{i\in I_{k}}\frac{\alpha_{j;i,k}}{\eta}h_{i,k},\ \sigma_{j}=\sum_{q\in Q}\frac{\beta_{j;q}}{\eta}\sigma_{q}
\]
for unique tuples $\left(\alpha_{j;i,k}\right)_{i\in I_{k}}$ and
$(\beta_{j;q})_{q\in L}$ of elements of $\mathbb{Z}$. With this
notation we have
\begin{align*}
f\left(x,t\right) & =\sum_{j=1}^{n}f_{j}\left(x\right)\mathrm{e}^{\mathrm{i}\sigma_{j}t+\mathrm{i}\sum_{k=1}^{d}h_{j,k}\left(x\right)\mathrm{e}^{\frac{k}{d}t}}\\
 & =\sum_{j=1}^{n}f_{j}\left(x\right)\mathrm{e}^{\mathrm{i}\sum_{q\in Q}\frac{\beta_{j;q}}{\eta}\sigma_{q}t+\mathrm{i}\sum_{k=1}^{d}\sum_{i\in I_{k}}\frac{\alpha_{j;i,k}}{\eta}h_{i,k}\left(x\right)\mathrm{e}^{\frac{k}{d}t}}\\
 & =\sum_{j=1}^{n}f_{j}\left(x\right)\prod_{q\in Q}\left(\mathrm{e}^{2\pi\mathrm{i}\rho_{q}(t)}\right)^{\beta_{j;q}}\prod_{\left(i,k\right)\in I}\left(\mathrm{e}^{2\pi\mathrm{i}\phi_{i,k}\left(x,t\right)}\right)^{\alpha_{j;i,k}}=F\left(x,\left(\phi_{i,k}\left(x,t\right)\right)_{\left(i,k\right)\in I},(\rho_{q})_{q\in Q}\right)
\end{align*}
where for each $\left(i,k\right)\in I$, $\phi_{i,k}\left(x,t\right)=\frac{h_{i,k}\left(x\right)}{2\pi\eta}\mathrm{e}^{\frac{k}{d}t}$,
for each $q\in Q$, $\rho_{q}\left(t\right)=\frac{\sigma_{q}}{2\pi\eta}t$,
and
\[
F(x,(z_{i,k})_{(i,k)\in I},(z_{q})_{q\in Q})=\sum_{j=1}^{n}f_{j}\left(x\right)\prod_{q\in Q}\left(\mathrm{e}^{2\pi\mathrm{i}z_{q}}\right)^{\beta_{j;q}}\prod_{\left(i,k\right)\in I}\left(\mathrm{e}^{2\pi\mathrm{i}z_{i,k}}\right)^{\alpha_{j;i,k}}.
\]

For each $j\in\{1,\ldots n\}$, $f_{j}$ is a nonzero analytic function
on the connected and open set $A_{0}$, so the set
\[
U:=\{x\in A_{0}:\text{\ensuremath{f_{j}(x)\neq0} for all \ensuremath{j\in\{1,\ldots,n\}}}\}
\]
satisfies $\mathrm{vol}_{m}\left(A_{0}\setminus U\right)=0$. Denote
by $\mathcal{F}$ the Laurent polynomial associated to $F$
\[
\mathcal{F}(x,(Z_{i,k})_{(i,k)\in I},(Z_{q})_{q\in Q})=\sum_{j=1}^{n}f_{j}\left(x\right)\prod_{q\in Q}Z_{q}^{\beta_{j;q}}\prod_{\left(i,k\right)\in I}Z_{i,k}^{\alpha_{j;i,k}}.
\]
Note that $F(x,z)=F(x,(z_{i,k})_{(i,k)\in I},(z_{q})_{q\in Q})=\mathcal{F}(x,(\mathrm{e}^{2\pi\mathrm{i}z_{i,k}})_{(i,k)\in I},(\mathrm{e}^{2\pi\mathrm{i}z_{q}})_{q\in Q})$.

Since we assumed $\sigma_{1}\not=0$ or $p_{1}\not=0$, we can always
suppose $\sigma_{1}\in(\sigma_{q})_{q\in Q}$ or, for some $k$, $h_{1,k}\in(h_{i,k})_{i\in I_{k}}$,
respectively. Thus $\mathcal{F}$ certainly contains a term of the
form $f_{1}(x)Z_{1}$ or $f_{1}(x)Z_{1,k}$. Moreover, since for $j\not=j'$
in $\{1,\ldots,n\}$, $\sigma_{j}t+p_{j}(x,t)\not=\sigma_{j'}t+p_{j'}(x,t)$
(as functions), the monomial terms in the above expression of $\mathcal{F}$
cannot cancel out. It follows that for each $x\in U$, $\mathcal{F}$
is not constant as a Laurent polynomial, and in particular, for each
$x\in U$, not constant on the real torus $(S^{1})^{\vert Q\vert+\vert I\vert}$.
As a consequence, for each $x\in U$, the trigonometric polynomial
$z\mapsto F\left(x,z\right)$ is not constant.

Observe that since $\left(h_{i,k}\right)_{i\in I_{k}}$ is independent
over $\mathbb{Q}$ (as functions of $x$), for each $k\in\{1,\ldots,D\}$
and nonzero tuple $c=\left(c_{i}\right)\in\mathbb{Z}^{\vert I_{k}\vert}$,
$\sum_{i\in I_{k}}c_{i}h_{i,k}$ is a nonzero analytic function on
$A_{0}$, so the set $\left\{ x\in U:\sum_{i\in I_{k}}c_{i}h_{i,k}\left(x\right)=0\right\} $
cannot have positive measure, and the set
\[
A_{0}':=U\setminus\left(\bigcup_{k=1}^{D}\bigcup_{c\in\mathbb{Z}^{\vert I_{k}\vert}\setminus\{0\}}\left\{ x\in U:\sum_{i\in I_{k}}c_{i}h_{i,k}\left(x\right)=0\right\} \right)
\]
satisfies $\mathrm{vol}_{m}\left(A_{0}\setminus A_{0}'\right)=0$
as well. This gives \eqref{F par non nul}, for this set $A_{0}'\subset U$.

The set $A_{0}'$ is defined such that for each $k\in\{1,\ldots,D\}$,
for each $x\in A_{0}'$, the family of numbers $\left(h_{i,k}\left(x\right)\right)_{\left(i,k\right)\in I}$
is linearly independent over $\mathbb{Q}$. By Remark \ref{rem:independence through J_k},
for each $x\in A_{0}'$ the family of functions $\left(t\mapsto\phi_{i,k}\left(x,t\right)\right)_{\left(i,k\right)\in I}$
is also linearly independent over $\mathbb{Q}$. On the other hand
the family of functions $\left(t\mapsto\rho_{q}\left(t\right)\right)_{q\in Q}$
is linearly independent over $\mathbb{Q}$, since so is the family
of real numbers $(\sigma_{q})_{q\in Q}$. In particular, for each
$x\in A_{0}'$, the family of functions $t\mapsto\rho(x,t)=\left(\left(\phi_{i,k}\left(x,t\right)\right)_{\left(i,k\right)\in I},(\rho_{q}(t))_{q\in Q}\right)$
is linearly independent over $\mathbb{Q}$.

Given an open set $\Omega\subseteq A_{0}$ and a positive real number
$\lambda$ with $\lambda<\mathrm{vol}_{m}\left(\Omega\right)=\mathrm{vol}_{m}\left(\Omega\cap A_{0}'\right)$,
the inner regularity of the Lebesgue measure shows that we may fix
a compact set $K\subseteq\Omega\cap A_{0}'$ with $\mathrm{vol}_{m}\left(K\right)\geq\lambda$.
Since $K$ is compact and $a\left(x\right)$ is continuous, we may
fix $T_{0}$ sufficiently large so that $K\times[T_{0},+\infty)\subseteq A$.
Proposition \ref{prop:WeylToCUD} then shows that the restriction
of $\rho$ to $K\times[T_{0},+\infty)$ is c.u.d. mod $1$ on $K$,
which completes the proof of \eqref{Psi CUD}.
\end{proof}
Recall that
\[
f\left(x,y\right)=\sum_{j=1}^{n}f_{j}\left(x\right)y^{\mathrm{i}\sigma_{j}}\mathrm{e}^{\mathrm{i}p_{j}(x,y)},
\]
where $\sigma_{1},\ldots,\sigma_{n}$ are real numbers, $\left(f_{1},\ldots,f_{n}\right)$
is a family of (nonzero) analytic functions in $\mathcal{C}^{\mathbb{C},\mathcal{F}}\left(A_{0}\right)$,
$p_{1}(x,T),\ldots,p_{n}(x,T)$ are polynomials (in $T^{\frac{1}{d}}$,
for some positive integer $d$) of $\mathcal{S}(A_{0})\left[T^{\frac{1}{d}}\right]$,
with analytic coefficients in $\mathcal{S}\left(A_{0}\right)$, and
$p_{j}\left(x,0\right)=0$ for all $j\in\{1,\ldots,n\}$ and $x\in A_{0}$.
Furthermore we assume that for $j\not=j'$ in $\{1,\ldots,n\}$, $\sigma_{j}+p_{j}(x,y)\not=\sigma_{j'}+p_{j'}(x,y)$
(as functions).
\begin{lem}
\label{lem:Param:Dovetail}In the notation above, there exist $\varepsilon>0$,
$\Delta>0$, a strictly increasing sequence $\left(y_{j}\right)_{j\in\mathbb{N}}$
in $\mathbb{R}$ diverging to $+\infty$, a compact set $K\subset A_{0}$,
and a sequence $\left(X_{j}\right)_{j\in\mathbb{N}}$ of Lebesgue
measurable subsets of $K$, with, for all $j\in\mathbb{N}$, $\mathrm{vol}_{m}\left(X_{j}\right)\ge\Delta$,
$X_{2j+1}\subseteq X_{2j}$, and such that , for all $x_{0}\in X_{2j}$,
$x_{1}\in X_{2j+1}$,
\[
\vert f\left(x_{0},y_{2j}\right)\vert\ge\varepsilon\ \mathrm{and}\ \vert f\left(x_{0},y_{2j}\right)-f\left(x_{1},y_{2j+1}\right)\vert\ge\varepsilon.
\]
\end{lem}
\begin{proof}
Let $\tilde{f}(x,t):=f(x,\mathrm{e}^{t})$ for any $(x,t)$ such that
$(x,\mathrm{e}^{t})\in A$. Then we can apply Lemma \ref{lem:Param:OscPrep}
to $\tilde{f}$, so that the hypothesis of \cite[Lemma 8.10]{ccmrs:integration-oscillatory}
is satisfied by $\tilde{f}$. It immediately follows that the conclusions
of our lemma are satisfied by $\tilde{f}$, for a sequence of real
numbers $(t_{j})_{j\in\mathbb{N}}$ diverging to $+\infty$. It now
suffices to set $y_{j}=\mathrm{e}^{t_{j}}$ to conclude the proof
of the lemma.
\end{proof}
\begin{defn}
\label{def:Cauchy} Let $X\subseteq\mathbb{R}^{m}$ and $f\colon X\times\mathbb{R}\to\mathbb{C}$
be Lebesgue measurable, and $p\in[1,+\infty]$. For each $y\in\mathbb{R}$,
define $f_{y}:X\to\mathbb{C}$ by $f_{y}\left(x\right)=f\left(x,y\right)$
for all $x\in X$. We say that the family of functions $\left(f_{y}\right)_{y\in\mathbb{R}}$
is {\emph{Cauchy in $L^{p}\left(X\right)$ as $y\to+\infty$}} if
for each $y\in\mathbb{R}$, $f_{y}\in L^{p}\left(X\right)$ and for
all $\varepsilon>0$ there exists $y_{0}\in\mathbb{R}$ such that
\[
\|f_{y}-f_{y'}\|_{p}<\varepsilon\quad\text{for all \ensuremath{y,y'\geq y_{0}}.}
\]
\end{defn}
\begin{thm}
\label{prop:complete} Let $p\in[1,+\infty]$ and $f\in\mathcal{C}^{\mathbb{C},\mathcal{F}}\left(X\times\mathbb{R}\right)$,
for some subanalytic set $X\subseteq\mathbb{R}^{m}$, and suppose
that $\left(f_{y}\right)_{y\in\mathbb{R}}$ is Cauchy in $L^{p}\left(X\right)$
as $y\to+\infty$. Then there exist $g\in\mathcal{C}^{\mathbb{C},\mathcal{F}}\cap L^{p}\left(X\right)$
and a subanalytic set $X_{0}\subseteq X$ such that $\mathrm{vol}_{m}\left(X\setminus X_{0}\right)=0$,
\[
\lim_{y\to+\infty}\|f_{y}-g\|_{p}=0,
\]
and
\[
\lim_{y\to+\infty}f\left(x,y\right)=g\left(x\right)\quad\text{for all \ensuremath{x\in X_{0}}.}
\]
\end{thm}

\begin{proof}
Writing $f$ as a sum of generators as in Theorem \ref{prop:limits},
we proceed as in the proof of \cite[Proposition 8.2]{ccmrs:integration-oscillatory},
using Lemma \ref{lem:Param:Dovetail} instead of \cite[Lemma 8.10]{ccmrs:integration-oscillatory}.
\end{proof}
As a direct consequence of Proposition \ref{prop:complete} (see for
instance the proof of \cite[Theorem 8.3]{ccmrs:integration-oscillatory})
we obtain the following result.
\begin{thm}
\label{cor:Plancherel} Let $\widetilde{\hbox{\calli F}\ }$ be the
Fourier-Plancherel extension of the Fourier transform to $L^{2}\left(\mathbb{R}^{n}\right)$.
Then, the image of $\mathcal{C}^{\mathbb{C},\mathcal{F}}(\mathbb{R}^{n})\cap L^{2}\left(\mathbb{R}^{n}\right)$
under $\widetilde{\hbox{\calli F}\ }$ is $\mathcal{C}^{\mathbb{C},\mathcal{F}}(\mathbb{R}^{n})\cap L^{2}\left(\mathbb{R}^{n}\right)$. transform
\end{thm}
Stability under parametric Plancherel-Fourier transforms is formulated and shown similarly. 
\bibliographystyle{alpha}
\newcommand{\etalchar}[1]{$^{#1}$}
\def\cprime{$'$}

\end{document}